\documentclass[a4paper,11pt]{amsart}
\usepackage{amsfonts,amssymb,amsmath,amsthm,abstract,color}
\usepackage{bbm}
\usepackage[mathscr]{euscript}
\usepackage[ps,all,arc,rotate]{xy}
\usepackage[lmargin=1in,rmargin=1in,tmargin=1in,bmargin=1in]{geometry}
\usepackage{fancyhdr}
\usepackage{pb-diagram}
\usepackage[shortlabels]{enumitem}
\usepackage{hyperref}

\usepackage{stmaryrd} % for GIT quotient double slashes

\newtheorem{prop}{Proposition}[section]
\newtheorem{lemma}[prop]{Lemma}
\newtheorem{thm}[prop]{Theorem}
\newtheorem{cor}[prop]{Corollary}

\theoremstyle{definition}
\newtheorem{defn}[prop]{Definition}
\newtheorem{constr}[prop]{Construction}
\newtheorem{notation}[prop]{Notation}

\newtheorem{rmk}[prop]{Remark}
\newtheorem{ex}[prop]{Example}

\DeclareMathOperator{\Sp}{sp}

\DeclareMathOperator{\gr}{gr}
\DeclareMathOperator{\coh}{coh}
\DeclareMathOperator{\res}{res}
 
\DeclareMathOperator{\sm}{sm}
\DeclareMathOperator{\supp}{supp}

\DeclareMathOperator{\FDiv}{FDiv}

\DeclareMathOperator{\ord}{ord}
\DeclareMathOperator{\orb}{orb}
\DeclareMathOperator{\Gal}{Gal}
\DeclareMathOperator{\tr}{tr} 
\DeclareMathOperator{\rk}{rk}        
\DeclareMathOperator{\spec}{Spec} \DeclareMathOperator{\Sym}{Sym}

\DeclareMathOperator{\id}{id}

\newcommand{\ra}{\rightarrow}      

\newcommand{\lra}{\longrightarrow}
\newcommand{\rap}{\stackrel{+}{\ra}}

\DeclareMathOperator*{\hocolim}{hocolim}
\DeclareMathOperator*{\colim}{colim}
\DeclareMathOperator*{\Lim}{lim}
 
\DeclareMathOperator{\Div}{Div} 
\DeclareMathOperator{\Jac}{Jac} 
\DeclareMathOperator{\Hom}{Hom}
\DeclareMathOperator{\Spec}{Spec}
\DeclareMathOperator{\Pic}{Pic}

\DeclareMathOperator{\codim}{codim}

\newcommand{\et}{\mathrm{\acute{e}t}}

\def\cA{\mathcal A}\def\cB{\mathcal B}\def\cC{\mathcal C}\def\cD{\mathcal D}
\def\cE{\mathcal E}\def\cF{\mathcal F}\def\cG{\mathcal G}
\def\cI{\mathcal I}\def\cL{\mathcal L}
\def\cM{\mathcal M}\def\cN{\mathcal N}\def\cO{\mathcal O}\def\cP{\mathcal P}
\def\cT{\mathcal T}
\def\cV{\mathcal V}\def\cW{\mathcal W}
\def\cY{\mathcal Y}\def\cZ{\mathcal Z}

\def\AA{\mathbb A}\def\CC{\mathbb C}
\def\GG{\mathbb G}
\def\LL{\mathbb L}
\def\NN{\mathbb N}\def\PP{\mathbb P}
\def\QQ{\mathbb Q}\def\RR{\mathbb R}

\def\ZZ{\mathbb Z}

\def\fd{\mathfrak d}
\def\fg{\mathfrak g}\def\fh{\mathfrak h}
\def\fl{\mathfrak l}

\def\fs{\mathfrak s}

\def\fM{\mathfrak M}

\def\fX{\mathfrak X}

 \def\GL{\mathrm{GL}} \def\SL{\mathrm{SL}}
\def\PGL{\mathrm{PGL}}
\def\DM{\mathrm{DM}}  \def\DA{\mathrm{DA}}
\def\SH{\mathrm{SH}}
 
\def\CH{\mathrm{CH}}
\def\ab{\mathrm{ab}}
\def\an{\mathrm{an}}
\def\num{\mathrm{num}}
\def\rat{\mathrm{rat}}
\def\fd{\mathrm{fd}}
\def\ev{\mathrm{ev}}

\makeatletter
\def\@tocline#1#2#3#4#5#6#7{\relax
  \ifnum #1>\c@tocdepth % then omit
  \else
    \par \addpenalty\@secpenalty\addvspace{#2}%
    \begingroup \hyphenpenalty\@M
    \@ifempty{#4}{%
      \@tempdima\csname r@tocindent\number#1\endcsname\relax
    }{%
      \@tempdima#4\relax
    }%
    \parindent\z@ \leftskip#3\relax \advance\leftskip\@tempdima\relax
    \rightskip\@pnumwidth plus4em \parfillskip-\@pnumwidth
    #5\leavevmode\hskip-\@tempdima
      \ifcase #1
       \or\or \hskip 1em \or \hskip 2em \else \hskip 3em \fi%
      #6\nobreak\relax
    \hfill\hbox to\@pnumwidth{\@tocpagenum{#7}}\par% <---- \dotfill -> \hfill
    \nobreak
    \endgroup
  \fi}
\makeatother

\makeatletter
\newsavebox{\@brx}
\newcommand{\llangle}[1][]{\savebox{\@brx}{\(\m@th{#1\langle}\)}%
  \mathopen{\copy\@brx\kern-0.5\wd\@brx\usebox{\@brx}}}
\newcommand{\rrangle}[1][]{\savebox{\@brx}{\(\m@th{#1\rangle}\)}%
  \mathclose{\copy\@brx\kern-0.5\wd\@brx\usebox{\@brx}}}
\makeatother

\DeclareMathOperator{\ch}{\mathcal{C}h}

\DeclareMathOperator{\m}{\cM}
\newcommand{\Bun}{\mathfrak{Bun}}
\newcommand{\Coh}{\mathfrak{Coh}}
\newcommand{\Ch}{\mathfrak{Ch}_{\underline{m},\underline{e}}}
\newcommand{\Chinj}{\Ch^{\mathrm{inj}}}
\newcommand{\ChinjL}{\mathfrak{Ch}_{\underline{m},\underline{e},L}^{\mathrm{inj}}}
\newcommand{\fChinj}{\widetilde{\mathfrak{Ch}}_{\underline{m},\underline{e}}^{\mathrm{inj}}}
\newcommand{\fChinjL}{\widetilde{\mathfrak{Ch}}_{\underline{m},\underline{e},L}^{\mathrm{inj}}}
\newcommand{\Chtau}{\Ch^{\alpha,\tau}}

\newcommand{\Chss}{\Ch^{\alpha,ss}}
\newcommand{\ChDss}{\Ch^{\alpha_D,ss}}

\newcommand{\mChDss}{\ch^{\alpha_D,ss}_{\underline{m},\underline{e}}}

\newcommand{\GLHiggs}{\m}
\newcommand{\GLHit}{\cA}
\newcommand{\SLHiggs}{\m_{L}}
\newcommand{\SLHit}{\cA_{L}}
\newcommand{\piHiggs}{\m_{\pi}}
\newcommand{\piHit}{\cA_{\pi}}
\newcommand{\gammaHiggs}{\m_{\gamma}}
\newcommand{\gammaHit}{\cA_{\gamma}}
\newcommand{\PGLHiggs}{\overline{\m}}
\newcommand{\PGLHit}{\overline{\cA}}
\newcommand{\hGamma}{\widehat{\Gamma}}
\newcommand{\mot}{\mathrm{mot}}

\newcommand{\Log}{\mathcal{L}\mathrm{og}^\vee}
\newcommand{\un}{\mathrm{uni}}
\newcommand{\one}{\mathbbm{1}}

\title[Motivic mirror symmetry and $\chi$-independence]{Motivic mirror symmetry and $\chi$-independence for Higgs bundles in arbitrary characteristic}

\author{Victoria Hoskins and Simon Pepin Lehalleur}

\begin{document}

\maketitle

\begin{abstract}
We prove that the (twisted orbifold) motives of the moduli spaces of $\SL_n$ and $\PGL_n$-Higgs bundles of coprime rank and degree on a smooth projective curve over an algebraically closed field in which the rank is invertible are isomorphic in Voevodsky's triangulated category of motives. The equality of twisted orbifold Hodge numbers of these moduli spaces was conjectured by Hausel and Thaddeus and recently proven by Groechenig, Ziegler and Wyss via $p$-adic integration and then by Maulik and Shen using the decomposition theorem, an analysis of the supports of $D$-twisted Hitchin fibrations and vanishing cycles. Our proof in characteristic zero combines the geometric ideas of Maulik and Shen with the conservativity of the Betti realisation on abelian motives; to apply the latter, we prove that the relevant motives are abelian. In particular, we prove that the motive of the $\SL_n$-Higgs moduli space is abelian, building on our previous work in the $\GL_{n}$-case. We then use motivic nearby cycles to deduce the result in positive characteristic from that in characteristic zero. Using the same ideas, we prove motivic $\chi$-independence for $\GL_n$-Higgs bundles.
\end{abstract}

\vspace{-0.5cm}

\setcounter{tocdepth}{1}
\tableofcontents

\section{Introduction}

Let $C$ be a smooth projective geometrically connected genus $g$ curve over a field $k$. The moduli space $\GLHiggs:=\cM_{n,d}(C)$ of stable Higgs bundles $(E,\Phi: E \ra E \otimes \omega_C)$ of coprime rank $n$ and degree $d$ is a smooth quasi-projective variety of dimension $2(n^2(g-1) +1)$. Taking the characteristic polynomial of the Higgs field $\Phi$ defines the Hitchin fibration, a morphism $h: \GLHiggs \ra \GLHit$ to an affine space $\GLHit$ called the Hitchin base \cite{Hitchin}. The variety $\GLHiggs$ admits an algebraic symplectic form, and the morphism $h$ is a proper Lagrangian fibration with respect to the symplectic structure; the generic fibres of $h$ are torsors under the Jacobians of the corresponding spectral curves. Over the complex numbers, $\GLHiggs$ is a non-compact hyperk\"{a}hler manifold and is diffeomorphic to both the moduli space of holomorphic flat connections and the character variety of topological local systems on $C$ by the non-abelian Hodge correspondence \cite{simpson_NAHT} (or more precisely, to variants of those moduli spaces taking into account the non-zero degree $d$).

More generally, for a reductive group $G/k$, there is a notion of $G$-Higgs bundles and corresponding moduli spaces $\cM_G$ of semistable $G$-Higgs bundles, which in the case of $G = \GL_n$ coincides with $\GLHiggs$. These $G$-Higgs moduli spaces are algebraic symplectic and also come with (proper Lagrangian) Hitchin fibrations; when $k=\CC$, the non-abelian Hodge correspondence extends to $G$-Higgs bundles. Given two Langlands dual reductive groups $G$ and ${^L}G$, the corresponding Hitchin fibrations have (almost) canonically isomorphic bases and it is expected that $\cM_{G}$ and $\cM_{^{L}G}$ are closely related via these Hitchin fibrations. This relationship can be understood heuristically as a ``limit'' of the geometric Langlands correspondence and also as a form of mirror symmetry. The first concrete statement in that direction is a relationship between the generic fibres of the two fibrations which has been established by Hausel and Thaddeus \cite{HT} in the case of $G = \SL_n$ and by Donagi-Pantev \cite{DP} and Derryberry \cite{Derryberry} in general: these generic fibres are torsors under dual abelian varieties, and certain natural gerbes match up.

\subsection{Mirror symmetry for $\SL$ and $\PGL$-Higgs bundles}

In this paper, we study $G$-Higgs bundles for the Langlands dual groups $\SL_n$ and $\PGL_n$. 
We fix a degree $d$ coprime to $n$ and choose a degree $d$ line bundle $L$ and let $\SLHiggs:= \cM_{n,L}(C)$ denote the moduli space of stable Higgs bundles of rank $n$ with determinant $L$ and trace-free Higgs field; we refer to these as $\SL_n$-Higgs bundles\footnote{Strictly speaking, these Higgs bundles should not quite be called $\SL$-Higgs bundles, as specialising the general definition of $G$-Higgs bundle to $G=\SL_n$ forces the determinant to be trivial and $d=0$ (but leads to singular moduli spaces). This variant of $\SL_n$-Higgs bundles turns out to be the right thing to consider in the context of the Hausel-Thaddeus conjecture.}. The $\SL_n$-Higgs moduli space $\SLHiggs$ is a smooth closed subvariety of $\GLHiggs=\cM_{n,d}(C)$, which can be realised as a fibre of the map $(\det,\tr) : \cM_{n,d}(C) \ra \cM_{1,d}(C)$. The Jacobian $\Jac(C)$ acts on $\GLHiggs$ by tensorisation and the action of the $n$-torsion subgroup $\Gamma: = \Jac(C)[n]$ restricts to $\SLHiggs$. The corresponding $\PGL_n$-Higgs moduli space $\PGLHiggs = \PGLHiggs_{n,d}$ is singular, but smooth as an orbifold: we can identify it with the following quotients
\[  \PGLHiggs \simeq [\SLHiggs/\Gamma] \simeq [\GLHiggs/T^*\Jac(C)]. \]
Consequently $\PGLHiggs$ can be viewed as a smooth Deligne--Mumford stack, which inherits a $\mu_{n}$-gerbe $\delta_{L}$ from the $\mu_{n}$-gerbe on $\SLHiggs$ coming from the corresponding moduli stack \cite[\S 3]{HT} .

The corresponding Hitchin bases for these $\SL$-Higgs and $\PGL$-Higgs moduli spaces are canonically isomorphic $\SLHit\simeq \PGLHit$ and the generic fibres of the $\SL$-Hitchin (resp. $\PGL$-Hitchin) fibration are torsors under Prym varieties (resp. $\Gamma$-quotients of Prym varieties) \cite{Hitchin}. 

Over $k = \CC$, Hausel and Thaddeus \cite{HT} predicted a``topological mirror symmetry'' relation between the $E$-polynomial of the $\SL_n$-Higgs moduli space and the stringy $E$-polynomial of the $\PGL_n$-Higgs moduli space, which encodes the twisted orbifold Hodge numbers with respect to the gerbe $\delta_{L}$; as we will see below this stringy $E$-polynomial has a more concrete description in terms of the $\Gamma$-action on $\SLHiggs$. The conjecture of Hausel and Thaddeus was proved by Groechenig, Wyss and Ziegler \cite{GWZ} using $p$-adic integration. Recently, Maulik and Shen \cite{MS} upgraded the agreement of (stringy) $E$-polynomials to an agreement of (orbifold) Hodge structures using perverse sheaves, the decomposition theorem, support theorems for Hitchin fibrations and vanishing cycles (see $\S$\ref{overview MS} for a summary). In both \cite{GWZ} and \cite{MS}, a crucial ingredient is the non-trivial $\Gamma$-action on the cohomology of $\SLHiggs$ and its isotypical decomposition.

In this paper, we build on the ideas and techniques of \cite{MS} to prove a motivic version of mirror symmetry (Theorem \ref{main_thm}) relating the (orbifold) Voevodsky motives of the $\SL_n$-Higgs and the $\PGL_n$-Higgs moduli spaces. The Voevodsky motive $M(X)$ of a smooth $k$-variety $X$ with coefficient in a $\QQ$-algebra $\Lambda$ is an object of the triangulated category $\DM(k,\Lambda)$ of mixed motives over $k$ with coefficients in $\Lambda$. It is a very fine cohomological invariant of $X$, which contains information both about the cohomology of $X$ with its mixed Hodge structure (when $k\subset \CC$) and also the rational Chow groups and rational algebraic K-theory groups of $X$. If $X$ is also projective, $M(X)$ contains the same information as the perhaps more familiar Chow motive of $X$. However, Voevodsky motives are much more flexible than Chow motives and admit a fully-fledged theory of ``motivic sheaves'' on schemes and stacks with a six operation formalism and vanishing cycles functors which are crucial to our results.

A virtual motivic version of topological mirror symmetry was already established by Loeser and Wyss \cite{LW} who prove an equality of (orbifold) virtual Chow motives in the Grothendieck ring of Chow motives using motivic integration and the ideas of \cite{GWZ}. Our result implies and can be thought of as a categorification of the main theorem in \cite{LW}; see Corollary \ref{corr results} \emph{\ref{cor virtual}}.

\subsection{Results and methods} First, we motivically lift topological mirror symmetry \cite{HT,GWZ,MS}.

\begin{thm}[Motivic mirror symmetry]\label{main_thm}
Let $C$ be a smooth projective connected genus $g \geq 2$ curve over an algebraically closed\footnote{This result also holds for non-algebraically closed fields after passing to a finite extension (see Remark \ref{rmk non-alg closed field}).} field $k$. For a rank $n$ (invertible in $k$), coprime degree $d$ and line bundle $L \in \Pic^d(C)$, there is an isomorphism in $\DM(k,\QQ(\zeta_n))$
\[ M(\SLHiggs) \simeq M_{\orb}(\PGLHiggs,\delta_{L})\]
between the motive of the $\SL_n$-Higgs moduli space $\SLHiggs = \cM_{n,L}$ and the orbifold motive of the $\PGL_n$-Higgs moduli space $\PGLHiggs =\PGLHiggs_{n,d}$ with respect to the gerbe $\delta_{L}$ (see $\S$\ref{sec orbifold motive}).
\end{thm}

% TODO: mention that those motives are all pure, and hence Chow, even though the moduli spaces are only smooth and not proper?

The theorem implies that the (orbifold) Chow groups of the $\SL$-Higgs and $\PGL$-Higgs moduli spaces are isomorphic (see Corollary \ref{corr results}, which discusses consequences on motivic cohomology and algebraic K-theory). We relate this to other results and conjectures in $\S$\ref{sec related results}. As the topological mirror symmetry is part of a more general mirror symmetry picture for hyperkähler orbifolds, one may ask if other mirror symmetric pairs have isomorphic (orbifold) motives.

We need $\QQ(\zeta_n)$-coefficients to decompose the motive of $\SLHiggs$ into isotypical components for the $\Gamma$-action. By definition, the twisted orbifold motive of $\PGLHiggs = [\SLHiggs/\Gamma]$ is the sum over $\gamma \in \Gamma$ of Tate twists of isotypical pieces of the motives of $\gamma$-fixed loci in $\SLHiggs$. The above isomorphism is constructed as a sum over each $\gamma \in \Gamma$ of isomorphisms $\nu_{\gamma}$ constructed from relative morphisms $\beta_{\gamma}$ over the Hitchin base $\SLHit$ (see Definitions \ref{def beta mot otherwise} and $\S$\ref{end proof} for details).

Over a field of characteristic zero, the relative morphisms $\beta_{\gamma}$ are built as motivic lifts of morphisms in the constructible derived category constructed by Maulik--Shen \cite{MS} (which we review in Section \ref{overview MS}). We use motivic correspondences and motivic vanishing cycles, and require an extension of some results of \cite[Chapter 3]{Ayoub_these_2} to Artin stacks and a computation of motivic vanishing cycles for homogeneous functions (see Appendix \ref{sec motivic van cycles}). 

We do not know if the morphisms $\beta_{\gamma}$ of relative motives in $\DM(\SLHit,\Lambda)$ are isomorphisms, as the method of \cite{MS}, based on perverse sheaves and the decomposition theorem, is not at all available in $\DM(-,\Lambda)$. We expect they are isomorphisms, as by construction, the Betti realisations of $\beta_{\gamma}$ and $\nu_\gamma$ are the cohomological isomorphisms of Maulik--Shen \cite{MS}, and the Betti realisation on the category $\DM_{c}(-,\Lambda)$ of constructible motives with coefficients in a $\QQ$-algebra $\Lambda$ is conjectured to be conservative, but this is a difficult open question in general \cite{Ayoub_Survey}. However a result of Wildeshaus \cite{Wildeshaus_Picard} ensures that this is true when restricting to the subcategory $\DM_{c}^{\ab}(k,\Lambda)$ of constructible \emph{abelian} motives. Hence, to finish the proof of Theorem \ref{main_thm} in characteristic zero, we show that both source and target of $\nu_{\gamma}$ are abelian by extending our previous work \cite{HPL_Higgs} from $\GL_n$-Higgs bundles to $\SL_n$-Higgs bundles (see Theorem \ref{main thm motives abelian} below).

In positive characteristic, the $\ell$-adic realisation is not known to be conservative on $\DM_{c}^{\ab}(k,\Lambda)$ (except to some extent over a finite field \cite{Ancona_cons}) and it is not clear if the corresponding cohomological results hold, as the support theorems of Chaudouard-Laumon \cite{CL} and de Cataldo \cite{dC_supportSL} used in \cite{MS} only apply in characteristic zero. Instead, to pass from characteristic zero to positive characteristic, we consider a family of curves over the spectrum of a mixed characteristic DVR and use motivic nearby cycles to specialise from the generic fibre to the special fibre. Our argument here is related to \cite{dCZ_NAHTcharp} in the $\ell$-adic setting, but we use the semi-projective $\GG_m$-action on moduli of Higgs bundles rather than a compactification.

This idea of lifting cohomological results in characteristic zero using conservativity on abelian motives and then passing to positive characteristic using motivic nearby cycles can be applied to other cohomological results. In particular, cohomological $\chi$-independence for Higgs bundles, which states that the cohomology of moduli spaces $\cM_{n,d}$ of Higgs bundles of rank $n$ and coprime degree $d$ is independent of $d$ \cite{mellit,GWZ,Yu_H,Kinjo-Koseki}, admits the following motivic refinement.

\begin{thm}[Motivic $\chi$-independence]\label{main_thm_chi}
Let $C$ be a smooth projective connected genus $g \geq 2$ curve over a field $k$ with $C(k)\neq\emptyset$. For a rank $n$ and coprime degrees $d$ and $d'$, there is an isomorphism in $\DM(k,\QQ)$
\[ M(\cM_{n,d}) \simeq M(\cM_{n,d'}).\]
\end{thm}

We expect that motivic $\chi$-independence for the  $\SL_{n}$-Higgs moduli space $\cM_{n,L}$ (when $n$ is invertible in $k$) should be obtained by similar techniques from the cohomological version in \cite{MS}.

The $\chi$-independence phenomenon for Higgs bundles is now understood as an instance of $\chi$-independence of BPS cohomology in the context of enumerative geometry of Calabi-Yau 3-folds (e.g.\ see \cite[$\S$0.3]{MS-chi-ind} and \cite{Kinjo-Koseki}). Once again, it would be very interesting to know if this more general $\chi$-independence has a motivic nature.

A key ingredient in both these proofs is the fact that the motives involved are abelian, which boils down to our previous work \cite{HPL_Higgs} for $\GL_n$ and the following new result for $\SL_n$.

\begin{thm}\label{main thm motives abelian}
Let $C$ be a smooth projective geometrically connected genus $g \geq 2$ curve over an arbitrary field $k$ with $C(k) \neq \emptyset$. For a rank $n$ and line bundle $L$ of coprime degree, the motive in $\DM(k,\QQ)$ of the moduli space $\SLHiggs:=\cM_{n,L}(C)$ of $\SL$-Higgs bundles is a direct factor of the motive of a product of étale covers of $C$. In particular, it is abelian.
\end{thm}

In \cite{HPL_Higgs}, we obtained a stronger result for the $\GL_n$-Higgs moduli space $\cM_{n,d}$ of coprime rank $r$ and degree $d$: its motive is generated by the motive of $C$. Our approach was based on the geometric ideas in \cite{GPHS}: via a motivic Bia{\l}ynicki--Birula decomposition for the scaling $\GG_m$-action on $\cM_{n,d}$ we relate to motives of moduli spaces of chains of vector bundle homomorphisms, and using variation of stability for chains and wall-crossing, we reduce to describing the motive of the stack of injective chain homomorphisms. We explicitly compute the latter using a full flag version of the stack of injective chain homomorphisms which by \cite{Heinloth_LaumonBDay} admits a small map to the stack of injective chain homomorphisms and is an iterated projective bundle over the product of a power of $C$ with a moduli stack of vector bundles, whose motive we computed in \cite{HPL_formula}.

The $\SL_n$-case is more involved: $M(\SLHiggs)$ is not generated by $M(C)$ in general by \cite[Proposition 5.7]{FHPL_rank3} and, although we can reduce to describing the motive of the stack of injective chain homomorphisms with fixed total determinant, we do not obtain an explicit formula for this motive. There is still a small map from a full flag version of the stack of injective chain homomorphisms with fixed total determinant, which is an iterated projective bundle over a stack of vector bundles with fixed determinant on a family of curves over a base scheme $B$ built from powers of $C$ and multiplication maps on $\Jac(C)$. We show $M(B)$ is abelian and extend our motivic formula \cite{HPL_formula} for the stack of vector bundles to bundles with fixed determinant and provide relative versions of these formulae (see Appendix \ref{sec appendix motive stacks of bundles fixed det}); this should be of independent interest.

Finally, we note that Theorems \ref{main_thm} and \ref{main_thm_chi} have corresponding versions for Higgs bundles twisted by a divisor $D$ (see Theorems \ref{main_thm_D} and \ref{main_thm_chi_D}), since analogously to Maulik--Shen \cite{MS} the result is first proved for a divisor $D$ of sufficiently large degree and then one passes to smaller $D$ (and eventually $D = K_C$) using vanishing cycles. Consequently, we need to prove the moduli spaces of $D$-twisted Higgs bundles for $\GL_n$ and $\SL_n$ are abelian (see Theorems \ref{thm motive GL D Higgs gen by C} and \ref{thm mot SL Higgs ab}).

\subsection{Related works and further directions}\label{sec related results}

Besides the already mentioned works \cite{HT,GWZ,LW,MS} which inspired this paper, there have been many other interesting results on topological mirror symmetry extending the original setup of Hausel-Thaddeus.

One can ask about similar statements for moduli of parabolic Higgs bundles (of fixed coprime rank and degree). A version of topological mirror symmetry in this case was established by Gothen and Oliveira \cite{Gothen-Oliveira} in ranks $2$ and $3$ and by Shiyu Shen \cite{shen-parabolic} in general.

If $n$ and $d$ are non-coprime and the corresponding moduli spaces are singular, one could ask if topological mirror symmetry holds for a suitable cohomology theory. One reasonable choice is intersection cohomology, and indeed a form of topological mirror symmetry for IH was conjectured and established in rank $2$ for $D=K_{C}$ by Mauri \cite{Mauri} and by Maulik and Shen \cite[Theorem 0.2]{MS-IH} in arbitrary rank and degree but only for $D$-twisted Higgs bundles with $\deg(D)>2g-2$. Maulik and Shen also raised the question of whether topological mirror symmetry holds for BPS cohomology, an invariant coming from mathematical physics which has been given a rigourous definition for Higgs bundles in \cite{Kinjo-Koseki} following ideas of Davison-Meinhardt and Toda. If $\deg(D)>2g-2$, BPS cohomology coincides with intersection cohomology, but they differ when $D=K_{C}$. As discussed in \cite[End of \S 3.6]{MS-IH}, topological mirror symmetry for BPS cohomology for $D=K_{C}$ should follow from results in \cite[Theorem 0.2]{MS-IH} and \cite{Kinjo-Masuda}.

Another incarnation of mirror symmetry in this context is an expected derived equivalence
\begin{equation}
  \label{der eq}
D^{b}_{\mathrm{coh}}(\SLHiggs)\simeq D^{b}_{\mathrm{coh}}(\PGLHiggs,\delta_{L})
\end{equation}
induced by a twisted Fourier-Mukai kernel relative to their common Hitchin base $\SLHit\simeq \PGLHit$ whose restriction to smooth Hitchin fibres should be the Mukai derived equivalence between dual abelian varieties (such a derived equivalence was proven for the stack of all $G$-Higgs bundles over an open subset of the Hitchin base in \cite[Corollary 5.5]{DP}). The Chern character of this Fourier-Mukai kernel should not induce an isomorphism of motives, in the same way that for an abelian variety $A$, the Chern character of the Poincaré line bundle inducing the Mukai equivalence $D^{b}_{\mathrm{coh}}(A)\simeq D^{b}_{\mathrm{coh}}(\widehat{A})$ does not induce an isomorphism $M(A)\simeq M(\widehat{A})$ of rational Chow motives, but rather the Fourier transform on rational Chow groups $\CH^{*}(A,\QQ)\simeq \CH^{*}(\widehat{A},\QQ)$, which does not preserve the cohomological grading. However in the case of abelian varieties, we know that any fixed isogeny $A\to \widehat{A}$ induces an isomorphism $M(A)\simeq M(\widehat{A})$.

All results mentioned so far concern \emph{rational} cohomology. One can also ask if topological mirror symmetry holds integrally. Groechenig and Shiyu Shen \cite{GS} observed that this cannot be the case for integral singular cohomology, because they showed that the integral cohomology on the $\SL_n$-side is torsion-free, while the orbifold cohomology on the $\PGL_n$-side has torsion. Nevertheless they show an integral analogue at the level of (twisted) topological complex K-theory: in the coprime case, there is an isomorphism of topological $K$-theory spectra
\[
\mathrm{KU}(\SLHiggs^{\an})\simeq \mathrm{KU}(\PGLHiggs^{\an},\delta^{\an}_{L}).
\]
Unlike the result of \cite{MS} and Theorem \ref{main_thm}, this isomorphism is expected to be directly induced by the conjectural derived equivalence \eqref{der eq}. The proof in \cite{GS} shares nevertheless some interesting similarities to \cite{MS} and to our arguments, which we would like to point out by sketching their strategy. They start with a correspondence over the elliptic locus of the Hitchin base (which in their case is a Fourier-Mukai kernel inspired by the one studied by Arinkin in the $\GL_{n}$-case). They then use the decomposition theorem, results of de Cataldo on supports for the $\SL_{n}$-Hitchin fibration \cite{dC_supportSL} and vanishing cycles to get a statement about rational topological K-theory over the full Hitchin base. To be able to argue relatively to the Hitchin base, they use and refine ideas of Blanc, Moulinos and Brown about the topological K-theory of (sheaves of) dg-categories. To complete the proof, they upgrade this to a statement about integral topological K-theory by showing that the topological K-theory on both sides is torsion-free. In the case of $\SLHiggs$, they reduce using the Atiyah-Hirzebruch spectral sequence to show that the singular cohomology of $\SLHiggs$ is torsion-free. For this they use the chain wall-crossing techniques of \cite{GPHS,GPH} similarly to our Section \ref{sec mot SL-Higgs} (see Remark \ref{rmk comparison GS}). We only became aware of this similarity in both proofs after our paper was completed.

We expect that Theorem \ref{main_thm} follows from a relative isomorphism of motivic sheaves over the Hitchin base, i.e.\ that the morphism $\beta^D_{\gamma}$ in Definition \ref{def beta mot otherwise} is an isomorphism. Its Betti realisation is an isomorphism by \cite{MS} and Lemma \ref{lemma betti realisation of beta mot is beta}, thus by conservativity on abelian motives and \cite[Proposition 3.24]{Ayoub_etale}, it suffices to show that all fibres of the relevant Hitchin fibrations have abelian motives.

\subsection*{Acknowledgements} We thank Michael Groechenig, Junliang Shen and Dimitri Wyss for helpful comments. The second author is supported by Radboud Universiteit Nijmegen and the NWO-TOP1 grant ``Geometry and arithmetic beyond Shimura varieties'' of Ben Moonen and Lenny Taelman.

\subsection*{Motivic set-up}\label{sec set up}
Let $S$ be a finite type scheme over a field $k$ and $\Lambda$ be a $\QQ$-algebra. We denote by $\DM(S,\Lambda)$ the triangulated category\footnote{In Appendix \ref{sec motivic van cycles}, where we also review how to extend $\DM$ and motivic vanishing cycles to Artin stacks; there it is essential for technical reasons to we consider the $\infty$-categorical enhancement of $\DM(S,\Lambda)$, but this is not necessary in the body of the paper.} of (Morel-Voevodsky) étale motivic sheaves over $S$ with coefficients in $\Lambda$. This category can be defined in several equivalent ways (even more so because we work with rational coefficients), see \cite[\S 16]{Cisinski-Deglise-book}. For concreteness we adopt the construction of \cite[\S 3]{Ayoub_etale} where this category is denoted $\DA^{\et}(S,\Lambda)$. When $S=\Spec(k)$ is a perfect field, we write $\DM(k,\Lambda):=\DM(\Spec(k),\Lambda)$, and $\DM(k,\Lambda)$ is equivalently to the original definition of $\DM$ by Voevodsky using presheaves with transfers, see again \cite[\S 16]{Cisinski-Deglise-book}.

A central feature of $\DM(-,\Lambda)$ for the purpose of this paper is that it admits a full ``six operation formalism'', as well as a formalism of nearby and vanishing cycles. Almost everything that we need about this is contained in \cite{Ayoub_these_1,Ayoub_these_2,Ayoub_etale}. In Appendix \ref{sec motivic van cycles}, we give a few complements about motivic sheaves and motivic vanishing cycles on algebraic stacks.

Let $S$ be a finite type over a field $k$ with structure morphism $\pi:S\to \Spec(k)$. In terms of the six operation formalism, $S$ admits both an homological motive $M(S):=\pi_{!}\pi^{!}\one$ and a cohomological motive $M_{\mathrm{coh}}(S):=\pi_{*}\pi^{*}\one$, and the Verdier duality functor exchanges the two. Most of the paper is devoted to computing with (relative) cohomological motives, but we sometimes dualise the result to obtain statements about homological ones (which are in a sense more natural in Voevodsky's theory).

We denote by $\DM_c(S,\Lambda)\subset \DM(S,\Lambda)$ the subcategory of constructible objects \cite[Definition 8.1]{Ayoub_etale} which in this context coincide with the compact objects in the triangulated sense \cite[Proposition 8.3]{Ayoub_etale}. In our context, the six operations preserve constructible motives \cite[Theoremes 8.10-12]{Ayoub_etale}, so that almost all motivic sheaves appearing in this paper are constructible. 

We denote by $\DM^{\ab}_{c}(k,\Lambda)\subset \DM_{c}(k,\Lambda)$ the subcategory of abelian motives, i.e. the thick tensor triangulated subcategory generated by the motives of all smooth projective curves over $k$ (or equivalently the thick triangulated subcategory generated by the motives of all abelian varieties over $k$ by \cite[Proposition 4.5]{AEH}, hence the name).

Suppose we have a complex embedding\footnote{Such an embedding may not exist for every $k$ for cardinality reasons, but an application of the Lefschetz principle immediately reduces the proof of the main theorem to the case where it does; see $\S$\ref{end proof} for details.} $\sigma:k\to \CC$. If $S$ is a finite type $k$-scheme, we denote by $S^{\an}:=(S\times_{k,\sigma}\CC)^{\an}$ the associated complex analytic space. By \cite{Ayoub_Betti}, there is an associated Betti realisation functor, which is an exact symmetric monoidal functor
\[
R_{B}:\DM(S,\Lambda)\to D(S^{\an},\Lambda)
\]
to the triangulated category of sheaves of $\Lambda$-modules on $S^{\an}$. This functor ``commutes'' with the six operations in the best possible sense (without restrictions for the left adjoint functors $f^*$ and $f_!$, and when restricted to constructible objects for the right adjoint functors) \cite[Theorem 3.19]{Ayoub_Betti}. It also commutes with nearby cycles functors when restricted to constructible objects \cite[Theorem 4.9]{Ayoub_Betti} (including in the context of algebraic stacks, see \ref{thm van cycles commute with Betti realisation}).\\

\noindent \textbf{Notation.} For $j\in\ZZ$ and $M\in\DM(S,\Lambda)$, we write $M\{j/2\}:= M(\lfloor j/2 \rfloor)[j]\in \DM(S,\Lambda)$. Note that the ``Tate twist'' $\{j/2\}$ is pure if and only if $j$ is even.

Let $M\in \DM(S,\Lambda)$ be a motive. We write $\langle M \rangle^\otimes$ (resp. $\llangle M\rrangle^\otimes$) for the smallest thick tensor subcategory (resp. smallest localising tensor subcategory) of $\DM(S,\Lambda)$ containing $M$.

\section{Background on Higgs moduli spaces}

In this section, we introduce $D$-twisted Higgs bundles and their $\SL$ and $\PGL$-counterparts.

\subsection{$D$-twisted Higgs bundles}

For a divisor $D$ on $C$ with $D = K_C$ or $\deg(D) > 2g-2$, we define $D$-twisted Higgs bundles as follows to extend the classical case where $D= K_C$.

\begin{defn}
A $D$-twisted Higgs bundle is a pair $(E,\Phi: E \ra E \otimes \cO_C(D))$ consisting of a vector bundle $E$ and an $\cO_C$-linear homomorphism $\Phi$. 
\end{defn}

Note that this notion only depends on $D$ up to linear equivalence. In particular, if $\deg(D)>2g-2$, we can and will assume that $D$ is effective, as in \cite{MS}.

There is a notion of (semi)stability for $D$-twisted Higgs bundles involving verifying an inequality of slopes for all $\Phi$-invariant subbundles of $E$. We let $\GLHiggs^D:=\cM_{n,d}^D(C)$ denote the moduli space of semistable $D$-twisted Higgs bundles of rank $n$ and degree $d$. 

Generally speaking, whenever we introduce an object for $D$-Higgs bundles, we use a superscript $D$, which we drop when $D=K_{C}$.

\subsection{$\SL$-Higgs bundles}

In order to have analogues of moduli spaces of ($D$-twisted) $\SL$-Higgs bundles in non-zero degrees $d$, we consider fibres of the morphism
\[ (\det,\tr) : \GLHiggs^D:=\cM_{n,d}^D(C) \ra \cM_{1,d}^D(C) \]
over $(L,0)$. We write $\SLHiggs^D=\cM_{n,L}^D(C):= (\det,\tr)^{-1}(L,0)$ for the $D$-twisted $\SL$-Higgs moduli space. If $n$ and $d$ are coprime, then $\SLHiggs^D$ is smooth. We have $\dim(\SLHiggs^{D})=(n^{2}-1)\deg(D)$.

\subsection{$\PGL$-Higgs bundles}

The moduli of $\PGL_{n}$-Higgs bundles does not really play a role in this paper, however it is central to the original motivation of Hausel-Thaddeus \cite{HT} so we include a brief overview. A $\PGL_n$-Higgs bundle of degree $d$ on $C$ consists of a principal $\PGL_n$-bundle $P \ra C$ of degree $d$ and a Higgs field $\Phi \in H^0(C,P \times^{\PGL_n} \mathfrak{pgl}_n)$.   

The surjection $\GL_n \ra \PGL_n$ induces a morphism between the stacks of $\GL_n$ and $\PGL_n$-Higgs bundles preserving semistability. If the degree and rank are coprime, then the stack of semistable $\PGL_n$-Higgs bundles is the quotient of the stack of semistable $\GL_n$-Higgs bundles by the action of $T^*\Jac(C)$ (see \cite[Section 6]{GPH}). Equivalently, using the natural surjection $\SL_n \ra \PGL_n$, the stack of semistable $\PGL_n$-Higgs bundles is the quotient of the stack of semistable $\SL_n$-Higgs bundles by the finite group $\Gamma:=\Jac(C)[n]$. Consequently, the moduli space of semistable $\PGL_n$-Higgs bundles of degree $d$ can be defined as the following orbifold
\[ \PGLHiggs:= [\SLHiggs/\Gamma] \simeq [\GLHiggs/T^*\Jac(C)] \]
which is smooth as a Deligne--Mumford stack, although its coarse moduli space $\SLHiggs/\Gamma$ is singular. Note that because of that second presentation $\PGLHiggs$ does not depend on the choice of line bundle $L$ appearing in the $\SL$-Higgs moduli space. There are also $D$-twisted variants of the $\PGL_n$-Higgs moduli spaces given by
\[ \PGLHiggs^D:= [\SLHiggs^D/\Gamma] \simeq [\GLHiggs^D/\Jac(C) \times H^0(C,\cO(D))]. \]

\section{Motives of $D$-twisted Higgs moduli spaces}\label{sec mot D-twisted Higgs}

In this section, we generalise our previous result concerning motives of moduli spaces of Higgs bundles \cite[Theorem 1.1]{HPL_Higgs} to moduli spaces of $D$-twisted Higgs bundles. Throughout this section, we work over an arbitrary field $k$ and assume that $C(k) \neq \emptyset$. We fix a coprime rank $n$ and degree $d$ and a divisor $D$ on $C$ with $\deg(C)>2g-2$.

\subsection{Hitchin's scaling action and moduli of chains}

The action of $\GG_m$-action on $\GLHiggs^D:=\cM_{n,d}^D(C)$ scaling the Higgs field has fixed points $[E,\Phi]$ where $E \cong \oplus_{i=0}^r E_i$ (with $r=0$ allowed) and $\Phi(E_i) \subset E_{i+1} \otimes \cO_C(D)$ which determine a chain of vector bundle homomorphisms
\[ F_0 \ra F_1 \ra F_2 \ra \cdots \ra F_r \]
where $F_i :=  E_{i} \otimes \cO_C(iD)$. 

For a brief overview of semistability for chains (which depends on a choice of stability parameter) and chain moduli spaces and stacks see \cite[$\S$2.3]{HPL_Higgs} and \cite{ACGPS,GPHS}. Let $\mChDss$ denote the moduli space of chains $F_0 \ra \cdots \ra F_r$ with tuples of ranks $\underline{m}$ and degrees $\underline{e}$ which are semistable with respect to the chain stability parameter
\[ \alpha_D := (r \deg(D), (r-1) \deg(D), \dots , \deg(D), 0) \in \RR^{r+1}. \]
We refer to $\alpha_{D}$ as the $D$-twisted Higgs stability parameter. Since $\deg(D) > 2g -2$, this stability parameter lies in the cone $\Delta_r^\circ:=\{ (\alpha_0,\dots, \alpha_r) \in \RR^{r+1}: \alpha_{i-1} - \alpha_{i} > 2g-2 \text{ for } 1 \leq i \leq r \}$ of stability parameters with well-understood deformation theory \cite[Section 3]{ACGPS}. Therefore, if $\alpha_D$ is non-critical for the invariants $\underline{m}$ and $\underline{e}$ (so $\alpha_D$-semistability and $\alpha_D$-stability coincide for chains with these invariants), then $\mChDss$ is a smooth projective variety by \cite[Theorem 3.8 vi)]{ACGPS}. 

\begin{prop}\label{prop BB GL D-Higgs}
The $\GG_m$-action on $\GLHiggs^D$ is semi-projective (as in \cite[Definition A.1]{HPL_Higgs}) and
\[(\GLHiggs^D)^{\GG_m}=\bigsqcup_{(\underline{m},\underline{e}) \in \cI} \mChDss\]
where the $D$-twisted Higgs stability parameter $\alpha_D$ is non-critical for all the invariants $\underline{m}$ and $\underline{e}$ appearing in this finite index set $\cI$. Consequently, there is a motivic Bia{\l}ynicki-Birula decomposition
\[ M(\GLHiggs^D) \simeq \bigoplus_{(\underline{m},\underline{e}) \in \cI} M(\mChDss)\{ c_{\underline{m},\underline{e}} \}, \]
where $c_{\underline{m},\underline{e}}$ denotes the codimension of the corresponding Bia{\l}ynicki-Birula stratum in $\GLHiggs^D$.
\end{prop}
\begin{proof}
The fact that this $\GG_m$-action is semi-projective follows by the same argument for $K_C$-Higgs bundles (see \cite[Proposition 2.2]{HPL_Higgs}). As in \cite[Proposition 2.5]{HPL_Higgs}, one shows that semistability of the $\GG_m$-fixed point $[E \cong\oplus_{i\geq 0} E_i, \Phi]$ in $\GLHiggs^D$ corresponds to $\alpha_D$-semistability of the corresponding chain $F_{\bullet}$ with $F_i :=  E_{i} \otimes \cO_C(iD)$. Hence, the claimed description of the $\GG_m$-fixed locus follows as in \cite[Corollary 2.6]{HPL_Higgs}. Since the $\GG_m$-action is semi-projective, one applies the motivic Bia{\l}ynicki-Birula decomposition \cite[Theorem A.4]{HPL_Higgs} to finish the proof.
\end{proof}

\subsection{Stacks of chains, wall-crossings and Harder--Narasimhan recursions}

We let $\Chss$ denote the substack of $\alpha$-semistable chains and let $\Chtau$ denote the substack of chains of $\alpha$-HN type $\tau$. Let us recall a few facts we will need (for details, see \cite[$\S$2.3]{HPL_Higgs} and \cite{ACGPS,GPHS}). If $C(k) \neq \emptyset$ and $\alpha$ is non-critical for $\underline{m}$ and $\underline{e}$, then the stack of $\alpha$-semistable chains is a trivial $\GG_m$-gerbe over the moduli space $\ch^{\alpha,ss}_{\underline{m},\underline{e}}$. Furthermore, if $\alpha \in \Delta_r^\circ$, then $\Chss$ is smooth and so is the stack $\Chtau$ of chains with $\alpha$-HN type $\tau$, as taking the associated graded for the $\alpha$-HN filtration is an affine space fibration by \cite[Lemma 4.6 and Proposition 4.8]{GPHS}.

These HN strata appear in wall-crossings as we vary the stability parameter. We will employ the wall-crossing argument in \cite{GPHS} to end up in a chamber where the stack of semistable chains can be more readily described. More precisely, we will relate them to stacks of injective chain homomorphisms in the following sense.

\begin{defn}
Let $\Chinj$ denote the substack of $\Ch$ consisting of chains 
\[ F_\bullet = (F_0 \stackrel{\phi_1}{\ra} F_{1} \ra \cdots \ra F_{r-1} \stackrel{\phi_r}{\ra} F_r)\]
such that all the homomorphisms $\phi_i$ are injective. 
\end{defn}

If $\underline{m} = (m, \dots, m)$ is constant, $\Chinj$ coincides with the stack of generically surjective chain maps (see \cite{Heinloth_LaumonBDay}); in particular, this stack is smooth by \cite[Lemma 4.9]{GPHS} and \cite[Theorem 3.8 v)]{ACGPS}. In \cite{HPL_Higgs}, we computed the motive of $\Chinj$ using Hecke modification maps and motivic descriptions of small maps; since we will later prove something similar for $\SL$-Higgs bundles, we include a short proof.

\begin{prop}\label{prop motive inj chains}
If $\underline{m} = (m, \dots, m)$ is constant and for $\underline{e} = (e_0,\dots, e_r)$ we set $l_i:=e_i - e_{i-1}$, then the motive of $\Chinj$ in $\DM(k,\QQ)$ is given by
\[M(\Chinj)  \simeq  M(\Bun_{m,e_r}) \otimes \bigotimes_{i=1}^r M(\Sym^{l_i}(C \times \PP^{m-1})).
\]
In particular, the motive of $\Chinj$ lies in $\llangle M(C) \rrangle^{\otimes}$.
\end{prop}
\begin{proof}
Let $\Coh_{0,l}$ denote the stack of length $l$ torsion sheaves and let $\widetilde{\Coh}_{0,l}$ be the stack of length $l$ torsion sheaves with a full flag; that is for a $k$-scheme $S$ we have
\[ \widetilde{\Coh}_{0,l} (S) = \langle \cT_1 \hookrightarrow \cT_2  \hookrightarrow \cdots \hookrightarrow \cT_l \: : \: \cT_i \in \Coh_{0,i}(S)  \rangle. \]
Then Laumon showed the forgetful map $\widetilde{\Coh}_{0,l} \ra \Coh_{0,l}$ is a small map and a $S_l$-torsor over the dense open on which the support of the torsion sheaf consists of $l$ distinct points \cite[Theorem 3.3.1]{Laumon}. Similarly, Heinloth considered in \cite[Proposition 11]{Heinloth_LaumonBDay} a full flag version $\fChinj$ of $\Chinj$ defined by
\[ \fChinj(S) = \langle \cF_0 =\cF_1^0  \hookrightarrow \cF_1^1  \hookrightarrow \cdots \hookrightarrow \cF_1^{l_1} = \cF_1 = \cF_2^0 \hookrightarrow \cdots \hookrightarrow \cF_{r}^{r_l}=\cF_{r}  : \cF_i^j \in \Bun_{m,e_i+j}(S)  \rangle,\]
which admits a forgetful map $f' \colon \fChinj \ra \Chinj$ that is also small and a $ \left(\prod_{i=1}^r S_{l_i}\right)$-torsor on the dense open consists of chains $F_0 \hookrightarrow F_1 \hookrightarrow \cdots \hookrightarrow F_r$ such that the support of $F_i/F_{i-1}$ consists of $l_i$ distinct points for each $1 \leq i \leq r$, as it is obtained by pulling back products of the small maps considered by Laumon under a smooth morphism. In fact, as in \cite[Proposition 11]{Heinloth_LaumonBDay} we have a commutative diagram
\[  \xymatrixcolsep{5pc} \xymatrix{ \Chinj \ar[r]^{\gr \quad \quad \quad \quad} & \prod_{i=1}^r \Coh_{0,l_i} \times \Bun_{m,e_r} \ar[r]^{\beta} & \prod_{i=1}^r \Sym^{l_{i}}(C) \times \Bun_{m,e_r} \\
  \fChinj \ar[r] \ar[u]_{f'} \ar@/_2pc/[rr]_{\mathrm{supp}} & \prod_{i=1}^r \widetilde{\Coh}_{0,l_i} \times \Bun_{m,e_r} \ar[r] \ar[u]_{f} & \prod_{i=1}^r C^{l_i} \times \Bun_{m,e_r}, \ar[u]
    } \]
where $\mathrm{gr}(F_0 \hookrightarrow \cdots \hookrightarrow F_r) :=(F_1/F_0, \dots, F_r/F_{r-1},F_r)$ is a smooth morphism and 
\[ \mathrm{supp}(F_0 =F_1^0 \hookrightarrow \cdots \subset F_1^{l_1} = F_1 =F_2^0 \hookrightarrow \cdots \hookrightarrow F_{r}^{l_r} = F_r):=(\supp(F_i^j/F_i^{j-1})_{1 \leq i \leq r, 1 \leq j \leq l_i}, F_r)\] is a $\left(\sum_{i=1}^{r} l_i\right)$-iterated $\PP^{m-1}$-bundle, as we take $\sum_{i=1}^{r} l_i$ successive elementary Hecke modifications of a rank $m$ bundle. Since $\mathrm{supp}$ is a $\left(\sum_{i=1}^{r} l_i\right)$-iterated $\PP^{m-1}$-bundle, we have 
\[ M(\fChinj) \simeq M(\Bun_{m,e_r}) \otimes \bigotimes_{i=1}^r M(C \times \PP^{m-1})^{\otimes l_i}. \]
Since $f'$ is a small map and a torsor under $G = \prod_{i=1}^r S_{l_i}$ on a dense open, by applying the motivic description of such small maps (see Theorem \ref{thm mot small map} below and \cite[Theorem 2.1]{HPL_coh}) we obtain
\[  M(\Chinj) \simeq  M(\fChinj)^G \simeq M(\Bun_{m,e_r}) \otimes \bigotimes_{i=1}^r M(\Sym^{l_i}(C \times \PP^{m-1})),\]
which completes the proof.
\end{proof}

To complete the proof, we state the relevant result on small maps from \cite{HPL_coh}, which generalises  \cite{dCM_small} and \cite[Theorem 2.13]{HPL_formula}.

\begin{thm}\label{thm mot small map} \cite[Theorem 2.1]{HPL_coh} 
Let $\pi : \mathfrak{X} \ra \cY$ be a small proper surjective morphism of smooth Artin stacks such that there exists a dense open $\cY^\circ \subset \cY$ with preimage $\mathfrak{X}^\circ$ such that $\pi : \mathfrak{X}^\circ \ra \cY^\circ$ is a $G$-torsor. Then the $G$-action on $M(\mathfrak{X}^\circ)$ in $\DM(k,\QQ)$ extends to $M(X)$ and we have
\[ M(\mathfrak{X})^G \simeq M(\cY). \]
\end{thm}

We can now employ the wall-crossing argument of \cite{GPHS} as in \cite[$\S$6]{HPL_Higgs} to conclude an analogous result to \cite[Theorem 1.1]{HPL_Higgs}.

\begin{thm}\label{thm motive GL D Higgs gen by C}
Assume that $C(k) \neq \emptyset$. Let $D$ be a divisor on $C$ with either $D=K_{C}$ or $\deg(D)>2g-2$. The motive of the $D$-twisted Higgs moduli space $\GLHiggs^D:=\cM_{n,d}^D(C)$ for coprime rank and degree lies in the subcategory $\langle M(C) \rangle^{\otimes}$ of $\DM(k,\QQ)$ and is a direct factor of $M(C^m)$ for sufficiently high $m$.
\end{thm}
\begin{proof}
The case where $D=K_{C}$ is covered by \cite[Theorem 1.1]{HPL_Higgs}, so we assume $\deg(D)>2g-2$. The final statement follows from the first using a purity argument as in \cite[\S 6.3]{HPL_Higgs}. For the first, by Proposition \ref{prop BB GL D-Higgs}, it suffices to show for each index $(\underline{m},\underline{e}) \in \cI$ that $M(\mChDss) \in \langle M(C) \rangle^{\otimes}$. As the map  $\ChDss \ra \mChDss$ from the stack to its good moduli space is a trivial $\GG_m$-gerbe, it suffices to show that  $M(\mChDss)\in \llangle M(C) \rrangle^{\otimes}$ by \cite[Lemma 6.5]{HPL_Higgs}. 
For this, we use Proposition \ref{prop motive inj chains} and \cite[Theorem 6.3]{HPL_Higgs} recursively, which employs the wall-crossing argument and HN recursion of \cite{GPHS}. The basic idea is that, as stated in \cite[Proposition 2.9]{HPL_Higgs}, there is a ray in the cone $\Delta_r^\circ$ of stability parameters starting at $\alpha_D$ and ending at a stability parameter $\alpha_\infty$ such that either
\begin{itemize}
\item $\Ch^{\alpha_\infty,ss} = \emptyset$ if $\underline{m}$ is non-constant or
\item $\Ch^{\alpha_\infty,ss} \subset \Chinj$ and $\Chinj$ is a union of $\alpha_\infty$-HN strata if $\underline{m}$ is constant.
\end{itemize} 
The cone $\Delta_r^\circ$ of stability parameters admits a wall and chamber decomposition and as this ray crosses each of the (finitely many) walls between $\alpha_D$ and $\alpha_\infty$, the semistable locus for the stability parameter on the wall is the union of the semistable loci on either side of the wall and finitely many HN-strata. Since by induction the motives of HN-strata can be related to motives of stacks of semistable for $\alpha \in \Delta_r^\circ$ (\textit{cf.} \cite[Lemma 6.2]{HPL_Higgs}), we see from looking at the Gysin triangles associated to each wall-crossing that it suffices to show by a HN-induction that $\Chinj$ lies in $\llangle M(C) \rrangle^{\otimes}$ for $\underline{m}$ constant, which is proven in  Proposition \ref{prop motive inj chains}.
\end{proof}

\section{Motives of $\SL$-Higgs moduli spaces}\label{sec mot SL-Higgs}

Throughout this section, we assume $k$ is an arbitrary field and $C(k) \neq \emptyset$. We fix a rank $n$, a line bundle $L$ on $C$ of degree $d$ coprime to $n$ and a divisor $D$ on $C$ with either $D = K_C$ or $\deg(D) > 2g -2$. Our goal is to prove that the motive of the $D$-twisted $\SL$-Higgs moduli space $\SLHiggs^D:=\cM_{n,L}^D(C)$ is abelian.

\subsection{The scaling action on the $\SL$-Higgs moduli space}

The scaling $\GG_m$-action on $\GLHiggs^D$ restricts to $\SLHiggs^D$ and the fixed loci are chain moduli spaces with fixed total determinant 
\[ \ch^{\alpha_D,ss}_{\underline{m},\underline{e},L\otimes \cO_C(\sum_i i m_i D)} := \left\{ F_0 \ra \cdots \ra F_r :  \det \left(\bigoplus_{i=0}^r F_i \right) \cong L \otimes \cO_C\left(\sum_i i m_i D\right) \right\},\]
as if $E \simeq \oplus_{i=0}^r E_i$ has determinant $L$ and $F_i := E_i \otimes \cO_C(iD)$, then $\oplus_{i=0}^r F_i$ has determinant $L \otimes \cO_C(\sum_i i m_i D)$. 
Consequently, we have the following motivic Bia{\l}ynicki-Birula decomposition 
\begin{equation}\label{motivic BB for SL-Higgs}
 M(\SLHiggs^D) \simeq \bigoplus_{(\underline{m},\underline{e}) \in \cI} M(\ch^{\alpha_D,ss}_{\underline{m},\underline{e},L\otimes \cO_C(\sum_i i m_i D)} )\{ c_{\underline{m},\underline{e}} \}, 
 \end{equation}
where $c_{\underline{m},\underline{e}}$ denotes the codimension of the corresponding Bia{\l}ynicki-Birula stratum in $\SLHiggs^D$.

\subsection{The stack of injective chain homomorphisms of fixed total determinant}

For an arbitrary line bundle $L$ on $C$ (different in general from the $L$ above, and in particular not assumed to be of degree $d$) and constant tuple of ranks $\underline{m}$, the aim of this section is to prove that the motive of the substack $\ChinjL \hookrightarrow \Chinj$ of injective chain homomorphisms with total determinant $L$
\[ \ChinjL := \left\langle F_0 \hookrightarrow F_1 \hookrightarrow \cdots \hookrightarrow F_r \in \Chinj : \det \left(\bigoplus_{i=0}^r F_i \right) \cong L \right\rangle\]
is abelian. We will prove this by considering stacks of full flags $\fChinjL \hookrightarrow \fChinj$ with total determinant $L$ and proving that
\begin{enumerate}[i)]
\item the motive of $\fChinjL$ is abelian (Proposition \ref{prop motive flag chain inj L abelian} below) and
\item the forgetful map $f'_L : \fChinjL \ra \ChinjL$ is small and a torsor under a product of symmetric groups on a dense open (Proposition \ref{prop fL is small} below).
\end{enumerate}

In both proofs it will be useful to consider a commutative diagram given by taking fibres of the commutative diagram in Proposition \ref{prop motive inj chains} over $L \in \Pic(C)$ under a certain weighted determinant map. As above, let us write $l_i :=e_i - e_{i-1}$. First, similarly to \cite[$\S$5]{GPHS}, we define a weighted determinant map
\[ \begin{array}{rccl} \omega \colon & \prod_{i=1}^{r} C^{(l_i)} \times \Bun_{m,e_r} & \ra &  \Pic(C) \\ & (D_1, \dots, D_r, F_r) & \mapsto & \det(F)^{\otimes r +1} \otimes \bigotimes_{i=1}^r \cO_C(-iD_i) \end{array}\]
so that $\ChinjL$ is the fibre over $L$ of the composition $\omega \circ \beta \circ \gr : \Chinj \ra \Pic(C)$. Let us write $\cT:= \prod_{i=1}^r \Coh_{0,l_i}$ and $\widetilde{\cT}:= \prod_{i=1}^r \widetilde{\Coh}_{0,l_i}$ and $C^{(\underline{l})}:=\prod_{i=1}^r C^{(l_i)} $ and $C^{\underline{l}}:=\prod_{i=1}^r C^{l_i}$. Let $\gamma:= \omega \circ \beta$ and we denote the corresponding morphisms in the full flag setting by $\widetilde{\gamma}:= \widetilde{\omega} \circ \widetilde{\beta}$. Then we have a commutative diagram
\[ \xymatrixcolsep{1.5pc} 
 \xymatrix{ \Chinj \ar[rr]^{\gr} & & \cT \times \Bun_{m,e_r} \ar[rr]^{\beta}  & & C^{(\underline{l})} \times \Bun_{m,e_r} \ar@/^3pc/[rrdd]^{\omega} & &  \\
& \ChinjL  \ar@{_{(}->}[lu] \ar[rr]^{\gr_L  \quad \quad \quad } & & \gamma^{-1}(L)  \ar@{_{(}->}[lu] \ar[rr] & & \omega^{-1}(L)  \ar@{_{(}->}[lu]  & \\
& & & & & & \Pic(C) \\
& \fChinjL  \ar@{^{(}->}[ld] \ar[rr] \ar[uu]^{f'_L} & & \widetilde{\gamma}^{-1}(L)  \ar@{^{(}->}[ld] \ar[rr] \ar[uu]^{f_L} & & \widetilde{\omega}^{-1}(L)  \ar@{^{(}->}[ld]\ar[uu]  & \\
\fChinj \ar[rr]_{\widetilde{\gr}} \ar[uuuu]^{f'}  & & \widetilde{\cT} \times \Bun_{m,e_r} \ar[rr]_{\widetilde{\beta}} \ar[uuuu]^{f}  & & C^{\underline{l}} \times \Bun_{m,e_r} \ar[uuuu] \ar@/_3pc/[rruu]_{\tilde{\omega}} &
 } 
 \]
where $f$ and $f'$ are small and $\prod_{i=1}^r S_{l_i}$-torsors on a dense open, $\gr$ and also its base change $\gr_L$ are smooth, and $\supp:=\widetilde{\beta} \circ \widetilde{\gr}$ is an iterated projective bundle (see Proposition \ref{prop motive inj chains}). 

\begin{prop}\label{prop motive flag chain inj L abelian}
The motive of $\fChinjL$ is abelian.
\end{prop}
\begin{proof}
Since $\mathrm{supp}:= \widetilde{\beta} \circ \widetilde{\gr}: \fChinj \ra C^{\underline{l}} \times \Bun_{m,e_r}$ is a $\left(\sum_{i=1}^r l_i\right)$-iterated $\PP^{m-1}$-bundle (see Proposition \ref{prop motive inj chains}), its restriction $\mathrm{supp}_L : \fChinjL \ra \widetilde{\omega}^{-1}(L)$ is also an iterated projective bundle and so it suffices to prove that the motive of $\widetilde{\omega}^{-1}(L)$ is abelian.

For $\underline{c}:=(c_{i,j})_{1\leq i \leq r,1 \leq j \leq l_i} \in C^{\underline{l}}$, let $\cO_C(\underline{c}):=\cO_C(- \sum_{i=1}^r \sum_{j=1}^{l_i} i c_{ij})$. Then by definition
\[ \widetilde{\omega}^{-1}(L) = \left\langle ( \underline{c},F) \in C^{\underline{l}} \times \Bun_{m,e_r} : \det(F)^{\otimes r+1} \cong L \otimes \cO_C(\underline{c}) \right\rangle. \]
Note that as $\det(F)^{\otimes r+1}$ has degree $(r+1)e_r$, so does $L \otimes \cO_C(\underline{c})$. Consider the fibre product $B$
\[  \xymatrix{ B \ar[r] \ar[d] & C^{\underline{l}} \ar[d] \\ \Pic^{e_r}(C) \ar[r]^{\cdot (r+1) \quad} & \Pic^{(r+1)e_r}(C)
}\]
where the bottom arrow is the map taking $(r+1)$-powers and the right arrow is given by $\underline{c}\mapsto L \otimes \cO_C(\underline{c})$. Hence, the scheme $B$ parametrises pairs $(N,\underline{c}) \in \Pic^{e_r}(C) \times C^{\underline{l}}$ such that $N^{\otimes r+1} \cong L \otimes \cO_C(\underline{c})$. Let $\cN \in \Pic^{e_r}(C\times B)$ denote the pullback of the Poincar\'{e} line bundle on $C \times \Pic^{e_r}(C)$; then 
\[ \widetilde{\omega}^{-1}(L) \cong \Bun_{C \times B/B,m, \cN} \]
is isomorphic to the stack of rank $m$ vector bundles on $C \times B/B$ with determinant $\cN$.

To show that the motive of $\Bun_{C \times B/B,m, \cN}$ is abelian, it suffices to show that $M(B)$ is abelian by Theorem \ref{thm motive rel Bun fixed det}. To prove that the motive of $B$ is abelian, we let $C' \ra C$ denote the finite \'{e}tale cover of $C$ obtained by base change under the multiplication by $r+1$ map on $\Jac(C)$, where we fix $x_0 \in C$ to determine an Abel--Jacobi map $C \ra \Jac(C)$ given by $c \mapsto \cO_C(x_0 - c)$. We shall view elements in $C'$ as pairs $(c,M)$ where $c \in C$ and $M$ is a line bundle with $M^{\otimes r+1} \cong \cO(x_0 -c)$. We claim that there is a surjective morphism $(C')^{\underline{l}} \ra B$. To define this map, we use the universal property of the fibre product $B$: we take the natural morphism $(C')^{\underline{l}} \ra C^{\underline{l}}$ and the morphism $(C')^{\underline{l}} \ra \Pic^{e_r}(C)$ given by \[ (c_{ij},M_{ij})_{1 \leq i \leq r, 1 \leq j \leq l_i} \mapsto L' \otimes \bigotimes_{i=1}^r \bigotimes_{j = 1}^{l_i} M_{ij}^{\otimes i},\]
where $L' \in \Pic^{e_r}(C)$ is a fixed $(r+1)$-root of $\cO_C(\underline{x_0}):=L \otimes \cO_C(-\sum_{i=1}^r \sum_{j=1}^{l_i}  i x_0)$. Since
\[ \left(L' \otimes \bigotimes_{i=1}^r \bigotimes_{j = 1}^{l_i} M_{ij}^{\otimes i} \right)^{\otimes r + 1} \simeq L \otimes \cO_C\left(-\sum_{i=1}^r \sum_{j=1}^{l_i} i x_0\right)\otimes \bigotimes_{i=1}^r \bigotimes_{j = 1}^{l_i} \cO(i(x_0 - c_{ij})) \simeq L \otimes \cO_C(\underline{c})\]
the corresponding compositions from $(C')^{\underline{l}}$ to $\Pic^{(r+1)e_r}(C)$ commute and consequently there is a morphism $(C')^{\underline{l}} \ra B$, which is clearly surjective.
By Lemma \ref{lemma surj map ab motive} below, we conclude that $B$ is abelian and thus also $\widetilde{\omega}^{-1}(L)$ is abelian.
\end{proof}

To complete the proof, we use the following result, which is well-known to experts but for which we did not find a suitable reference. Recall that there is a fully faithful embedding of the category of Chow motives with rational coefficients over $k$ into $\DM(k,\QQ)$, so that it is enough to show that $M(B)$ is abelian as a Chow motive in the sense of \cite[Definition 1.1]{Wildeshaus}. Recall as well that abelian Chow motives in this sense are Kimura finite-dimensional \cite[Proposition 1.8]{Wildeshaus}.

\begin{lemma}\label{lemma surj map ab motive}
Let $X \ra Y$ be a surjective morphism of smooth projective $k$-varieties. If the Chow motive $M(X)$ is Kimura finite-dimensional (in particular, if $M(X)$ is abelian), then $M(Y)$ is a direct factor of $M(X)$. In particular, if $M(X)$ is abelian, then $M(Y)$ is abelian.
\end{lemma}
\begin{proof}
  The surjective morphism $X\ra Y$ induces a surjective morphism of Chow motives in the sense of \cite[Definition 6.5]{Kimura} by \cite[Remark 6.6]{Kimura}. By \cite[Lemma 6.8]{Kimura}, this implies there is a morphism $\eta:M(Y)\to M(X)$ of Chow motives such that $M(f)\circ \eta=\id_{M(Y)}$. By \cite[Proposition 6.9]{Kimura}, the fact that $M(X)$ is Kimura finite-dimensional then also implies that $M(Y)$ is Kimura finite-dimensional.

Now consider the category $\cM_{\rat}^{\fd}(k,\QQ)$ of Kimura finite-dimensional Chow motives over $k$ with $\QQ$-coefficients and the category $\cM_{\num}(k,\QQ)$ of pure motives over $k$ with rational coefficients with respect to numerical equivalence. Denote by $(-)^{\num}:\cM_{\rat}^{\fd}(k,\QQ)\to \cM_{\num}(k,\QQ)$  the natural functor. This functor is conservative and lifts idempotents by \cite[Corollaires 3.15-16]{Andre-bourbaki}.
  
The category $\cM_{\num}(k,\QQ)$ is abelian semi-simple by Jannsen's theorem, so the existence of $\eta$ implies that $M(f)^{\num}:M(Y)^{\num}\to M(X)^{\num}$ makes $M(Y)^{\num}$ into a direct factor of $M(X)^{\num}$. Let $p:M(X)\to M(X)$ be a projector such that $p^{\num}$ has image $M(Y)^{\num}$. In other words, the composition $M(Y)^{\num}\stackrel{M(f)^{\num}}{\lra} M(X)^{\num}\stackrel{p^{\num}}{\lra} \mathrm{Im}(p^{\num})$ is an isomorphism. Since $(-)^{\num}$ is conservative, this implies that $M(Y)\simeq \mathrm{Im}(p)$ is a direct factor of $M(X)$.
\end{proof}

It is not obvious that $f'_L$ is small, as it is the base change of (the small map) $f'$ under the non-flat morphism $\ChinjL \hookrightarrow \Chinj$. To prove that $f'_L$ is small we will instead prove that $f_L$ is small.

\begin{prop}\label{prop fL is small}
The forgetful map $f'_L : \fChinjL \ra \ChinjL$ is small and a $\prod_{i=1}^r S_{l_i}$-torsor on a dense open. 
\end{prop}
\begin{proof}
It suffices to prove that $f_L$ is small and a $\prod_{i=1}^r S_{l_i}$-torsor on a dense open, as $f'_L$ is the pullback of $f_L$ under the smooth map $\gr_L$. Recall that the forgetful map 
\[ \widetilde{\cT}:=\prod_{i=1}^r \widetilde{\Coh}_{0,l_i}  \ra \cT:= \prod_{i=1}^r \Coh_{0,l_i}\]
is small and a $\prod_{i=1}^r S_{l_i}$-torsor on a dense open. 

Let $\cL \ra C \times \cT$ be the line bundle with $\cL_t:=L \otimes \bigotimes_{i=1}^r \cO_C(i\supp T_i)$ over $t = (T_1,\dots,T_r) \in \cT$. Then by definition
\[ {\gamma}^{-1}(L) = \left\langle ( t,F) \in \cT \times \Bun_{m,e_r} : \det(F)^{\otimes r+1} \cong \cL_t \right\rangle=:\Bun_{C \times \cT/\cT,m, \det( \: )^{\otimes r+1} \simeq \cL}. \]
Let $\cT'$ be the fibre product
\[  \xymatrix{\cT' \ar[r] \ar[d] & \cT \ar[d]^{\cL} \\ \Pic^{e_r}(C) \ar[r]^{\cdot (r+1) \quad} & \Pic^{(r+1)e_r}(C)
}\]
where the bottom morphism is multiplication by $(r+1)$ and the right morphism is given by $\cL$. Let $\cL' \in \Pic^{(r+1)e_r}_{C \times \cT'/\cT'}(\cT')$ denote the pullback of $\cL$. Then by construction all $(r+1)$-roots of $\cL'$ exist and
\[ \Bun_{C \times \cT'/\cT',m, \det( \: )^{\otimes r+1} \simeq \cL'} = \bigsqcup_{\begin{smallmatrix} (r+1)\mathrm{-root } \\ \cN \text{ of } \cL' \end{smallmatrix}}  \Bun_{C \times \cT'/\cT',m, \cN}. \]
Since each $\Bun_{C \times \cT'/\cT',m, \cN}$ is smooth\footnote{Here $\cT'$ is a stack, but we can base change to an atlas of $
\cT'$ to prove $\Bun_{C \times \cT'/\cT',m, \cM}$ is smooth via descent.} over $\cT'$, we have that $\Bun_{C \times \cT'/\cT',m, \det( \: )^{\otimes r+1} \simeq \cL'}$ is smooth over $\cT'$ is smooth. Consequently 
\[\gamma^{-1}(L)\simeq \Bun_{C \times \cT/\cT,m, \det( \: )^{\otimes r+1} \simeq \cL} \ra \cT\]
is smooth by \'{e}tale descent under the finite \'{e}tale cover $\cT' \ra \cT$. Since $f_L$ is the base change of $\widetilde{\cT} \ra \cT $ under the smooth map $\gamma^{-1}(L) \ra \cT$, we conclude that the morphism $f_L$ is also small and a $\prod_{i=1}^r S_{l_i}$-torsor on a dense open.
\end{proof}

We can now conclude that the stack of injective chain homomorphisms with fixed total determinant has abelian motive. 

\begin{cor}\label{cor motive inj chain fixed det abelian}
For a constant tuple of ranks $\underline{m}$, the motive of $\ChinjL$ is abelian.
\end{cor}
\begin{proof}
This follows from Propositions \ref{prop motive flag chain inj L abelian} and \ref{prop fL is small}, as by Theorem \ref{thm mot small map}, the motive of $\ChinjL$ is a direct factor of the motive of $\fChinjL$ and thus is also abelian.
\end{proof}

\subsection{The $\SL$-Higgs moduli space has abelian motive}

We are now able to prove the main result of this section.

\begin{thm}\label{thm mot SL Higgs ab}
Assume that $C(k) \neq \emptyset$. Let $D$ be a divisor on $C$ with either $D=K_{C}$ or $\deg(D)>2g-2$. The motive in $\DM(k,\QQ)$ of the $D$-twisted $\SL$-Higgs moduli space $\SLHiggs^D:=\cM_{n,L}^D(C)$ for a line bundle $L$ of coprime degree to $n$ is abelian.
\end{thm}
\begin{proof}
The proof follows by adapting the argument of Theorem \ref{thm motive GL D Higgs gen by C} but using the motivic Bia{\l}ynicki-Birula decomposition \eqref{motivic BB for SL-Higgs} in place of Proposition \ref{prop BB GL D-Higgs}, and Corollary \ref{cor motive inj chain fixed det abelian} in place of Proposition \ref{prop motive inj chains}.
\end{proof}

\begin{rmk}\label{rmk SL Higgs gen by etale covers}
From the proof of Proposition \ref{prop motive flag chain inj L abelian}, one sees that 
\[ M(\SLHiggs^D) \in \langle M(C_r) : 1 \leq r \leq n \rangle^{\otimes}, \]
where $C_r \ra C$ is the $r^{2g}$-\'{e}tale cover obtained by pullback along the multiplication by $r$ map on $\Jac(C)$. In fact, one could ask if this category is genuinely larger than the category generated by $M(C)$. In \cite[Proposition 5.7]{FHPL_rank3}, we show with L. Fu that if $C$ is a general complex curve then $M(\SLHiggs) \notin \langle M(C) \rangle^{\otimes}$ and so this category really is larger in this case.
\end{rmk}

We can now prove Theorem \ref{main thm motives abelian}.

\begin{proof}[Proof of Theorem \ref{main thm motives abelian}]
This follows from Theorem \ref{thm mot SL Higgs ab} and Remark \ref{rmk SL Higgs gen by etale covers} combined with the argument in \cite[\S 6.3]{HPL_Higgs} to show the further claim that the motive is a direct factors of a products of \'{e}tale covers of $C$.
\end{proof}

\begin{rmk}\label{rmk comparison GS}
As mentioned in the introduction, in upcoming work of Groechenig and Shiyu Shen \cite{GS}, they also adapt the techniques of \cite{GPHS,GPH} to study the class of the $\SL$-Higgs moduli space in the Grothendieck ring of varieties in order to deduce its Betti cohomology is torsion-free. However, in \cite{GS}, they describe the class of the stack of injective chain homomorphisms of fixed total determinant using \'{e}tale covers of products of symmetric powers of curves, whereas our proof is simplified by working with the flag version of the stack and using properties of small maps similarly to the approach for $\GL$-Higgs moduli spaces in \cite{Heinloth_LaumonBDay} .
\end{rmk}

\section{Motivic isotypical decompositions and orbifold motives}

In this section, we introduce the the motivic isotypical decomposition for the $\Gamma$-action on $M(\SLHiggs)$, as well as the set-up and notation of \cite{MS} needed to describe this. Throughout we work with motives with coefficients in $\Lambda = \QQ(\zeta_n)$, where $\zeta_n$ is a primitive $n$th root of unity.

\subsection{Isotypical decompositions}

Fix a rank $n$ and line bundle $L \in \Pic^d(C)$ of coprime degree to $n$. Recall that $\Jac(C)$ acts on $\GLHiggs^D:=\cM_{n,d}^D(C)$ and also $\Gamma := \Jac(C)[n]$ acts on $\SLHiggs^D:=\cM_{n,L}^D(C)$ by tensoring. Associated to the $\Gamma$-action on the motive $M(\SLHiggs^D)$ of the $D$-twisted $\SL$-Higgs moduli space, we have the following isotypical decomposition in $\DM(k,\Lambda)$
\begin{equation}\label{eq isotypical mot}
 M(\SLHiggs^D) \simeq \bigoplus_{\kappa \in \hGamma} M(\SLHiggs^D)_\kappa \simeq M(\SLHiggs^D/\Gamma) \oplus \bigoplus_{\kappa \neq 0 \in \hGamma} M(\SLHiggs^D)_\kappa,
\end{equation}
where the $\Gamma$-invariant piece  $M(\SLHiggs^D)^{\Gamma} \simeq M(\SLHiggs^D/\Gamma) \simeq M(\PGLHiggs^D)$ is the motive of the $D$-twisted $\PGL_n$-Higgs moduli space. Note that the piece indexed by non-trivial characters $\kappa$ is non-zero, as $\Gamma$ acts non-trivially on $M(\SLHiggs^D)$. Indeed this was already observed on the level of cohomology in rank $n = 2$ by Hitchin \cite{Hitchin} and the above decomposition on the level of cohomology was recently described by Maulik and Shen \cite{MS}.

\begin{rmk}
This is one way in which moduli of Higgs bundles behaves very different to moduli of vector bundles. For the moduli spaces $\cN:=\cN_{n,d}(C)$ (resp. $\cN_{L}$) of stable vector bundles (resp. of fixed determinant) of rank $n$ and coprime degree $d$, Harder and Narasimhan \cite{HN} showed that $\Gamma$-action on the $\ell$-adic cohomology of $\cN_{L}$ is trivial and concluded that the $\ell$-adic cohomology of $\cN$ is isomorphic to the tensor product of the cohomology of $\cN_{L}$ and $\Jac(C)$. In joint work with Fu \cite[Theorem 1.1]{FHPL_rank2}, we lifted this to an isomorphism of Chow motives. However, the situation for Higgs bundles is very different.

Another difference concerns tautological generation of the cohomology ring. For $n$ and $d$ coprime, Atiyah and Bott \cite{atiyah_bott} showed the cohomology of $\cN$ is generated by tautological classes (the K\"{u}nneth components of the Chern classes of the universal bundle) and by the result of Harder and Narasimhan \cite{HN}, the same is true for the moduli space $\cN_{L}$ of stable vector bundles of fixed determinant. Markman proved that the cohomology of the $\GL$-Higgs moduli space is also generated by the tautological classes \cite{Markman}. However, the cohomology of the $\SL$-Higgs moduli space is not generated by tautological classes.

On a motivic level, this difference between the case of moduli of vector bundles and Higgs bundles (with fixed determinant) is illustrated by the fact that the motives of the vector bundle moduli spaces $\cN$ and $\cN_L$ are both generated by the motive of $C$ (see \cite[Proposition 4.1]{FHPL_rank2}), whereas in the case of Higgs bundles, although the motive of the $\GL$-Higgs moduli space $\GLHiggs$ is generated by the motive of $C$ by \cite{HPL_Higgs}, this is not true for the $\SL$-Higgs moduli space $\SLHiggs$ for a general complex curve $C$ by \cite[Proposition 5.7]{FHPL_rank3}.
\end{rmk}

\subsection{Weil pairing and cyclic covers}\label{sec Weil pairing}

The $n$-torsion in the Jacobian $\Gamma := \Jac(C)[n]$ has a natural non-degenerate Weil pairing
\[ \langle - , - \rangle : \Gamma \times \Gamma \ra \mu_n \]
which allows us to identify $\Gamma \cong \hGamma$. Alternatively, we can use the Abel-Jacobi map to identify $\Gamma$
\[ \Gamma \cong H^1(C,\ZZ/n\ZZ) \cong \Hom(\pi_1(C),\ZZ/n\ZZ) \]
with cyclic covers of $C$ of degree dividing $n$. The Weil pairing corresponds to the intersection pairing on $H^1(C,\ZZ/n\ZZ)$.

\begin{notation}
For $\gamma \in \Gamma$, we denote the corresponding character in $\hGamma$ by $\kappa:=\kappa(\gamma)= \langle \gamma , - \rangle$ and the corresponding cyclic cover by $\pi:=\pi_\gamma : C_\gamma \ra C$ and write $m_\gamma:= \ord(\gamma) = \deg (\pi)$, which divides $n$ and so we write $n_\gamma := n/m_{\gamma}$, and write $G_\pi:= \Gal (C_\gamma/C) \cong \ZZ/m_\gamma \ZZ$. 
\end{notation}

Note that $g_{C_\gamma} = 1 + m_\gamma(g_C -1)$ by Riemann--Hurwitz. More concretely, if $\gamma \in \Gamma$ of order $m\gamma$ corresponds to $L_\gamma \in \Jac(C)$, then $C_\gamma$ is constructed as a closed subscheme of the total space of $L_\gamma$ given by taking fibrewise the $m_\gamma$-roots of unity.

\subsection{Fixed loci and relative Higgs moduli spaces for cyclic covers}\label{sec fixed loci and rel Higgs}

In this section, we fix $\gamma \in \Gamma$. Our goal is to interpret the motive of the $\gamma$-fixed locus in the $D$-twisted $\SL$-Higgs moduli space in terms of a direct summand of a motive of a relative Higgs moduli space for the associated cyclic cover $\pi : C_\gamma \ra C$. In fact, a description of the $\gamma$-fixed locus in the moduli spaces of vector bundles in terms of a relative moduli space for $\pi$ was given by Narasimhan and Ramanan \cite{NR}, and as observed in \cite[Section 7]{HT}, the arguments also extend to Higgs bundles. The compatibility of this description with the corresponding Hitchin fibrations was described in \cite[Section 1.5]{MS}.

\begin{defn}\label{def fixed locus}
The fixed locus in $\SLHiggs^D$ of an element $\gamma \in \Gamma$ is denoted $\cM^D_\gamma := (\SLHiggs^D)^\gamma$. We let $\gammaHit^D$ denote the image of $\cM^D_\gamma$ under the $D$-twisted $\SL$-Hitchin map $h_{L}^D \colon \SLHiggs \ra \SLHit^D$ and write $h_\gamma^D \colon\cM^D_\gamma \ra \gammaHit^D$ for the restricted Hitchin map and define
\[ d^D_\gamma:= \codim(i_\gamma^D : \gammaHit^D \hookrightarrow \SLHit^D).\]
\end{defn}

The fixed locus $\cM^D_\gamma$ is smooth as the fixed locus of an action of a cyclic group on a smooth scheme in characteristic $0$. 

There is an induced $\Gamma$-action on $\cM^D_\gamma$, which also has an associated motivic isotypical decomposition in $\DM(k,\Lambda)$. Note that there is no $\Gamma$-action on the $D$-twisted $\SL$-Hitchin base $\SLHit^D$ and thus $\gammaHit^D$ is not a fixed locus, but rather the image of a fixed locus under the $\SL$-Hitchin map.

Now consider the cyclic cover $\pi : C_\gamma \ra C$ associated to $\gamma$ of degree $m_\gamma:= \ord(\gamma) = \deg (\pi_\gamma)$ with $n_\gamma := n/m_{\gamma}$ and $G_\pi:= \Gal (C_\gamma/C) \cong \ZZ/m_\gamma \ZZ$.

\begin{defn}\label{def relative SLHiggs moduli space}
We let $D_\gamma := \pi^*D$ and define a map
\[ \begin{array}{rccl} Q_{\gamma} \colon & \cM_{n_\gamma,d}^{D_\gamma}(C_\gamma) & \ra  & \cM_{1,d}^{D}(C) \\
& [E,\Phi]  & \mapsto & [\det(\pi_{*}(E)),\tr(\pi_{*}(\Phi))] \end{array} \]
and we define the $\pi$-relative $D$-twisted $\SL$-Higgs moduli space to be $\piHiggs^D:=Q_{\gamma}^{-1}([L,0])$. The $D$-twisted $\GL$-Hitchin fibration $h_{n_\gamma,d}^{D_\gamma}(C_\gamma) :\cM_{n_\gamma,d}^{D_\gamma}(C_\gamma) \ra \cA_{n_\gamma,d}^{D_\gamma}(C_\gamma)$ restricts to a Hitchin fibration
\[ h_{\pi}^D : \piHiggs^D \twoheadrightarrow \cA_{\pi}^D.\]
\end{defn}

As explained in \cite[$\S$1.2]{MS}, the morphism $Q_{\gamma}$ is smooth, as it is the composition of two smooth maps. Hence the $\pi$-relative $D$-twisted $\SL$-Higgs moduli space is smooth, but is not connected (see \cite[Proposition 1.1]{MS}). We summarise the geometric properties from \cite{MS} which we need.

\begin{prop}\label{prop fixed relative} \cite[$\S$1.5]{MS}
The morphism $\piHiggs^D \ra \SLHiggs^D$ given by pushforward along $\pi$ has image $\gammaHiggs^D$. Furthermore, there is a commutative diagram
\[ \xymatrix{ \piHiggs^D \ar[r]^{ p^D_\gamma} \ar[d]_{h_{\pi}^D} & \gammaHiggs^D  \ar[d]^{h_\gamma^D}  \\
\piHit^D \ar[r]^{q^D_\gamma} & \gammaHit^D }
\]
where $p^D_\gamma$ and $q^D_\gamma$ are geometric $G_{\pi}$-quotients and $p^D_\gamma$ is $\Gamma$-equivariant. The action of $G_{\pi}$ on $\piHiggs^D$ is free and permutes the connected components (however, the action of $G_{\pi}$ on $\SLHit$ is not free). The quotient $\gammaHiggs^{D}\simeq \piHiggs^D/G_{\pi}$ is connected.
\end{prop}

Consequently, we compute some dimension, which are used (sometimes implicitly) in \cite{MS}.

\begin{lemma}\label{lemma dim form}
  The following dimension formulae hold.
  \begin{enumerate}[label=\emph{\roman*)}]
  \item $\dim \cM_\pi^D=\dim \cM_{\gamma}^{D}=(nn_{\gamma}-1)\deg(D).$
\item $\dim \cA_\pi^D=\dim \cA^D_\gamma=\frac{n(n_{\gamma}+1)\deg(D)}{2}-(n-1)(g-1)-\deg(D).$    
\item 

$\codim_{\SLHiggs^D}(\gammaHiggs^D)=n(n-n_{\gamma})\deg(D).$
  
\item
$d^{D}_{\gamma}=\frac{n(n-n_{\gamma})\deg(D)}{2}.$
\end{enumerate}
\end{lemma}
\begin{proof}
The last two formulae follow from the first two together with the formulae \cite[Eq.(78)]{dC_supportSL} for $\dim(\cM^{D}_{L})$ and $\dim(\cA^{D}_{L})$.

By Proposition \ref{prop fixed relative}, we have $\cM_{\gamma}^{D}\simeq \cM^{D}_{\pi}/G_{\gamma}$ hence $\dim \cM_{\pi}^D =\dim \cM_{\gamma}^{D}$ and the same argument implies $\dim \cA^D_{\pi}=\dim \cA^D_{\gamma}$. If $\gamma=\id$, then $\cM_{\pi}^{D}=\cM^{D}_{L}$ and $\cA^{D}_{\pi}=\cA^{D}_{L}$, in which case we also have $n_{\gamma}=n$ and we are done by \cite[Eq.(78)]{dC_supportSL}. We thus assume $\gamma\neq\id$.

Since $\piHiggs^D$ is a fibre of the fibration $Q_{\gamma}$, we have
  \[
\dim \piHiggs^D = \dim \cM_{n_\gamma,d}^{D_\gamma}(C_\gamma) - \dim \cM_{1,d}^{D}(C).
  \]
Since $\gamma\neq \id$, we have $m_{\gamma}\neq 1$ and $\deg(D_{\gamma})=m_{\gamma}\deg(D)>2g-2$ (even if $D=K_{C}$). By the formula in \cite[Proposition 7.1.(c)]{Nitsure} and Riemann-Roch, we deduce that
\[
\dim \piHiggs^D  = n^{2}_{\gamma}(m_{\gamma}\deg(D))+1-(g+(\deg(D)+1-g))=(nn_{\gamma}-1)\deg(D)
\]
as claimed.

By \cite[Eq.(18)]{MS}, \cite[Eq.(77)]{dC_supportSL} and Riemann-Roch, we have
  \begin{align*}
    \dim \cA^{D}_{\pi} &= \dim \cA^{D_{\gamma}}_{n_{\gamma}}(C_{\gamma}) - \dim H^{0}(C_{\gamma},D_{\gamma})^{G_{\gamma}} \\
                                        &= \frac{n_{\gamma}(n_{\gamma}+1)}{2}m_{\gamma}\deg(D)-n_{\gamma}m_{\gamma}(g-1)-\dim H^{0}(C,D) \\
                       &= \frac{n(n_{\gamma}+1)}{2}\deg(D)-n(g-1)-n(g-1)-(\deg(D)+1-g)\\
    &= \frac{n(n_{\gamma}+1)\deg(D)}{2}-(n-1)(g-1)-\deg(D)
  \end{align*}
which finishes the proof.  
\end{proof}  

Note that, in the case $D=K_{C}$, we have
\[ d_\gamma = \frac{1}{2} \codim_{\SLHiggs}(\gammaHiggs)\]
(reflecting the symplectic geometry of that case) but this does not hold otherwise.

The following table records some dimension formulae, with references to the literature or the above lemma. Note that from the second line on in this table (i.e. for all the moduli spaces related to $\SL_{n}$ rather than $\GL_{n}$-Higgs bundles), the second column is equal to the first column specialised to $\Delta:=\deg(D)=2g-2$; we chose to keep this presentation for ease of reference.\\

\begin{center}
\begin{tabular}{|l|l|l|l|} 
\hline
                  & $\Delta:=\deg(D) > 2g-2$ & $D = K_C$ & Reference \\ \hline \hline
$\dim \GLHiggs^D$ &  $n^{2}\Delta+1$                      &      $n^{2}(2g-2)+2$               &  \cite[Prop. 7.1]{Nitsure}             \\ \hline
$\dim \cA^{D}$                  &    $\frac{n(n+1)\Delta}{2}-n(g-1)$                      &  $n^{2}(g-1)+1 $              &  \cite[Eq.(77),\S 6.1]{dC_supportSL}             \\ \hline

  $\dim \SLHiggs^{D}$                  &   $(n^{2}-1)\Delta$                &   $(n^{2}-1)(2g-2)$                  & \cite[Eq.(78), \S 6.1]{dC_supportSL}              \\ \hline
$\dim \cA^{D}_{L}$                  &      $\frac{n(n+1)\Delta}{2}-(n-1)(g-1)-\Delta $                   & $(n^{2}-1)(g-1)$               & \cite[Eq.(78), \S 6.1]{dC_supportSL}              \\ \hline
$\dim \cM^{D}_{\pi}=\dim \cM^{D}_{\gamma}$                  &   $(nn_{\gamma}-1)\Delta$                            &   $(nn_{\gamma}-1)(2g-2)$             &       Lemma \ref{lemma dim form} i)          \\ \hline
$\dim\cA^{D}_{\pi}=\dim \cA^{D}_{\gamma}$                  &  $\frac{n(n_{\gamma}+1)\Delta}{2}-(n-1)(g-1)-\Delta$                                & $(nn_{\gamma}-1)(g-1)$               &      Lemma \ref{lemma dim form} ii)         \\ \hline
$\codim_{\SLHiggs^{D}}\cM^{D}_{\gamma}$                 &   $n(n-n_{\gamma})\Delta$                                       & $n(n-n_{\gamma})(2g-2)$                & Lemma \ref{lemma dim form} iii)                \\ \hline
$d_{\gamma}:=\codim_{\cA^D}\cA^D_{\gamma}$                   &  $\frac{n(n-n_{\gamma})\Delta}{2}$    &                   $n(n-n_{\gamma})(g-1)$              & Lemma \ref{lemma dim form} iv)     \\ \hline
\end{tabular}\newline
\end{center}

To relate the motives of the fixed loci of the $\Gamma$-action and the relative Higgs moduli spaces, we need the following standard result.

\begin{lemma}\label{lemma motive quotient}
For an arbitrary field $k$, let $G$ be a finite group acting on a $k$-variety $Y$ such that the action admits a geometric quotient  $f:Y\to X$. For $M\in \DM(X,\Lambda)$, there is an induced $G$-action on $p_{*}p^{*}M$ and the unit map $M\to p_{*}p^{*}M$ factors via an isomorphism $M\simeq (p_{*}p^{*}M)^{G}$.
\end{lemma}
\begin{proof}
This is a special case of \cite[Corollaire 2.1.166]{Ayoub_these_1} which applies as, in the terminology of \textit{loc.\ cit.}, $\DM(-,\Lambda)$ is $\QQ$-linear and separated (in characteristic $0$ this last condition follows from étale descent by \cite[Proposition 2.1.162]{Ayoub_these_1}, but in fact holds for any $k$ \cite[Theoreme 3.9]{Ayoub_etale}). 
\end{proof}  

\begin{cor}\label{cor fixed locus and rel moduli}
The object $q^D_{\gamma_*}h_{\pi_*}^D\one$ admits a $G_\pi$-action such that there is an isomorphism
\[
(q^D_{\gamma_*}(h_{\pi_*}^D\one))^{G_{\pi}}\simeq h^D_{\gamma_*}\one
\]
in $\DM(\gammaHit^D,\Lambda)$. Consequently, we have the following isomorphism in $\DM(k,\Lambda)$
\[ M(\piHiggs^D)^{G_\pi}\simeq M(\gammaHiggs^D).\]
Moreover, this isomorphism is $\Gamma$-equivariant.
\end{cor}
\begin{proof}
The first statement follows from Proposition \ref{prop fixed relative} and Lemma \ref{lemma motive quotient} applied to $p^D_{\gamma}$, as well as the fact that $h_\pi^D$ is $G_{\pi}$-equivariant:
  \[
(q_{\gamma_*}(h^{D}_{\pi_{*}}\one))^{G_{\pi}}\simeq (h^D_{\gamma_*}p_{\gamma_*}\one)^{G_{\pi}}\simeq h^D_{\gamma_*}(p_{\gamma_*}\one)^{G_{\pi}}\simeq h^{D}_{\gamma_*}\one.
\]
The second follows by pushforward to $\Spec(k)$. The $\Gamma$ and $G_{\pi}$-actions commute and the constructions are all $\Gamma$-equivariant, so the resulting isomorphism is $\Gamma$-equivariant.
\end{proof}
  
\subsection{The orbifold motive of the $\PGL$-Higgs moduli space}\label{sec orbifold motive}

We consider the $D$-twisted $\PGL_n$-Higgs moduli space as the orbifold quotient
\[  \PGLHiggs^D = [\SLHiggs^D/\Gamma] \]
which has a natural gerbe $\delta_{L}$ obtained as by descending a $\Gamma$-equivariant $\mu_n$-gerbe on $\SLHiggs^D$ (see \cite[Section 3]{HT}). 

The following definition is just the natural extension of Hausel and Thaddeus's description  \cite{HT} of the stringy E-polynomial of  $\PGLHiggs^D$ to the motivic context (see also \cite{LW}).

\begin{defn}
The orbifold motive of the $D$-twisted $\PGL$-Higgs moduli space $\PGLHiggs^D$ with respect to the gerbe $\delta_{L}$ is defined in $\DM(k,\Lambda)$ as follows
\[ M_{\orb}(\PGLHiggs^D,\delta_{L}) := \bigoplus_{\gamma \in \Gamma} M(\gammaHiggs^D)_{\kappa(\gamma)} \{ d_\gamma  \}  =  M(\PGLHiggs^D) \oplus \bigoplus_{0 \neq \gamma  \in \Gamma} M(\gammaHiggs^D)_{\kappa(\gamma)} \{ d_\gamma \},\]
where $\gammaHiggs^D$ is the $\gamma$-fixed locus in the $D$-twisted $\SL$-Higgs moduli space and $d_\gamma$ is the codimension appearing in Definition \ref{def fixed locus}.
\end{defn}

\section{Motivic mirror symmetry in characteristic zero}\label{sec proof}

In this section, we assume that $k$ is an algebraically closed field of characteristic zero and consider motives with coefficients in $\Lambda = \QQ(\zeta_n)$. We fix $\gamma \in \Gamma$ corresponding to $\kappa = \kappa(\gamma) \in \hGamma$ and cyclic cover $\pi = \pi_\gamma : C_\gamma \ra C$ of degree $m_\gamma$ with $n = n_\gamma m_\gamma$ and Galois group $G_\pi$ as in $\S$\ref{sec Weil pairing}.

Let us outline the structure of this section. In $\S$\ref{overview MS}, we give a short summary of \cite{MS}. In $\S$\ref{sec motivic corr}-\ref{sec beta hat mot}, we will construct the morphism $\beta_{\gamma}^D \in \DM(\SLHit^D,\Lambda)$ whose Betti realisation is map $\beta_{\gamma,\coh}^D$ appearing in \eqref{coh corr MS hat} below, which Maulik and Shen prove is an isomorphism. Since the cohomological construction of $\beta_{\gamma,\coh}^D$ uses cohomological correspondences and vanishing cycles, we will use motivic correspondences and motivic vanishing cycles, which are discussed in $\S$\ref{sec motivic corr} and  Appendix \ref{sec motivic van cycles} respectively. In $\S$\ref{sec beta even}, we use motivic correspondences to construct $\beta_{\gamma}^D$ in the case when $\deg(D) > 2g-2$ and even. In $\S$\ref{sec beta hat mot}, we use motivic vanishing cycles to construct $\beta_{\gamma}^D$ in general by passing from $D + p$ to $D$. To complete the proof, we need to show the pushfoward $\nu_{\gamma}^D$ of $\beta_{\gamma}^D$ to $k$ is an isomorphism. In $\S$\ref{sec Higgs motives ab}, we prove that the motives appearing as the source and target of $\nu_{\gamma}^D$ are abelian in order to conclude that $\nu_{\gamma}^D$ is an isomorphism using a conservativity argument in $\S$\ref{end proof}.

\subsection{Overview of Maulik and Shen's cohomological mirror symmetry}\label{overview MS}

Let us outline the cohomological constructions of \cite{MS}. Maulik and Shen obtain their Hodge-structure theoretic version of mirror symmetry as a sum over all $\gamma \in \Gamma$ of isomorphisms (see \cite[Theorem 0.5]{MS}) of pure Hodge structures
\begin{equation}\label{MS_isotypical_piece_iso}
\nu_{\gamma,\coh}:H^*(\SLHiggs,\CC)_{\kappa} \cong H^{*-2d_\gamma}(\gammaHiggs,\CC)_{\kappa}(- d_\gamma), 
\end{equation}
where $d_\gamma=n(n-n_{\gamma})(g-1)$. Note that the isomorphism $\nu_{\gamma,\coh}$ is not induced by a Gysin morphism, as $d_\gamma$ is only half the codimension of $\gammaHiggs$ in $\SLHiggs$ (see Lemma \ref{lemma dim form}). In fact, $\nu_{\gamma,\coh}$ is constructed from an isomorphism $\beta_{\gamma,\coh}$ relative to the Hitchin base $\SLHit$.

For simplicity, as in the main body of the paper \cite{MS}, let us forget the Hodge structure and concentrate on the isomorphism of cohomology groups; the isomorphism of Hodge structures is then obtained by running the same argument using the theory of mixed Hodge modules.

To construct the isomorphism $\beta_{\gamma,\coh}$, Maulik and Shen use a cohomological correspondence following work of Yun \cite{Yun-global-springer-III}, based on ideas of Ng\^{o} \cite{Ngo}, to construct a morphism
\begin{equation}\label{coh corr MS}
\beta_{\gamma,\coh}^{\mathrm{naive}}: (h_{L})_*\CC_\kappa \ra i_{\gamma *}(h_{\gamma*}\CC)_\kappa\{-d_{\gamma}\}\in D^b_c(\SLHit)
\end{equation} 
where $\{i\}:=(i)[2i]$ for $i\in\ZZ$ denotes a pure Tate twist. This morphism is an isomorphism when restricted to a dense open of $\SLHit$ by an explicit computation of Yun \cite{Yun-global-springer-III}. Note that Yun works in a much more general setting (Higgs bundles for an arbitrary reductive group $G$ and its Langlands dual), interprets this isomorphism conceptually in terms of endoscopic groups of $G$, and already connects the picture to mirror symmetry. However, it is not clear whether this naive morphism is an isomorphism over the full Hitchin base, and a key innovation of Maulik and Shen is to construct a variant that they can prove is an isomorphism over all of $\SLHit$.

For that purpose, they first consider versions of the above constructions for $D$-twisted Higgs bundles, where $D$ is a divisor with $\deg(D) > 2g-2$ and $\deg(D)$ is even; the corresponding moduli spaces and morphisms are denoted with a superscript $D$. This leads to a morphism
\begin{equation}\label{coh corr MS D}
\beta_{\gamma,\coh}^D : ((h_L^D)_*\CC)_\kappa \ra i^D_{\gamma *}(h^D_{\gamma*}\CC)_\kappa\{-d_{\gamma}^{D}\} \in D^b_c(\SLHit^D).
\end{equation} 
The advantage of working with a divisor of degree $>2g-2$ is that the geometry of the $D$-twisted Hitchin fibration $h^D: \GLHiggs^D \ra \GLHit^D$ 
is simpler, in the sense that the supports of in the decomposition theorem for $h^{D}$ are known by work of de Cataldo \cite{dC_supportSL}; however, $\dim \GLHiggs^D \neq 2 \dim \GLHit^D$ and the $D$-twisted Higgs moduli space $\GLHiggs^D$ is no longer algebraic symplectic but rather admits an algebraic Poisson structure. The morphism $\beta_{\gamma,\coh}^D$ is again generically an isomorphism by \cite{Yun-global-springer-III}. Maulik and Shen use the decomposition theorem together with an analysis of the supports for the $\pi$-relative Hitchin maps $h_{\pi}^D : \piHiggs^D \ra \piHit^D$ over the full Hitchin base, based on work of Chaudouard-Laumon \cite{CL} in the $\GL_{n}$-case and of de Cataldo \cite{dC_supportSL} in the $\SL_{n}$-case, to prove that $\beta_{\gamma,\coh}^D$ is actually an isomorphism.

Maulik and Shen then use vanishing cycles to treat the case of odd degree divisors with $\deg(D)>2g-2$ as well as the original case of classical Higgs bundles $D=K_{C}$. Let us concentrate on the case of odd degree divisors with $\deg(D)>2g-2$; the case $D=K_{C}$ is obtained by iterating the construction twice. For $p\in C$, the divisor $D+p$ has even degree $>2g-2$, so there is an isomorphism $\beta_{\gamma,\coh}^{D+p}$. Maulik and Shen construct a function $\mu_{\cA}:\SLHit^{D+p}\to \AA^{1}$ using restriction to $p$ and the Killing form on $\mathfrak{sl}_n$ such that $\SLHit^{D}\subset \SLHit^{D+p}$ is the critical locus of $\mu_{\cA}$ \cite[Theorem 4.5]{MS}. They show the vanishing cycles functor 
$\phi_{\mu_{\cA}}:D^b_c(\SLHit^{D+p})\to D^b_c(\SLHit^D)$
satisfies
\[
\phi_{\mu_{\cA}}((h_L^{D+p})_*\CC)\simeq (h_L^{D})_*\CC(a)[b]
\]
for a Tate twist $a$ and shift $b$ whose precise values do not matter; this requires working with vanishing cycles on certain simple Artin stacks. They prove a similar statement relating the $\pi$-relative Hitchin fibrations $h_{\pi}^{D+p}$ and $h_{\pi}^{D}$, which then allows them to construct an isomorphism
\begin{equation}\label{coh corr MS hat}
\beta_{\gamma,\coh}^D : ((h_L^D)_*\CC)_\kappa \simeq i^D_{\gamma *}(h_{\gamma*}^D \CC)_\kappa\{-d^D_\gamma\} \in D^b_c(\SLHit^D)
\end{equation} 
from $\beta_{\gamma,\coh}^{D+p}$ and finish their proof \cite{MS}.

\subsection{Motivic correspondences}\label{sec motivic corr}

In this section, we discuss motivic correspondences, which are lifts to motivic sheaves of the cohomological correspondences of \cite[Exposé III \S3]{SGA5}. We follow the presentation of \cite[Appendix A]{Yun-global-springer-pub}, since this is the source used by \cite{MS} and we wish to lift their results to motives; however, we use a covariant convention like in \cite{SGA5} instead of the contravariant one in \cite{Yun-global-springer-pub} which seems unnecessarily confusing.

Note that the results of this subsection apply more generally when $k$ is any field and $\Lambda$ is any $\QQ$-algebra (and indeed much more generally to other types of motivic sheaves). We sometimes denote $\DM(X,\Lambda)$ by $\DM(X)$.

\begin{defn}
For a commutative (but not necessarily cartesian) diagram of finite type separated $k$-schemes\footnote{For us it suffices to work with finite type separated schemes, but the construction works in greater generality.}
\begin{equation}\label{diagram-corr}
\xymatrix{
  & Z \ar[ld]_{p} \ar[rd]^{q} & \\
X \ar[rd]_{f} & & Y \ar[ld]^{g} \\
& S &
}
\end{equation}
such that the induced morphism $r:=(p,q):Z\to X\times_{S}Y$ is proper, we define a \emph{motivic correspondence supported on $Z$} from $M\in \DM(X,\Lambda)$ to $N\in \DM(Y,\Lambda)$ to be a morphism
  \[
\zeta: p^*M\to q^!N
\]
in $\DM(Z,\Lambda)$.
\end{defn}

Motivic correspondences can be pushed forward to morphisms in $\DM(S,\Lambda)$ as follows.

\begin{constr}\label{push corr}
Given a motivic correspondence $\zeta:p^{*}M\to q^!N$ as above, we will construct a morphism 
\[
\zeta_\sharp:f_!M\to g_*N
\]
in $\DM(S,\Lambda)$. For this, we first associate to the motivic correspondence $\zeta$ supported on $Z$ a motivic  correspondence $\zeta_{X \times_S Y}$ supported on $X\times_SY$ via the following map of morphism groups
\begin{eqnarray*}
  \DM(Z)(p^*M,q^!N) \stackrel{r_*}{\to}  \DM(X\times_SY)(r_*p^*M,r_*q^!N) \simeq & \hspace{-0.2cm}  \DM(X\times_S Y)(r_* r^* \pi_X^*M,r_!r^!\pi_Y^!N) \\
                                & \quad \to  \DM(X\times_SY)(\pi_X^*M,\pi_Y^!N)
\end{eqnarray*}
where $\pi_X$ and $\pi_Y$ denote the projections from $X \times_S Y$ to $X$ and $Y$. Here the first map comes from the functoriality of $r_*$, the next isomorphism relies on the properness of $r$ and the final map is given by pre- and post-composition with the unit and counit for the adjunctions $r^*\dashv r_*$ and $r_!\dashv r^!$. Then we define $\zeta_\sharp$ to be the image of the motivic correspondence $\zeta_{X \times_S Y}$ supported on $X\times_SY$ under the following isomorphisms
\[
\DM(X\times_{S}Y)(\pi_X^*M,\pi_Y^!N)\simeq \DM(X)(M,\pi_{X_*}\pi_Y^!N)\simeq \DM(X)(M,f^!g_*N)\simeq \DM(S)(f_!M,g_*N)
\]
where the first and last isomorphism are adjunctions and the middle one is base change.
\end{constr}

The construction above is natural in $Z$ and $\zeta$ in various ways. We only need the following lemma on compatibility with a finite group action, which is a straightforward consequence of the naturality of the construction with respect to isomorphisms.

\begin{lemma}\label{lemm corr equiv}
  In the situation of Diagram \eqref{diagram-corr}, assume that there exists a finite group $H$ acting on the whole diagram (i.e. on the individual schemes and such that the morphisms are equivariant). Let $M$ and $N$ be also be $H$-equivariant objects in $\DM(X,\Lambda)$ and $\DM(Y,\Lambda)$ respectively; this induces $H$-equivariant structures on $p^{*}M$, $q^{!}N$, $f_{!}M$ and $g_{*}N$. If the motivic correspondence
\[
\zeta:p^*M\to q^!N
  \]
  is $H$-equivariant, then the induced morphism
  \[
\zeta_\sharp:f_!M\to g_*N
\]
is also $H$-equivariant.
\end{lemma}  

We note that Construction \ref{push corr} commutes with Betti realisation.

\begin{lemma}\label{lemma betti corr}
Let $\sigma:k\to \CC$ be a complex embedding and suppose $\zeta:p^*M\to q^!N$ is a motivic correspondence supported on $Z$ as above between constructible motives $M$ and $N$. Since the Betti realisation functor $R_{B}$ commutes with $p^*$ and $q^{!}$ on constructible objects, we have an induced cohomological correspondence
  \[
R_{B}(\zeta): p^*R_{B}M\to q^!R_{B}N
\]
and similarly an induced morphism
\[
R_{B}(\zeta_{\sharp}):f_{!}R_{B}M\to g_{*}R_{B}N.
\]
Then we have an equality
\[
R_{B}(\zeta_{\sharp})=(R_{B}\zeta)_{\sharp}
\]
where this right side should be understood as the construction in \cite[\S A.1]{Yun-global-springer-pub}.

Moreover, if $\zeta$ is equivariant in the sense of Lemma \ref{lemm corr equiv}, then the induced equivariant structures on $R_{B}(\zeta_{\sharp})=(R_{B}\zeta)_{\sharp}$ coincide.
\end{lemma}  
\begin{proof}
This follows directly from the fact that the Betti realisation commutes with the six operations on constructible motives \cite[Theorem 3.19]{Ayoub_Betti}.
\end{proof}

Let us explain how, in certain situations, the fundamental class of $Z$ in its rational Chow group (i.e. Borel-Moore rational motivic homology) provides a natural motivic correspondence.

\begin{defn}\label{defn motivic corr for fund class}
  In the situation of Diagram \eqref{diagram-corr}, suppose that $Y$ is smooth of dimension $e$ over $k$ and that the morphism $q$ is equidimensional of dimension $d$. Then $Z$ is a (usually singular) equidimensional variety of dimension $d+e$, so it has a fundamental class
  \[
[Z]\in \CH_{d+e}(Z) \otimes \Lambda.
  \]
For any $i\in\ZZ$, we have isomorphisms
\begin{align*}
\CH_{d+i}(Z) \otimes \Lambda & \simeq \Hom_{\DM(k,\Lambda)}(\one\{d+i\}, (p_Y \circ q)_*(p_Y \circ q)^!\one) \\ 
& \simeq \Hom_{\DM(Z,\Lambda)}(\one_{Z}\{d+i\}, q^!\one_Y\{e\})  \\
& \simeq \Hom_{\DM(Z,\Lambda)}(\one_{Z}\{d+i-e\}, q^!\one_Y)
\end{align*}
where the second isomorphism follows from relative purity for the smooth structure morphism $p_{Y}:Y\to\Spec(k)$. Through these isomorphisms (for $i=e$), the class $[Z]$ induces a motivic correspondence supported on $Z$ from $M=\one_{X}\{d\}$ to $N=\one_{Y}$
\[ [Z] : \one_{Z}\{d \} \ra q^!\one_Y \]
and corresponding morphism $[Z]_{\sharp}: f_! \one_X\{d \} \ra g_* \one_Y$ in $\DM(Z,\Lambda)$.
\end{defn}

\subsection{A motivic endoscopic correspondence for Higgs bundles}
\label{sec beta even}

In $\S$\ref{sec fixed loci and rel Higgs} for $\gamma \in \Gamma$, we constructed various Higgs moduli spaces and Hitchin maps fitting into a commutative diagram
\[ \xymatrix{ \piHiggs^D \ar[r]^{p_\gamma^D} \ar[d]_{h_\pi^D} & \gammaHiggs^D \ar@{^{(}->}[r] \ar[d]^{h_\gamma^D} &  \SLHiggs^D \ar[d]^{h_L^D} \\ 
\piHit^D \ar[r]^{q_\gamma^D} & \gammaHit^D \ar@{^{(}->}[r]^{\iota_\gamma^D} & \SLHit^D. } \]
In this section, we assume that $\deg(D)$ is even with $\deg(D)>2g-2$, and construct a morphism
\[\beta^D_{\gamma} : (h_{L}^D)_*\one \ra (i^D_{\gamma})_*(h^D_{\gamma})_*\one)\{-d^D_\gamma\}\in \DM(\SLHit^D,\Lambda)\]
as an application of the formalism of motivic correspondences. This construction mimics the cohomological construction of Maulik and Shen \cite{MS} which builds on the work of Yun \cite{Yun-global-springer-pub}.

For the construction of this correspondence, Maulik and Shen first construct a line bundle $L'\in \Pic(C)$ from $L$ which satisfies
\begin{equation}
  \label{degree-L-prime}
\deg(L')\equiv \deg(L) \pmod{n}   
\end{equation}
and in particular $\gcd(\deg(L'),n)=\gcd(\deg(L'),n_{\gamma})=1$. Equation \eqref{degree-L-prime} ensures we can tensor by a line bundle to relate moduli spaces with determinant $L$ and $L'$. Consider the $\SL$-Higgs moduli space $\cM_{L'}^D := \cM_{n,L'}^D$ with determinant $L'$ which has Hitchin map $h_{L'}^D : \cM_{L'}^D \ra \cA_{L'}^D$. We can also consider the Hitchin maps $\gamma$-fixed locus and $\pi$-relative moduli spaces for $L'$ which we denote by adding a subscript $L'$ (e.g. $h_{\gamma,L'}^D : \cM_{\gamma,L'}^D \ra \cA_{\gamma,L'}^D$). All the Hitchin bases are independent of the choice of $L$ and so we drop this additional subscript for the Hitchin base. 

To define the correspondence, Maulik and Shen introduce a (singular) variety $\Sigma$ in \cite[\S 3.3]{MS}, which fits into a commutative diagram
\begin{equation}\label{diagram-sigma}
\xymatrix{
  & \Sigma \ar[ld] \ar[rd] & \\
\SLHiggs^D\times_{\SLHit^D}\piHit^D \ar[rd]_{h_{L}^D\times\id_{\piHit^D}} & & \cM^{D}_{\pi,L'} \ar[ld]^{h^{D}_{\pi,L'}} \\
& \piHit^D = \cA^D_{\pi,L'} &
}
\end{equation}
where all the morphisms are proper and both $G_{\pi}$ and $\Gamma$ act on $\Sigma$ in such a way that the diagram is $(G_{\pi}\times \Gamma)$-equivariant. The morphism $\Sigma\to  \cM^{D}_{\pi,L'}$ is equidimensional of relative dimension $d^{D}_{\gamma}$.

We do not need to know anything more about $\Sigma$ for this paper. Let us nevertheless indicate the idea of the construction. The variety $\Sigma$ is the Zariski closure of the graph of a morphism
\[
g_{u}:\SLHiggs^{D,\mathrm{reg}}\times_{\SLHit^D}\piHit^{D,\heartsuit}\to \cM_{\pi,L'}^{D,\mathrm{reg}}\times_{\cA_{\pi}^{D}}\piHit^{D,\heartsuit}
\]
where the additional decorations on the moduli spaces and Hitchin bases denote suitable open subsets. The restriction to these opens allows, as a special case of the BNR correspondence \cite{BNR}, to parametrise Higgs bundles via line bundles on spectral curves; moreover, the generic $\SL_{n}$-spectral curve on the locus $\gammaHit^D\subset \SLHit^D$ is nodal and the generic spectral curve for $\cM_{\pi,L'}^{D,\mathrm{reg}}$ is its smooth normalisation. The morphism $g_{u}$ is then given by pullback of line bundles from the nodal curve to its normalisation. The introduction of the line bundle $L'$ is  necessary to match up the determinants of the associated Higgs bundles. We refer to \cite[\S 3.1-3]{MS} for details (which, again, are immaterial to our argument).

By Definition \ref{defn motivic corr for fund class}, $\Sigma$ induces a $(G_{\pi}\times\Gamma)$-equivariant morphism in $\DM(\piHit^D,\Lambda)$
\[
[\Sigma]_{\sharp}:(q_{\gamma}^D)^*(i_\gamma^D)^*(h_L^D)_*\one \simeq ( h_{L}^D\times\id_{\piHit^D})_*\one \to (h_{\pi,L'}^D)_{*}\one\{-d^{D}_{\gamma},\}
\]
where the first isomorphism follows from proper base change. Since we are working with coefficients in $\Lambda$, we can thus take the $\kappa$-isotypical part for any $\kappa\in\widehat{\Gamma}$ and get a morphism
\[
[\Sigma]_{\sharp,\kappa}:(q_{\gamma}^D)^*((i_\gamma^D)^*(h_L^D)_*\one)_\kappa \to ((h_{\pi,L'}^D)_{*}\one)_{\kappa}\{-d^{D}_{\gamma}\}.
\]
We will first pushforward along $q_{\gamma}^D: \piHit^D \ra \gammaHit^D$ and take $G_{\pi}$-invariants and then we will pushfoward along $i_\gamma^D : \gammaHit^D \ra \SLHit^D$. In the first step, we note that in $\DM(\gammaHit^D,\Lambda)$, we have 
\[
\left((q_{\gamma}^D)_{*}(q_{\gamma}^D)^*((i_\gamma^D)^*(h_L^D)_*\one)_\kappa\right)^{G_{\pi}}\simeq ((i_\gamma^D)^*(h_L^D)_*\one)_\kappa
\]
by Lemma \ref{lemma motive quotient} and we have
\[
\left((q^D_{\gamma})_*((h^{D}_{\pi,L'})_*\one )_\kappa\right)^{G_{\pi}} \simeq ((h_{\gamma,L'}^D)_*\one )_\kappa
\]
by Corollary \ref{cor fixed locus and rel moduli}. Pushing forward these last two isomorphisms along $i_\gamma^D : \gammaHit^D \ra \SLHit^D$ and combining with $[\Sigma]_{\sharp,\kappa}$ as well as an adjunction, we obtain a morphism in $\DM(\SLHit^D,\Lambda)$
\begin{equation}\label{beta prelim}
\xymatrix{ 
  ((h_{L}^D)_*\one)_\kappa  \ar[r] &  (i_\gamma^D)_*(i_\gamma^D)^*((h_{L}^D)_*\one)_\kappa \ar[rrr]^{(i_\gamma^D)_*((q_\gamma^D)_*[\Sigma]_{\sharp,\kappa})^{G_\pi}} & & & i^D_{\gamma *}((h^D_{\gamma,L'})_{*}\one)_\kappa\{-d^D_\gamma\}. }
\end{equation}

The final step is to pass from $L'$ back to $L$ following \cite[\S 3.4]{MS}. Since $k$ is algebraically closed, Equation \eqref{degree-L-prime} implies the existence of a line bundle $N$ such that $L'=L\otimes N^{\otimes n}$ and tensoring with $N$ produces a $(G_{\pi}\times\Gamma)$-equivariant isomorphism
\[
\piHiggs^D:=\cM_{\pi,L}^D\simeq  \cM_{\pi,L'}^D
\]
of $\piHit^D$-schemes. This directly provides a $G_{\pi}$-equivariant isomorphism
\[
(h_{\pi}^{D})_*\one = (h_{\pi,L}^{D})_*\one\simeq (h_{\pi,L'}^{D})_*\one
\]
which by Corollary \ref{cor fixed locus and rel moduli} gives an isomorphism in $\DM(\gammaHit,\Lambda)$
\begin{equation}
  \label{change-L-Lprime}
(h^D_{\gamma})_{*}\one=(h^D_{\gamma,L})_{*}\one\simeq (h^D_{\gamma,L'})_*\one
\end{equation}

\begin{defn}\label{def beta mot large even deg}
For $D$ of even degree with $\deg(D) > 2g-2$, by combining the morphism \eqref{beta prelim} with the isomorphism \eqref{change-L-Lprime} above, we obtain a morphism in $\DM(\SLHit^D,\Lambda)$
\[\beta^D_{\gamma} :  ((h_{L}^D)_*\one)_\kappa \ra (i^D_{\gamma})_*((h^D_{\gamma})_*\one)_\kappa\{-d^D_\gamma\}.\]
\end{defn}

\begin{lemma}\label{lemma betti realisation of beta mot is beta even case}
Assume $D$ has even degree with $\deg(D) > 2g-2$ and fix a complex embedding $\sigma:k\to \CC$. The Betti realisation of $\beta^D_{\gamma}$ is the morphism $\beta^D_{\gamma}$ in \eqref{coh corr MS D}, which is precisely the isomorphism $c^{D}_{\kappa}$ of \cite[Theorem 3.2]{MS}.
\end{lemma}  
\begin{proof}
This follows from Lemma \ref{lemma betti corr} and further applications of \cite[Theorem 3.19]{Ayoub_Betti} on the compatibility of the Betti realisation with the six operations on constructible motives.
\end{proof}

\subsection{Passing from $D + p$ to $D$ with motivic vanishing cycles}\label{sec beta hat mot}

Our goal in this section is to extend the construction of $\beta^D_{\gamma}$ to the case where $\deg(D)$ is odd and the case $D=K_{C}$ by using vanishing cycles to pass from $D + p$ to $D$. In this section, we assume that $D$ is either $K_{C}$ or of degree $>2g-2$, and we fix an additional point $p\in C(k)$. We start by reviewing the relevant geometric constructions from \cite[\S 4.2-3]{MS}.

We need to work with moduli stacks of Higgs bundles rather than moduli spaces\footnote{In \cite[\S 4]{MS} the authors do not distinguish between the moduli stacks and the moduli spaces, but we do in order to spell out some arguments precisely. Moreover, their $D$ (resp. $D-p$) is what we call $D+p$ (resp. $D$).}, as in Diagram \eqref{diag sl} below, the map of stacks $\mathrm{ev}_p$ is smooth and we use smooth base change for vanishing cycles (Proposition \ref{prop SBC van cycles}). We write $\fM^{D}_{L}$ (resp. $\fM^{D}_{\pi}$) for the stack of stable $D$-twisted (resp. $\pi$-relative) Higgs bundles of rank $n$ and determinant $L$. These are smooth Deligne-Mumford stacks; moreover, the natural morphisms $\delta^{D}_L:\fM^{D}_{L}\to \SLHiggs^D$ and $\delta^{D}_{\gamma}:\fM^{D}_{\pi}\to \piHiggs$ are $\mu_{n}$-gerbes (the gerbe $\delta^D_L$ was used in $\S$\ref{sec orbifold motive}).

To pass from $D + p$ to $D$, we restrict to Higgs bundles on $p$.  Since $\SL_n$-Higgs bundles on a point up to isomorphism correspond to trace-free matrices up to conjugation, the stack of $\SL_n$-Higgs bundles on $p$ is $\fM_{\SL_n}(p) \simeq \left[\mathfrak{sl}_{n}/\SL_{n}\right]$ with good moduli space $\m_{\SL_n}(p) \simeq \mathfrak{sl}_{n}\sslash \SL_{n}$ which is isomorphic to the Hitchin base $\cA_{\SL_n}(p) \simeq \mathfrak{t}_n \sslash S_n$, where $\mathfrak{t}_n \subset \mathfrak{sl}_{n}$ is a Cartan subalgebra.

For $\fM^{D}_{L}$, Maulik and Shen construct a commutative diagram
\begin{equation}
\label{diag sl}
\begin{gathered}
  \xymatrix{
  \fM^{D}_{L} \ar@{^{(}->}[r]^{\iota_{\fM}} \ar[d]_{\delta^{D}_L} & \fM^{D+p}_{L} \ar[r]^{\ev_{p} \quad \quad \quad} \ar[d]^{\delta^{D+p}_L} \ar@/^3.5pc/[rrrd]_{\hspace{4.4cm} \mu_{\fM} \quad} & \fM_{\SL_n}(p)\simeq \left[\mathfrak{sl}_{n}/\SL_{n}\right] \ar[rrd]^{\widehat{\mu}} \ar[d] & \\
    \SLHiggs^{D} \ar@{^{(}->}[r]^{\iota_{\m}} \ar[d]_{h_L^{D}} & \SLHiggs^{D+p} \ar[r] \ar[d]^{h_L^{D+p}} \ar@/_1.4pc/[rrr]_{\hspace{2.5cm}\mu_{\m}}  & \m_{\SL_n}(p)\simeq \mathfrak{sl}_{n}\sslash \SL_{n} \ar[d]_{\wr} \ar[rr]^{\quad \quad \quad \mu} & & \AA^{1} \\
  \SLHit^{D} \ar@{^{(}->}[r]^{\iota_{\cA}} & \SLHit^{D+p} \ar[r]   \ar@/_3.5pc/[rrru]^{\hspace{4.0cm}\mu_{\cA}}  & \cA_{\SL_n}(p)\simeq \mathfrak{t}_{n}\sslash S_n  & 
}
\end{gathered}
\end{equation}
which roughly speaking relates the difference between $D$-twisted and $(D+p)$-twisted Higgs bundles with the restriction map to Higgs bundles on the point $p$. We do not review the full construction, but record the following properties which are used below.
\begin{enumerate}[label=(\roman*)]
\item The morphism $\ev_{p}$ is smooth \cite[Proposition 4.1]{MS}.  
\item The function $\mu$ is induced by the $\SL_{n}$-equivariant quadratic form
  \[
\mathfrak{sl}_{n}\to \AA^{1}, \quad  g\mapsto \mathrm{Tr}(g^{2})
\]
(see \cite[Equation (98)]{MS}).
\item The closed embedding $\iota_{\fM}$ is the critical locus of the function $\mu_{\fM}$ by \cite[Theorem 4.5(a)]{MS}.
\item The codimension $c$ of $\iota_{\fM}$ is equal to $\dim\mathfrak{sl}_{n} =n^{2}-1$ by \cite[Eq.(78) in \S 6.1]{dC_supportSL}.
\end{enumerate}

For the $\pi$-relative moduli stack $\fM^{D}_{\pi}$, we let $H_{\pi}$ be the appropriate subgroup of $\SL_{n}$ which parametrises automorphisms of $\pi$-relative Higgs bundles over the point $p$ (see \cite[$\S$4.2]{MS}) and $\fh_\pi$ denote its Lie algebra. Then Maulik and Shen show there is a similar diagram
\begin{equation}
\label{diag rel}
\begin{gathered}
  \xymatrix{
  \fM^{D}_{\pi} \ar@{^{(}->}[r]^{\iota_{\fM,\pi}} \ar[d]_{\delta^{D}_\pi} & \fM^{D+p}_{\pi} \ar[r]^{\ev_{p,\pi} \quad \quad} \ar[d]^{\delta^{D+p}_\pi} \ar@/^3.5pc/[rrrd]_{\hspace{4.4cm} \mu_{\fM,\pi} \quad} & \fM_{\pi}(p)\simeq \left[\mathfrak{h}_{\pi}/H_{\pi}\right] \ar[rrd]^{\widehat{\mu}_\pi} \ar[d] & \\
    \piHiggs^{D} \ar@{^{(}->}[r]^{\iota_{\m,\pi}} \ar[d]_{h_\pi^{D}} & \piHiggs^{D+p} \ar[r] \ar[d]^{h_\pi^{D+p}} \ar@/_1.4pc/[rrr]_{\hspace{2.5cm}\mu_{\m,\pi}}  & \m_{\pi}(p)\simeq \mathfrak{h}_{\pi}\sslash H_{\pi} \ar[d]_{\wr} \ar[rr]^{\quad \quad \quad \mu_\pi} & & \AA^{1} \\
  \piHit^{D} \ar@{^{(}->}[r]^{\iota_{\cA,\pi}} & \piHit^{D+p} \ar[r]   \ar@/_3.5pc/[rrru]^{\hspace{4.0cm}\mu_{\cA,\pi}}  & \cA_{\SL_n}(p)\simeq \mathfrak{t}_{\pi}\sslash W_{\pi}   & 
}
\end{gathered}
\end{equation}
with the following properties.
\begin{enumerate}[label=(\roman*$^\prime$)]
\item The morphism $\ev_{p,\pi}$ is smooth by \cite[Proposition 4.1]{MS}. 
\item The function $\mu_{\fM,\pi}$ is induced by the $\SL_{n}$-equivariant quadratic form
  \[
\mathfrak{h}_{\pi}\to \AA^{1}, \quad g\mapsto \mathrm{Tr}(g^{2})
\]
(see \cite[Equation (98)]{MS}).
\item The closed embedding $\iota_{\fM,\pi}$ is the critical locus of the function $\mu_{\fM,\pi}$ \cite[Theorem 4.5(a)]{MS}.
\item The codimension $c_{\pi}$ of $\iota_{\fM,\pi}$ is equal to $\dim \mathfrak{h}_{\pi} = n_{\gamma}n-1$ by Lemma \ref{lemma dim form}.
\end{enumerate}

\begin{rmk}
Diagram \eqref{diag sl} is a special case of Diagram \eqref{diag rel} and Properties (i)-(iv) are special cases of (i$^\prime$)-(iv$^\prime$). We present both separately, as this special case is going to be used in combination with the general case in the next section.  
\end{rmk}

Now we use the formalism of motivic nearby cycles of \cite[Chapter 3]{Ayoub_these_2}, extended to motivic vanishing cycles on Artin stacks as in Definition \ref{def mot van cycles} in Appendix \ref{sec motivic van cycles}. Recall for a half integer $r \in \frac{1}{2} \ZZ$, we defined Tate twists $\{r \} := (\lfloor r \rfloor)[2r]$, which are pure if and only if $r \in \ZZ$.

\begin{rmk}\label{rmk van cycles shift}
In \cite{MS}, the vanishing cycle functors are shifted by $[-1]$ in order to use the fact that $\phi_{f}[-1]$ preserves perverse sheaves. We prefer to stick to the conventions in \cite[Chapter 3]{Ayoub_these_2} and not shift by $[-1]$.
\end{rmk}

\begin{thm}\label{thm vanishing cycle isos}
We have the following isomorphisms.
  \begin{enumerate}[label=\emph{(\roman*)}]
  \item $\widetilde{\phi}_{\mu_{\fM}}\one\simeq \iota_{\fM*}\one\{-(c-1)/2\}$ in $\DM(\fM_{L}^{D+p}, \Lambda)$ as $\Gamma$-equivariant objects,
\item  $\widetilde{\phi}_{\mu_{\fM},\pi}\one\simeq \iota_{\fM,\pi*}\one\{-(c_{\pi}-1)/2\}$ in $\DM(\fM_{\pi}^{D+p}, \Lambda)$ as $(\Gamma\times G_\pi)$-equivariant objects,
\item $\widetilde{\phi}_{\mu_{\cA}}((h_{L}^{D+p})_*\one)_\kappa\simeq \iota_{\cA*}((h^{D}_{L})_*\one)_\kappa\{-(c-1)/2\}$ in $\DM(\SLHit^{D+p}, \Lambda)$,
\item $\widetilde{\phi}_{\mu_{\cA,\pi}}((h_{\pi}^{D+p})_*\one)_\kappa \simeq \iota_{\cA,\pi_*}((h^{D}_{\pi})_*\one)_\kappa\{-(c_{\pi}-1)/2\}$ in $\DM(\piHit^{D+p}, \Lambda)$ as $G_\pi$-equivariant object,
\end{enumerate}
\end{thm}
\begin{proof} It suffices to prove Statements (ii) and (iv), as (i) and (iii) are special cases of these.

For Statement (ii), by Theorem \ref{thm mot van cycles for homog quad form} we can compute the motivic vanishing cycles functor for the quadratic form $q: \mathfrak{h}_{\pi}\to \AA^{1}$ given by $g\mapsto \mathrm{Tr}(g^{2})$: we have $\widetilde{\phi}_q(\one) \simeq \iota_{0_*}\one \{ - (c_{\pi}-1)/2 \}$, where $c_{\pi}= \dim \mathfrak{h}_{\pi}$ and $\iota_0: \{ 0 \} = \mathrm{crit}(q) \hookrightarrow\mathfrak{h}_{\pi}$. Since everything is $H_{\pi}$-equivariant, on the stack quotient $ [\mathfrak{h}_{\pi}/H_{\pi}]$ we also have
\begin{equation}\label{eq van cycles quotient stack} \widetilde{\phi}_{\hat{\mu}_\pi}(\one) \simeq \hat{\iota}_{0_*}\one \{ - (c_{\pi}-1)/2 \},
\end{equation}
 where $\hat{\iota}_0: [\{ 0 \}/H_\pi] \hookrightarrow [\mathfrak{h}_{\pi}/H_{\pi}]$. Finally, we have isomorphisms
\[ \widetilde{\phi}_{\mu_{\fM},\pi}\one\simeq \widetilde{\phi}_{\mu_{\fM},\pi} \ev_{p,\pi}^*\one \simeq \ev_{p,\pi}^*\widetilde{\phi}_{\hat{u}_\pi}(\one) \simeq \ev_{p,\pi}^*\hat{\iota}_{0_*} \one \{ - (c_{\pi}-1)/2 \} \simeq \iota_{\fM,\pi*}\one\{-(c_{\pi}-1)/2\}, \]
where the second isomorphism is smooth base change for vanishing cycles (see Proposition \ref{prop SBC van cycles}) for the smooth map $\ev_{p,\pi}$, the third isomorphism is Equation \eqref{eq van cycles quotient stack} and the final isomorphism is smooth base change for stacks (see Theorem \ref{thm Adeel translated to DM} (iv)).

To prove Statement (iv), we first take the pushforward of the isomorphism (ii) along the morphism $\delta^{D+p}_{\pi}$, which is proper since it is a gerbe for a finite group and is also representable by Deligne--Mumford stacks. By proper base change for vanishing cycles (see Proposition \ref{prop PBC van cycles}) and the fact that the top-left square of Diagram \eqref{diag rel} commutes and is $\Gamma$-equivariant, we obtain a $\Gamma$-equivariant isomorphism
\[
\widetilde{\phi}_{\mu_{\cM},\pi}(\delta^{D+p}_{\pi})_{*}\one\simeq \iota_{\cM,\pi*}(\delta^{D}_{\pi})_{*}\one\{-(c_{\pi}-1)/2\}.
\]
By pushing forward this isomorphism along the proper morphism $h^{D+p}_{\pi}$ and using proper base change for vanishing cycles, we obtain isomorphisms
\begin{align*}
 \widetilde{\phi}_{\mu_{\cA,\pi}}(h^{D+p}_{\pi})_{*}(\delta^{D+p}_{\pi})_{*}\one \simeq  (h^{D+p}_{\pi})_*\widetilde{\phi}_{\mu_{\cM},\pi}(\delta^{D+p}_{\pi})_{*}\one\simeq \: & \: (h^{D+p}_{\pi})_*\iota_{\cM,\pi*}(\delta^{D}_{\pi})_{*}\one\{-(c_{\pi}-1)/2\} \\ 
 \simeq \: & \: \iota_{\cA,\pi_*}(h^{D}_{\pi})_*(\delta^{D}_{\pi})_{*}\one\{-(c_{\pi}-1)/2\}
\end{align*}
where the last isomorphism follows from the commutativity of Diagram \eqref{diag rel}. By Lemma \ref{lemma pushforward gerbe} below, we have $(\delta^{D}_{\pi})_*\one \simeq \one$ and similarly for $\delta^{D+p}_{\pi}$; therefore, we obtain
\[ \widetilde{\phi}_{\mu_{\cA,\pi}}(h^{D+p}_{\pi})_{*}\one \simeq \iota_{\cA,\pi_*}(h^{D}_{\pi})_*\one\{-(c_{\pi}-1)/2\}. \]
The whole bottom-left square of Diagram \eqref{diag rel} as well as the maps $\mu_{\cM,\pi}$ and $\mu_{\cA,\pi}$ are $\Gamma$-equivariant, and this implies that all the isomorphisms above commute with $\Gamma$-actions. We can then take the $\kappa$-isotypical component to obtain
\[
(\widetilde{\phi}_{\mu_{\cA,\pi}}(h_{\pi}^{D+p})_*\one)_\kappa\simeq   (\iota_{\cA,\pi_*}(h^{D}_{\pi})_*\one)_\kappa\{-(c_{\pi}-1)/2\},
\]
which we can rewrite as (iv) as $\widetilde{\phi}_{\mu_{\cA,\pi}}$ and $\iota_{\cA,\pi_*}$ commute with $\Gamma$-actions. The whole construction is $G_\pi$-equivariant and we obtain that the resulting isomorphisms are $G_\pi$-equivariant. 
\end{proof}

To complete the proof, we need the following lemma, which in our setting reflects the classical fact that $BG$ has trivial cohomology with rational coefficients when $G$ is a finite group.

\begin{lemma}\label{lemma pushforward gerbe}
Let $G$ be a finite group. 
\begin{enumerate}[label=\emph{(\roman*)}]
\item\label{pushforward torsor} For a morphism of Artin stacks $f : \cY \ra \fX$ which is an (étale) $G$-torsor, the unit natural transformation $\id\to f_{*}f^{*}$ induces a natural isomorphism $\id\simeq (f_{*}f^{*}-)^{G}$.
\item\label{pushforward gerbe} For a morphism of Artin stacks $\delta : \fX \ra \cZ$ which is an (étale) $G$-gerbe, the unit natural transformation $\id\to \delta_*\delta^*$ induces an isomorphism $\one\simeq\delta_*\one$.
\end{enumerate}
\end{lemma}
\begin{proof}
  If $g:\cW\to \fX$ is a étale surjective morphism, then $g^*:\DM(\fX,\Lambda)\to \DM(\cW,\Lambda)$ is conservative; this follows by étale descent in exactly the same way as in the scheme case, because $g$ then admits a section locally for the étale topology. By assumption on $f$, there exists such a $g$ for which $\cY\times_\fX\cW\simeq \cW\times G\to \cW$ is a trivial $G$-torsor. By conservativity and proper base change, this reduces the proof of \emph{\ref{pushforward torsor}} to the case of a trivial $G$-torsor, which is immediate.

For \emph{\ref{pushforward gerbe}}, by the same conservativity argument and proper base change argument (using that $\delta$ is proper and Deligne-Mumford representable, see Theorem \ref{thm Adeel translated to DM} \emph{\ref{dm stack pbc}}), we can reduce to the case where $\delta:\cZ\times_k BG\to \cZ$ is a trivial $G$-gerbe. In other words, we can consider $\cZ\times_k BG$ as the quotient of the trivial $G$-action on $\cZ$, and this gives us a (non-trivial) $G$-torsor $f:\cZ\to \cZ\times_k BG$. By applying \emph{\ref{pushforward torsor}} to $f$, we have an isomorphism
  \[
\delta_*\one\simeq \delta_*(f_*\one)^G.
\]
By definition, the object $\delta_*(f_*\one)^G$ is a direct factor of $\delta_*f_*\one\simeq \one$ in $\DM(\cZ,\Lambda)$ so we have an induced morphism $\theta:\delta_*\one\to \one$. Moreover, it is easy to see using the naturality of adjunctions that the composition $\one\stackrel{\eta}{\ra}\delta_*\one\stackrel{\theta}{\ra} \one$ is the identity. This shows that the direct factor $\delta_*(f_*\one)^G$ of $\one$ is isomorphic to $\one$, and completes the proof.
\end{proof}

\begin{lemma}
  \label{lemma dim comp}
 We have the following dimension formulae
  \begin{enumerate}[label=\emph{(\roman*)}]
\item $c-c_{\pi}=2(d^{D+p}_{\gamma}-d^{D}_{\gamma})$,
\item $\lfloor (c-1)/2\rfloor-\lfloor (c_{\pi}-1)/2\rfloor=d^{D+p}_{\gamma}-d^{D}_{\gamma}$.    
  \end{enumerate}  
\end{lemma}
\begin{proof}
  By Properties (iv) and (iv$^\prime$) from the lists following Diagrams \eqref{diag sl} and \eqref{diag rel}, we have $c-c_{\pi}=n^{2}-n^{2}_{\gamma}m=n(n-n_{\gamma})$ and then (i) follows from Lemma \ref{lemma dim form}. If $n$ is odd, then so is its divisor $n_{\gamma}$; this shows that $c-c_{\pi}$ is always even. The second statement then follows, as if $a,b\in \ZZ$ and $a-b$ is even, then $\lfloor a/2 \rfloor-\lfloor b/2 \rfloor=(a-b)/2$.  
\end{proof}

We now obtain a corollary analogous to \cite[Corollary 4.6]{MS}. For this, note that the morphism $\mu_{\cA,\pi} : \piHit^{D+p} \ra \AA^1$ is invariant under $G_\pi=\Gal(\pi: C_\gamma \ra C)$ and so induces $\mu_{\cA,\gamma} : \gammaHit^{D+p} \ra \AA^1$ which coincides with the restriction of $\mu_{\cA} : \SLHit^{D+p} \ra \AA^1$ to $i_\gamma^{D+p} :\gammaHit^{D+p} \hookrightarrow \SLHit^{D+p}$.

\begin{cor}\label{cor vanishing cycles isos}
There is an isomorphism $\widetilde{\phi}_{\mu_{\cA,\gamma}} ((h_\gamma^{D+p})_*\one)_{\kappa} \simeq (\iota_{\cA,\gamma})_* ((h_\gamma^D)_* \one)_{\kappa} \{-(c_{\pi}-1)/2\}$. 
\end{cor}
\begin{proof}
This follows from Theorem \ref{thm vanishing cycle isos} (including the $G_\pi$-equivariance) and \ref{cor fixed locus and rel moduli} exactly in \cite[Corollary 4.6]{MS}. 
\end{proof}

Before we define $\beta^{D}_{\gamma}$ in the missing cases, we need one final lemma.

\begin{lemma}\label{lemma iso for def beta mot}
There is an isomorphism 
\[ \iota_{\cA}^{*}\widetilde{\phi}_{\mu_{\cA}}(i^{D+p}_{\gamma})_*((h^{D+p}_{\gamma})_*\one)_\kappa \{(c-1)/2-d^{D+p}_{\gamma}\} \simeq (i^{D}_{\gamma})_*((h^{D}_{\gamma})_*\one)_{\kappa}\{-d^{D}_\gamma\}. \]
\end{lemma}
\begin{proof} This follows from the chain of isomorphisms
\begin{align*} 
 &\iota_{\cA}^{*}\widetilde{\phi}_{\mu_{\cA}}(i^{D+p}_{\gamma})_*((h^{D+p}_{\gamma})_*\one)_\kappa \{(c-1)/2-d^{D+p}_{\gamma}\} \\
 &\simeq \iota_{\cA}^{*}(i^{D+p}_{\gamma})_*\widetilde{\phi}_{\mu_{\cA,\gamma}}((h^{D+p}_{\gamma})_*\one)_\kappa \{(c-1)/2-d^{D+p}_{\gamma}\}\\
&\simeq \iota_{\cA}^{*}(i^{D+p}_{\gamma})_*(\iota_{\cA,\gamma})_*((h^{D}_{\gamma})_*\one)_\kappa (\lfloor (c-1)/2\rfloor-\lfloor (c_{\pi}-1)/2\rfloor-d^{D+p}_{\gamma})[c-c_{\pi}-2d^{D+p}_{\gamma}]\\
               &   \simeq (i^{D}_{\gamma})_*((h^{D}_{\gamma})_*\one)_\kappa\{-d^{D}_\gamma\},
\end{align*}
where the first isomorphism is proper base change for vanishing cycles (Proposition \ref{prop PBC van cycles}) as $\mu_{\cA} \circ i_\gamma^{D+p} = \mu_{\cA,\gamma}$, the second isomorphism follows from Corollary \ref{cor vanishing cycles isos} and the final isomorphism follows from Lemma \ref{lemma dim comp} and proper base change $\iota_{\cA}^{*}(i^{D+p}_{\gamma})_* \simeq (i^{D}_{\gamma})_* (\iota_{\cA,\gamma})^*$ for the cartesian square
\[ \xymatrix{  \gammaHit^D \ar[d]_{i_\gamma^D} \ar[r]^{\iota_{\cA,\gamma}} & \gammaHit^{D+p} \ar[d]^{i_\gamma^{D+p}}  \\ \SLHit^D \ar[r]^{\iota_{\cA}} & \SLHit^{D+p} } \]
together with the isomorphism $(\iota_{\cA,\gamma})^*(\iota_{\cA,\gamma})_* \simeq \mathrm{id}$.
\end{proof}

\begin{defn}[Construction of $\beta^{D}_{\gamma}$]\label{def beta mot otherwise}\
\begin{enumerate}
\item For $D$ of even degree, we have defined in $\S$\ref{sec beta even} the morphism
\[ \beta^{D}_{\gamma}: ((h_{L}^{D})_*\one)_\kappa \ra (i^{D}_{\gamma})_*((h^{D}_{\gamma})_*\one)_\kappa\{-d^{D}_\gamma\} \in \DM(\SLHit^{D},\Lambda).\]
\item For $D$ of odd degree greater than $2g-2$, we define $\beta^{D}_{\gamma} \in \DM(\SLHit^D,\Lambda)$ as
\begin{align*}
\beta^{D}_{\gamma}:  ((h_{L}^D)_*\one)_\kappa \stackrel{\sim}{\leftarrow} \iota_{\cA}^*\iota_{\cA_*} ((h_{L}^D)_*\one)_\kappa  \simeq & \: \: \iota_{\cA}^{*}\widetilde{\phi}_{\mu_{\cA}}((h^{D+p}_{L})_*\one)_\kappa \{ (c-1)/2 \} \\
                                            & \stackrel{(\star)}{\longrightarrow}  \iota_{\cA}^{*}\widetilde{\phi}_{\mu_{\cA}}(i^{D+p}_{\gamma})_*((h^{D+p}_{\gamma})_*\one)_\kappa \{(c-1)/2 -d^{D+p}_{\gamma} \} \\        
                                          & \quad \quad \simeq (i^{D}_{\gamma})_*((h^{D}_{\gamma})_*\one)_\kappa \{-d^{D}_\gamma\},
\end{align*}               
where the first map comes from adjunction, the next isomorphism comes from Theorem \ref{thm vanishing cycle isos} (iii), the morphism $(\star)$ is $\iota_{\cA}^{*}\widetilde{\phi}_{\mu_{\cA}}\beta^{D+p}_{\gamma}$ for the map $\beta^{D+p}_{\gamma}$ constructed in (1),  and the final isomorphism is given by Lemma \ref{lemma iso for def beta mot}.
\item For $D = K_C$, we define
\[ \beta_{\gamma}:= \beta_{\gamma}^{K_C}:  ((h_{L})_*\one)_\kappa \ra  (i_{\gamma})_*((h_{\gamma})_*\one)_\kappa\{-d_\gamma\} \text{ in } \DM(\SLHit,\Lambda). \]
by taking a point $p$ on $C$ so that $K_C + p$ is of odd degree greater than $2g -2$ and we have defined $\beta^{K_C + p}_{\gamma}$ in (2) and we then construct $\beta^{K_C}_{\gamma}$ from $\beta^{K_C + p}_{\gamma}$ in exactly the same way that $\beta^{D}_{\gamma}$ is constructed from $\beta^{D+p}_{\gamma}$ in (2).
\end{enumerate}
\end{defn}

\begin{lemma}\label{lemma betti realisation of beta mot is beta}
For $D$ with $D = K_C$ or $\deg(D) > 2g-2$ and $k \hookrightarrow \CC$, the Betti realisation of $\beta^D_{\gamma}$ is the morphism $\beta^D_{\gamma,\coh}$ in \eqref{coh corr MS hat}, which is the isomorphism $c^{D}_{\kappa}$ of \cite[Theorem 3.2]{MS}.
\end{lemma}  
\begin{proof}
Our construction is parallel to the one of \cite{MS}, so this follows from combining \ref{lemma betti realisation of beta mot is beta even case} and Theorem \ref{thm van cycles commute with Betti realisation}, noting that all the motives on stacks we encounter in the proof satisfy the condition of (iii) in Theorem \ref{thm van cycles commute with Betti realisation} by smooth base change.
\end{proof}

\begin{rmk}
The Betti realisation functor $R_{B}:\DM_{c}(S,\Lambda)\to D(S^{\an},\Lambda)$ is conjectured to be conservative \cite{Ayoub_Survey}, so we expect that $\beta^D_{\gamma}$ is an isomorphism. This is however not necessary for the proof of Theorem \ref{main_thm}.  
\end{rmk}

\subsection{Higgs moduli spaces with abelian motives}\label{sec Higgs motives ab}

We have already seen that $M(\SLHiggs^D)$ is abelian (Theorem \ref{thm mot SL Higgs ab}) and to apply a conservativity argument, we also need the following result.

\begin{prop}\label{prop isotypical piece motive fixed locus abelian}
For each $\gamma \in \Gamma$ with corresponding $\kappa = \kappa(\gamma) \in \hGamma$, the $\kappa$-isotypical piece of the motive of the $\gamma$-fixed locus $M(\gammaHiggs^D)_\kappa \in \DM(k,\Lambda)$ is abelian.
\end{prop}
\begin{proof}
The idea is to relate this motive to the motive of the $\pi$-relative moduli space $\piHiggs^D$. We first apply Corollary \ref{cor fixed locus and rel moduli} to get an isomorphism
 \begin{equation}\label{eq gamma fixed rel}
M(\gammaHiggs^D)_\kappa\simeq M(\piHiggs^D)^{G_\pi}_{\kappa}. 
 \end{equation}
In \cite[Proposition 2.10]{MS}, Maulik and Shen construct an isomorphism
  \[
R_{B}(((h^{D}_{\pi})_{*}\one)_{\kappa})\simeq R_{B}(((h^{D}_{\pi})_{*}\one_{})^{\Gamma})
  \]
in $D((\piHit^{D})^{\an},\Lambda)$. Their construction uses the interaction of the action of $\Gamma$ with the connected components of $\piHiggs^{D}$ described in \cite{HP} and is entirely motivic (it relies on averaging and taking isotypical components for actions of some finite abelian groups), and so lifts to an isomorphism
  \[
((h^{D}_{\pi})_{*}\one)_{\kappa}\simeq ((h^{D}_{\pi})_{*}\one_{})^{\Gamma}.
  \]
By pushing this isomorphism forward to $\Spec(k)$ and dualising, we get an isomorphism
  \[
M(\piHiggs^D)_{\kappa}\simeq M(\piHiggs^{D})^{\Gamma}.
\]
Combining this with the isomorphism \eqref{eq gamma fixed rel} above, we get an isomorphism
\[ M(\gammaHiggs^D)_\kappa \simeq M(\piHiggs^D)^{G_\pi \times \Gamma}. \]
We now claim that $M(\piHiggs^D)^\Gamma$ is a direct factor of the motive of $\cM^{D_\gamma}_{n_\gamma,d}(C_\gamma)$. The cohomological counterpart of this is explained in \cite[\S 5.3, Equations (118) and (119)]{MS} and we just adapt the argument (note that our notation is slightly different). Consider the algebraic group
\[
\cM_{1,0}^D\simeq \Pic^0(C)\times H^0(C,\cO_C(D)).
\]
There is a morphism
\[
\cM_{1,0}^D\times \piHiggs^D\to \cM_{n_\gamma,d}^{D_\gamma}(C_\gamma), \quad ((\cE_1,\theta_1),(\cF,\theta))\mapsto (\pi^*\cE_1\otimes\cF,\pi^*\theta_1+\theta).
\]
As explained in \cite{MS} after Equation (118), this morphism factors through the (free) diagonal action of $\Gamma$ on the LHS and gives rise to an isomorphism
\[
(\cM_{1,0}^D\times \piHiggs^D)/\Gamma\simeq \cM_{n_\gamma,d}^{D_\gamma}(C_\gamma).
\]
Moreover, we have $M(\cM_{1,0}^D)\simeq M(\Pic^0(C))$ by $\AA^1$-homotopy invariance and $\Gamma$ acts trivially on the motive $M(\Pic^0(C))$ by \cite[Proof of Theorem 4.2]{FHPL_rank2}. Combining this with the Künneth formula, we get an isomorphism
\[
M(\cM_{n_\gamma,d}^{D_\gamma}(C_\gamma))\simeq M(\Pic^0(C))\otimes M(\piHiggs^D)^\Gamma.
\]
The claim follows as $\one$ is a canonical direct factor of $M(\Pic^0(C))$.

By Theorem \ref{thm motive GL D Higgs gen by C}, the motive of $\cM_{n_\gamma,d}^{D_\gamma}(C_\gamma)$ is abelian and thus we deduce that the direct factor $M(\gammaHiggs^D)_\kappa$ is also abelian.
\end{proof}

\begin{rmk}
We do not know whether the motives of $\gammaHiggs^D$ or $\piHiggs^D$ are abelian (although we suspect they are), but we do not need to know this for the conservativity argument we employ in the proof of Theorem \ref{main_thm}.
\end{rmk}

\subsection{Statement and proof in characteristic zero}
\label{end proof}

By pushing-forward the morphism $\beta^D_{\gamma}$ constructed in Definitions \ref{def beta mot large even deg} and \ref{def beta mot otherwise} to $k$ and dualising, one obtains a morphism in $\DM(k,\Lambda)$ which we denote by
\[ \nu^{D}_{\gamma}: M(\cM^D_\gamma)_\kappa\{d_\gamma\} \ra M(\SLHiggs^D)_{\kappa}. \]

\begin{thm}\label{main_thm_D}
Let $C$ be a smooth projective connected genus $g \geq 2$ curve over an algebraically closed field $k$. Fix a rank $n$ and coprime degree $d$ and $L \in \Pic^d(C)$. If $\mathrm{char}(k) = p >0$, we assume $p \nmid n$. For a divisor $D$ on $C$ with either $D = K_C$ or $\deg(D)>2g-2$, the following statements hold in $\DM(k,\Lambda)$ where $\Lambda := \QQ(\zeta_n)$.
\begin{enumerate}[label=\emph{(\roman*)}]
\item For each $\gamma \in \Gamma$ corresponding to $\kappa:=\kappa(\gamma) \in \hGamma$, we have an isomorphism 
\[  \nu^{D}_{\gamma}:M(\gammaHiggs^D)_\kappa\{d_\gamma\}\stackrel{\sim}{\lra} M(\SLHiggs^D)_{\kappa}. \]
\item By summing these isomorphisms over $\gamma \in \Gamma$, we obtain a motivic mirror symmetry isomorphism
\[ M_{\orb}(\PGLHiggs^D,\delta_{L}) \simeq M(\SLHiggs^D),\]
where $M_{\orb}(\PGLHiggs^D,\delta_{L})$ is the orbifold motive with respect to the gerbe $\delta_{L}$ (see $\S$\ref{sec orbifold motive}).
\end{enumerate}
\end{thm}

\begin{proof}[Proof of Theorem \ref{main_thm_D} (in characteristic zero)]
We first note that the motives in the source and target of $\nu^{D}_{\gamma}$ are abelian: $M(\cM^D_\gamma)_\kappa$ is abelian by Proposition \ref{prop isotypical piece motive fixed locus abelian} above and $M(\SLHiggs^D)_\kappa$ is abelian as it is a direct factor of $M(\SLHiggs^D)$, which is abelian by Theorem \ref{thm mot SL Higgs ab}. 

In order to consider the Betti realisation and apply a conservativity argument, we can assume without loss of generality that $k$ admits an embedding $\sigma:k\hookrightarrow\CC$ by an application of the Lefschetz principle. More precisely, let $k_0$ be a subfield of $k$, finitely generated over $\QQ$, such that $C$ is the base change of a smooth projective curve $C_0$ over $k_0$, and let $\bar{k}_{0}$ be the algebraic closure of $k_{0}$ in $k$. Then $\bar{k}_{0}$ is algebraically closed, admits a complex embedding and the morphism $\nu^D_{\gamma}$ for $C$ is the image of the analoguous morphism $\nu^D_{\gamma}(C_0\times_{k_{0}}\bar{k}_{0})$ in $\DM_c(\bar{k}_0,\Lambda)$ via the base change functor $\DM(\bar{k}_0,\Lambda)\to \DM(k,\Lambda)$ (because all the constructions we have done commute with this base change functor), so it suffices to show $\nu^D_{\gamma}(C_0\times_{k_{0}}\bar{k}_{0})$ is an isomorphism.

Since the Betti realisation associated to $\sigma$ commutes with the six operations (hence sends motivic correspondences to the corresponding cohomological correspondence) and commutes with vanishing cycles \cite{Ayoub_Betti}, the Betti realisation of $\nu^{D}_{\gamma}$ is the pushforward to $k$ of the morphism $\beta_{\gamma,\coh}^D$ constructed by Maulik and Shen (see Lemmas \ref{lemma betti realisation of beta mot is beta even case} and \ref{lemma betti realisation of beta mot is beta} for details). Maulik and Shen showed that $\beta_{\gamma,\coh}^D$ is an isomorphism \cite[Theorem 0.3]{MS}. Since the motives appearing are abelian, the conservativity result of Wildeshaus \cite[Theorem 1.12]{Wildeshaus_Picard} implies that $\nu^{D}_{\gamma}$ is an isomorphism.
\end{proof}

The proof of this result in positive characteristic is completed in $\S$\ref{sec proof pos char}.

\begin{rmk}\label{rmk non-alg closed field}
For a smooth projective geometrically connected curve $C$ over an arbitrary field $k$, we can apply the above result to the base change to $\overline{k}$ to see that motivic mirror symmetry (and also motivic $\chi$-independence) holds after a finite extension $k'/k$. 
\end{rmk}

We end by giving some more concrete consequences of Theorem \ref{main_thm}.

\begin{cor}\label{corr results}
We adopt the notation of Theorem \ref{main_thm} and also fix a smooth $k$-variety $X$.
\begin{enumerate}[label=\emph{(\roman*)}]
\item \label{cor virtual} Let $[X]\in K_{0}(\cM_{\rat}(k,\QQ))$ denote the virtual motivic class of a $k$-variety in the Grothendieck ring of Chow motives, and write $\LL:=[\PP^1]-[\Spec(k)]\in K_{0}(\cM_{\rat}(k,\QQ))$ for the Lefschetz motive. Define the orbifold twisted virtual motive $[\PGLHiggs^{D}]^{\orb}_{\delta_{L}}\in K_{0}(\cM_{\rat}(k,\QQ))$ of $\PGLHiggs^{D}$ as
  \[
[\PGLHiggs^{D}]^{\orb}_{\delta_{L}}:=\sum_{\gamma\in \Gamma}[\gammaHiggs^{D}]_{\kappa(\gamma)}\LL^{d_{\gamma}}.
\]
Then
\[
[\SLHiggs^{D}]=[\PGLHiggs^{D}]^{\orb}_{\delta_{L}}.
\]
\item \label{cor mot coh} Let $\gamma \in \Gamma$ corresponding to $\kappa:=\kappa(\gamma) \in \hGamma$. For $a,b\in \ZZ$, there is an isomorphism of motivic cohomology groups (i.e. higher Chow groups)
  \[
H^a_{\mot}(\SLHiggs^D\times_k X,\Lambda(b))_\kappa \simeq H^{a-2d_\gamma}_{\mot}(\gammaHiggs^D\times_k X,\Lambda(b-d_\gamma))_\kappa.
\]
In particular, taking $a=2b=i\geq 0$, we have an isomorphism
\[
\CH^i(\SLHiggs^D\times_k X,\Lambda)_\kappa \simeq \CH^{i-d_\gamma}(\gammaHiggs^D\times_k X,\Lambda)_\kappa.
\]
These isomorphisms are `functorial in $M(X)$', so that if $X$ and $Y$ are smooth projective, they are natural with respect to the action of correspondences in $\CH^*(X\times_kY)$.
\item \label{cor mot coh orb}For $a,b\in \ZZ$, there is an isomorphism
  \[
H^a_{\mot}(\SLHiggs^D\times_k X,\Lambda(b)) \simeq H^{a-2d_\gamma}_{\mot,\orb}(\PGLHiggs^D\times_k X,\delta_L;\Lambda(b-d_\gamma))_\kappa
\]
where $H^*_{\mot,\orb}(-,\delta_{L};\Lambda(-))$ denotes orbifold twisted motivic cohomology for Deligne-Mumford stacks, with the same properties as in \emph{(i)}.
\item\label{cor k-theory} For $m\geq 0$, there is an isomorphism
  \[
K_{m}(\SLHiggs^D\times_{k}X,\Lambda)\simeq K_{m}(\PGLHiggs^D\times_{k}X,\delta_{L};\Lambda)
\]
of (twisted) algebraic K-theory groups with coefficients in $\Lambda$. If $X$ and $Y$ are smooth projective varieties, this isomorphism is natural with respect to K-theoretic correspondences in $K_{0}(X\times_{k}Y)$.
\end{enumerate}  
\end{cor}

\begin{proof}
  Part \emph{\ref{cor virtual}} follows from Theorem \ref{main_thm} together with the fact that Voevodssky's embedding of $\cM_{\rat}(k,\QQ)$ into $\DM_c(k,\QQ)$ induces an isomorphism on Grothendieck rings which sends the Lefschetz motive $\LL$ onto the Tate motive $\QQ\{1\}$ \cite[Corollary 6.4.3]{Bondarko}.

  Part \emph{\ref{cor mot coh}} follows from the representability of motivic cohomology/higher Chow groups for smooth varieties in $\DM(k,\Lambda)$ \cite[Theorem 19.1]{MVW}. Part \emph{\ref{cor mot coh orb}} is simply obtained by summing the first statement over all $\kappa$ (because this is how we define orbifold twisted motivic cohomology).

  For part \emph{\ref{cor k-theory}}, recall that if $X$ is a smooth $k$-variety, then there is a Chern character isomorphism
  \[
\mathrm{ch}:K_{m}(X,\Lambda)\simeq \bigoplus_{i\in \ZZ}H^{2i-m}_{\mot}(X,\Lambda(i))
\]
see e.g. \cite[Corollary 16.2.21]{Cisinski-Deglise-book}. Moreover, if $Y$ is a smooth $k$-variety with an action of a finite abelian $n$-torsion group $G$, then \cite[Theorem 1]{Vistoli-equivariant-K} provides an isomorphism
\[
v:K_{*}([Y/G],\Lambda)\simeq \bigoplus_{g\in G}K_{*}(Y^{g},\Lambda)^{G}.
\]
To obtain this precise form of the formula from Vistoli's theorem, observe that the formula simplifies when $G$ is abelian and one tensors with a large enough cyclotomic field so that all characters of $G$ become defined. We claim that if $\delta$ is a gerbe on the quotient stack $[Y/G]$, there is a similar isomorphism
\[
\tilde{v}:K_{*}([Y/G],\delta;\Lambda)\simeq \bigoplus_{g\in G}K_{*}(Y^{g},\Lambda)_{\kappa(g)}.
\]
where $\kappa(g)\in\widehat{G}$ is the character of $G$ provided by $\delta$. Unfortunately we do not have a reference for this claim, but it seems likely that Vistoli's argument can be adapted to the twisted case. The analoguous formula for twisted topological K-theory is established in \cite[Theorem 3.11]{lin-twisted-orbifold}.

We thus have a sequence of isomorphisms
\begin{eqnarray*}
  K_{m}(\SLHiggs^D\times_{k}X,\Lambda) & \stackrel{\mathrm{ch}}{\simeq} & \bigoplus_{i\in \ZZ}H^{2i-m}_{\mot}(\SLHiggs^D\times_{k}X,\Lambda(i)) \\
                                     & \simeq & \bigoplus_{i\in \ZZ}\bigoplus_{\kappa\in\widehat{\Gamma}} H^{2i-m}_{\mot}(\SLHiggs^D\times_{k}X,\Lambda(i))_{\kappa} \\
                                     & \simeq & \bigoplus_{i\in \ZZ}\bigoplus_{\kappa\in\widehat{\Gamma}} H^{2(i-d_{\gamma})-m}_{\mot}(\gammaHiggs^{D}\times_{k}X,\Lambda(i-d_{\gamma}))_{\kappa(\gamma)}\\
  & \simeq & \bigoplus_{\kappa\in\widehat{\Gamma}}\bigoplus_{i\in \ZZ} H^{2(i-d_{\gamma})-m}_{\mot}(\gammaHiggs^{D}\times_{k}X,\Lambda(i-d_{\gamma}))_{\kappa(\gamma)}\\
                                     & \stackrel{\mathrm{ch}^{-1}}{\simeq } & \bigoplus_{\gamma\in \Gamma} (K_{m}(\gammaHiggs^{D}\times_{k}X,\Lambda))_{\kappa(\gamma)}\\
                                     & \stackrel{\tilde{v}^{-1}}{\simeq}& K_{m}(\PGLHiggs^D,\delta_{L};\Lambda),                                            \end{eqnarray*}
where the third isomorphism comes from Part \emph{\ref{cor mot coh}}.  
\end{proof}  

\begin{rmk}
The isomorphism of rational algebraic K-theory of part \emph{\ref{cor k-theory}} is not intended to be induced by the conjectural derived equivalence \eqref{der eq}, and it seems unlikely that it extends to an isomorphism of integral algebraic K-theory.
\end{rmk}

\section{Motivic $\chi$-independence in characteristic zero}

Throughout this section, we assume $k$ is a field of characteristic zero such that $C(k)\neq\emptyset$.

\subsection{Correspondences for $\chi$-independence}

In this section, $C$ and $n$ will be fixed and we will primarily be interested in the relationship between the cohomology (and motives) of $D$-twisted $\GL_n$-Higgs bundles of degree $d$ and $d'$ coprime to $n$. Thus, for ease of notation, we let $\cM_d^D:=\cM_{n,d}^D(C)$ denote the moduli space of $D$-twisted $\GL_n$-Higgs bundles on $C$ of rank $n$ and degree $d$. The Hitchin base $\cA^D$ is independent of $d$ and so is the spectral curve $\cC^D/\cA^D$, thus we have Hitchin maps
\[ \xymatrix{ \cM_{d}^D \ar[rd]_{h^D_d} & &  \cM_{d'}^D \ar[ld]^{h^D_{d'}} \\ & \cA^D & }.\]
Let us write $\cP^{e}_{\sm}:=\Pic^{e}(\cC^{D,\sm}/\cA^{D,\sm})$ for the relative Picard scheme of degree $e\in \ZZ$. By the spectral correspondence, over the open locus $\cA^{D,\sm}$ in the Hitchin base where the spectral curve $\cC^{D,\sm} \ra \cA^{D,\sm}$ is smooth, the restrictions of these Hitchin maps are relative Picard schemes
 \[ \xymatrix{ \cP^{e}_{\sm}\simeq (h^D_d)^{-1}(\cA^{D,\sm})  \ar[rd]_{h^{D,\sm}_d} & &  (h^D_{d'})^{-1}(\cA^{D,\sm}) \simeq \cP^{e'}_{\sm} \ar[ld]^{h^{D,\sm}_{d'}} \\ & \cA^{D,\sm} & }\]
 where $e$ and $e'$ are determined from $d$ and $d'$ via the spectral correspondence. In particular, both sides are torsors under the relative Jacobian $\cP^{0}_{\sm}$. 

 Although these relative Picard schemes may not be isomorphic, let us explain how to construct a rational correspondence inducing an isomorphism of their (rational) motives relative to $\cA^{D,\sm}$. Let $\sigma$ be any relative effective $0$-cycle on $\cC^{D,\sm} \ra \cA^{D,\sm}$ of some degree $N>0$ (for instance, $\sigma$ is the graph of an étale multisection). Using a relative Abel-Jacobi map $\mathrm{AJ}:\cC^{\sm}\to \cP^{1}_{\sm}$ and the multiplication map $m_{e'-e}:\cP^{1}_{\sm}\to \cP^{e'-e}_{\sm}$, we define \[\tilde{\sigma}:=\frac{1}{N}(m_{e'-e}\circ \mathrm{AJ})_{*}[\sigma]\in \CH_{\dim{\cA^D}}(\cP^{e'-e}_{\sm}).\]
We can in particular interpret $\tilde{\sigma}$ as a morphism in $\DM(\cA^{D,\sm})$ of the form
\[
\tilde{\sigma}:\one\to M_{\cA^{D,\sm}}(\cP^{e'-e}_{\sm})
\]
We use the addition map $a_{e,e'-e}:\cP^{e}\times \cP^{e'-e}\to \cP^{e'}$ to define a morphism
\[
M_{\cA^{D,\sm}}(\cP^{e}_{\sm})\stackrel{\id\otimes \tilde{\sigma}}{\lra} M_{\cA^{D,\sm}}(\cP^{e}_{\sm}\times_{\cA^{D,\sm}} \cP^{e'-e})\stackrel{M(a_{e,e'-e})}{\lra}M_{\cA^{D,\sm}}(\cP^{e'}_{\sm})
\]
in $\DM(\cA^{D,\sm})$. We then dualise this morphism and use the spectral correspondence to get a morphism
\[
\lambda:(h^{D,\sm}_{d'})_{*}\one \to (h^{D,\sm}_{d})_{*}\one.
\]
\begin{lemma}\label{lemma:chi-iso-smooth}
The morphism $\lambda$ is an isomorphism in $\DM(\cA^{D,\sm},\QQ)$.
\end{lemma}
\begin{proof}
By \cite[Lemma A.6]{AHPL}, it is enough to check this statement after pulling back to a geometric point $\bar{x}$ of $\DM(\cA^{D,\sm})$. By looking at the construction of $\lambda$, we see that $\bar{x}^{*}\lambda$ is constructed in the same way from the effective $0$-cycle $\sigma_{\bar{x}}$ on the spectral curve $\cC^{D}_{\bar{x}}$. Over $\bar{x}$, we can trivialise the torsors $\cP^{e}_{\sm}$ and $\cP^{e'}_{\sm}$ and choose identifications $\cP^{e}_{\bar{x}}\simeq \cP^{0}_{\bar{x}}$ and $\cP^{e'}_{\bar{x}}\simeq \cP^{0}_{\bar{x}}$. Modulo these identifications, the morphism $\lambda_{\bar{x}}$ is identified with the endomorphism of $M(\cP^{0}_{\bar{x}})$ induced by translation by a $0$-cycle $\alpha$, thus $\lambda_{\bar{x}}$ is an isomorphism, with inverse induced by translation by $-\alpha$. This concludes the proof.
\end{proof}  

We can also interpret $\lambda$ as a cycle $\lambda\in \CH_{\dim \cM^{D}_{d}}(\cM_{d'}^{D,\sm} \times_{\cA^{D,\sm}} \cM_{d}^{D,\sm})$. The open immersion $j:\cM_{d'}^{D,\sm} \times_{\cA^{D,\sm}} \cM_{d}^{D,\sm}\to \cM_{d'}^{D} \times_{\cA^{D}} \cM_{d}^{D}$ induces a surjective restriction map 
\[j^{*}:\CH_{\dim \cM^{D}_{d}}(\cM_{d'}^{D} \times_{\cA^{D}} \cM_{d}^{D})\to \CH_{\dim \cM^{D}_{d}}(\cM_{d'}^{D,\sm} \times_{\cA^{D,\sm}} \cM_{d}^{D,\sm})\]
(see \cite[Proposition 1.8]{Fulton}) and we define $\overline{\lambda}$ to be any preimage of $\lambda$ under $j^{*}$. We are in the setup of Definition~\ref{defn motivic corr for fund class} with $e=\dim \cM_{d}^{D}$, $d=\dim \cM_{d}^{D}-\dim \cA^{D}$ and $i=\dim \cA^{D}$, and by the computation there we have
\[
\CH_{\dim \cM_{d}}(\cM_d^{D} \times_{\cA^{D}} \cM_{d'}^{D})\simeq \Hom_{\DM(\cA^{D})}(\one,q^{!}\one),
\]
where $q$ is the projection $\cM_{d'}^{D} \times_{\cA^{D}} \cM_{d}^{D}\to \cM_{d}^{D}$; thus 
we can interpret $\overline{\lambda}$ as a motivic correspondence. By the formalism recalled in Section~\ref{sec motivic corr}, we have an induced morphism
\begin{equation}\label{eq rel chi D indep morphism GL}
\beta^{D}_\chi:=[\overline{\lambda}]_{\sharp}:(h^{D}_{d'})_{*}\one \to (h^{D}_{d})_{*}\one
\end{equation}
in $\DM(\cA^{D},\QQ)$ whose restriction to $\cA^{D,\sm}$ is the above isomorphism $\lambda$ of Lemma \ref{lemma:chi-iso-smooth}.

If $\deg (D) > 2g-2$, we will show in $\S$\ref{sec mot chi GL} that by pushing forward $\beta_\chi^D$ to $k$ and dualising we obtain an isomorphism $\nu_\chi^D$ in $\S$\ref{sec mot chi GL}; we first show the corresponding cohomological version is an isomorphism (see $\S$\ref{sec coh chi GL}) and use conservativity of the Betti realisation on abelian motives.

\subsection{Critical locus description}

In this section, to deal with the case of classical Higgs bundles ($D = K_C$), we describe $\cM_d:=\cM^{K_C}_d$ as the critical locus of a function on $\cM^{K_C + p}_d$ built from evaluating at $p$ and the Killing form similarly to the description for $\SL_n$-Higgs bundles in \cite{MS} used in $\S$\ref{sec beta hat mot}. We will ultimately use this to deduce $\chi$-independence for $\GL_n$-Higgs moduli spaces from $\chi$-independence for $K_C + p$-twisted $\GL_n$-Higgs moduli spaces. As in the $\SL_n$-case, we work on the stack level and let $\fM^{D}_d$ denote the stack of $D$-twisted $\GL_n$-Higgs bundles of degree $d$.  

By fixing a point $p \in C(k) \neq \emptyset$ and taking the residue of a $D+p$-twisted Higgs field at $p$, we obtain a morphism to the stack of $\GL_n$-Higgs bundles over $p$
\[ \ev_p : \fM^{D+p}_d {\lra} [\fg\fl_n/\GL_n]. \]
For $D =K_C$, a $K_C+p$-twisted Higgs bundle is given by $(E,\theta: E \ra E \otimes \omega_C(p))$ and the trace of the Higgs field $\tr(\theta) \in H^0(C,\omega_C(p))$ is a meromorphic $1$-form with poles only allowed at $p$, so by the Residue theorem we have $\res_p (\tr(\theta)) = 0$; therefore, $\ev_p$ has image contained in the substack of trace-free $\GL_n$-Higgs fields over $p$ that is
\[ \ev_p : \fM^{K_C +p}_d {\lra} [\fs\fl_n/\GL_n]. \]
Then, similarly to $\S$\ref{sec beta hat mot}, there is a commutative diargram
\begin{equation}
\begin{gathered}
  \xymatrix{
  \fM^{K_C}_{d} \ar@{^{(}->}[r]^{\iota_{\fM}} \ar[d]_{\delta^{K_C}_d} & \fM^{K_C+p}_{d} \ar[r]^{\ev_{p} } \ar[d]^{\delta^{K_C+p}_d} \ar@/^3.5pc/[rrrd]_{\hspace{4.4cm} \mu_{\fM} \quad} &  \left[\mathfrak{sl}_{n}/\GL_{n}\right] \ar[rrd]^{{\mu}} \ar[d] & \\
    \GLHiggs^{K_C}_d \ar@{^{(}->}[r]^{\iota_{\m}} \ar[d]_{h_d^{K_C}} & \GLHiggs^{K_C+p}_d \ar[r] \ar[d]^{h_d^{K_C+p}} \ar@/_1.4pc/[rrr]_{\hspace{2.5cm}\mu_{\m}}  &  \mathfrak{sl}_{n}\sslash \GL_{n} \ar[d]_{\wr} \ar[rr]^{\:\: \overline{\mu}} & & \AA^{1} \\
  \cA^{K_C} \ar@{^{(}->}[r]^{\iota_{\cA}} & \cA^{K_C+p} \ar[r]   \ar@/_3.5pc/[rrru]^{\hspace{4.0cm}\mu_{\cA}}  &  \mathfrak{t}_{n}\sslash S_n  & 
}
\end{gathered}
\end{equation}
with the following properties.
\begin{enumerate}[label=(\roman*)]
\item The morphism $\ev_{p}$ is smooth ($\GL_n$-analogue of \cite[Proposition 4.1]{MS}).  
\item The function $\mu$ is induced by the $\GL_{n}$-equivariant quadratic form
  \[
\mathfrak{sl}_{n}\to \AA^{1}, \quad  g\mapsto \mathrm{Tr}(g^{2}).
\]
\item The closed embedding $\iota_{\fM}$ is the critical locus of the function $\mu_{\fM}$ ($\GL_n$-analogue of \cite[Theorem 4.5(a)]{MS}).
\item The codimension $c$ of $\iota_{\fM}$ is equal to $\dim\mathfrak{sl}_{n} =n^{2}-1$ (similarly to \cite[Eq.(78) in \S 6.1]{dC_supportSL})
\end{enumerate}

Let us compute the image of the unit $\one$ under the motivic vanishing cycles functor for $\mu_{\fM}$.

\begin{thm}\label{thm vanishing cycle iso chi indep GL}
We have the following isomorphisms.
  \begin{enumerate}[label=\emph{(\roman*)}]
  \item $\widetilde{\phi}_{\mu_{\fM}}\one\simeq \iota_{\fM*}\one\{-(c-1)/2\}$ in $\DM(\fM_{d}^{K_C+p}, \Lambda)$,
\item $\widetilde{\phi}_{\mu_{\cA}}(h_{d}^{D+p})_*\one\simeq \iota_{\cA*}(h^{D}_{d})_*\one\{-(c-1)/2\}$ in $\DM(\cA^{K_C+p}, \Lambda)$,
\end{enumerate}
\end{thm}
\begin{proof} This follows exactly as in the proof of Theorem \ref{thm vanishing cycle isos}.
\end{proof}

We define $\beta^{K_C}_{\chi}$ from the morphism  $\beta^{K_C + p}_{\chi} :(h^{K_C + p}_d)_* \one \ra (h^{K_C + p}_{d'})_* \one  \in \DM(\cA,\QQ)$ constructed in Equation \eqref{eq rel chi D indep morphism GL} using this motivic vanishing cycles analogously to Definition \ref{def beta mot otherwise}.

\begin{defn}
For $d$ and $d'$ both coprime to $n$, we define
\begin{equation}\label{eq rel chi indep morphism GL}
\beta_{\chi}:= \beta^{K_C}_\chi :(h_{d'})_* \one \ra (h_{d})_* \one  \in \DM(\cA,\QQ)
\end{equation}
as the following composition
\begin{align*}
 (h_{d'}^{K_C})_*\one \stackrel{\sim}{\leftarrow} \iota_{\cA}^*\iota_{\cA_*} (h_{d'}^{K_C})_*\one  \simeq & \: \: \iota_{\cA}^{*}\widetilde{\phi}_{\mu_{\cA}}(h^{K_C+p}_{d'})_*\one \{ (c-1)/2 \} \\
                                           & \stackrel{(\star)}{\longrightarrow}  \iota_{\cA}^{*}\widetilde{\phi}_{\mu_{\cA}}(h^{K_C+p}_{d})_*\one \{(c-1)/2 \} \simeq \iota_{\cA}^*\iota_{\cA_*}(h^{K_C}_{d})_*\one \stackrel{\sim}{\rightarrow} (h_{d}^{K_C})_*\one ,
\end{align*}               
where the first map and last maps comes from adjunction, the second and penultimate isomorphisms come from Theorem \ref{thm vanishing cycle iso chi indep GL} and the morphism $(\star)$ is $\iota_{\cA}^{*}\widetilde{\phi}_{\mu_{\cA}}\beta^{K_C+p}_{\chi}\{(c-1)/2 \}$.
\end{defn}

\subsection{Cohomological $\chi$-independence}\label{sec coh chi GL}

We will give a quick proof of a cohomological $\chi$-independence result that is compatible with our motivic constructions. The fact that the Hodge numbers of $\cM_{n,d}$ are independent of $d$ provided $(n,d) = 1$ was established by Mellit \cite{mellit}, Groechenig--Wyss--Ziegler \cite{GWZ} and Yu \cite{Yu_H}. In fact, de Cataldo, Maulik, Shen and Zhang \cite{dCMSZ} show for $(n,d) = 1 = (n,d')$  there is a canonical isomorphism between the cohomology of $\cM_{n,d}$ and $\cM_{n,d'}$ preserving the cup product, tautological classes and the perverse filtration for the Hitchin map. In the non-coprime case,  Kinjo and Koseki \cite{Kinjo-Koseki} show there is an isomorphism for BPS cohomology.

\begin{prop}\label{prop coh chi indep GL}
Let $d$ and $d'$ be coprime to $n$ and $D$ be a divisor on a smooth projective curve $C$ over a field $k$ of characteristic zero. Let $\sigma:k\to \CC$ be a complex embedding.
\begin{enumerate}
\item If $\deg(D) > 2g -2$, then the Betti realisation $R_B(\beta^{D}_{\chi}) : (h^D_{d'})_* \QQ \ra (h^D_{d})_* \QQ$ of the morphism of relative motives \eqref{eq rel chi D indep morphism GL} is an isomorphism relative to the Hitchin base $\cA^D$.
\item If $D = K_C$, the Betti realisation $R_B(\beta_{\chi}) : (h_{d'})_* \QQ \ra (h_{d})_* \QQ$ of the morphism of relative motives \eqref{eq rel chi indep morphism GL} is an isomorphism relative to the Hitchin base $\cA$.
\end{enumerate}
\end{prop}
\begin{proof}
The Hitchin map for $D$ of degree $>2g -2$ has full supports by \cite{CL} and so it suffices to verify the first statement over a dense open subset of the Hitchin base, which follows from Lemma~\ref{lemma:chi-iso-smooth}. The claim for $D=K_{C}$ then follows from the first part together with the commutation of the Betti realisation and vanishing cycles (Theorem \ref{thm van cycles commute with Betti realisation}).
\end{proof}

\subsection{Motivic $\chi$-independence}\label{sec mot chi GL}

By pushing-forward the morphisms $\beta^{D}_{\chi}$ constructed in Equations \eqref{eq rel chi D indep morphism GL} and \eqref{eq rel chi indep morphism GL} to $k$ and dualising, one obtains a morphism in $\DM(k,\QQ)$ which we denote by
\[ \nu^{D}_\chi: M(\cM^D_d) \ra M(\cM^D_{d'}). \]

\begin{thm}\label{main_thm_chi_D}
Let $C$ be a smooth projective curve over a field $k$ with $C(k)\neq\emptyset$ and let $D$ be a divisor such that $D = K_C$ or $\deg(D) > 2g-2$. Then for any degrees $d$ and $d'$ coprime to the rank $n$, the morphism
\[ \nu^{D}_{\chi}: M(\cM^D_d) \ra M(\cM^D_{d'}) \]
is an isomorphism in $\DM(k,\QQ)$.
\end{thm}
\begin{proof}[Proof of Theorem \ref{main_thm_chi_D} (in characteristic zero)]
Since the source and target of $\nu^{D}_{\chi}$ are both abelian motives by Theorem \ref{main thm motives abelian} (i), this follows from Proposition \ref{prop coh chi indep GL} and the conservativity of the Betti realisation on abelian motives over a field of characteristic zero \cite[Theorem 1.12]{Wildeshaus_Picard}, where again we can assume without loss of generality that $k$ admits an embedding $\sigma:k\hookrightarrow\CC$ by an application of the Lefschetz principle.
\end{proof}

The proof of Theorem \ref{main_thm_chi_D} in positive characteristic is given in $\S$\ref{sec proof pos char}.

\section{From characteristic zero to positive characteristic}

In this section, we will build on our main results in characteristic zero to prove analogous statements in positive characteristic by working in a relative setting over the spectrum of a mixed characteristic DVR and using motivic nearby cycles functors to specialise from the generic fibre. In particular, unlike in the rest of the paper, we consider schemes which are not necessarily of finite type over a field. This will require us to discuss semi-projective $\GG_m$-actions in a relative setting in $\S$\ref{sec rel semi-proj Gm} in order to prove some specialisation results using vanishing cycles in $\S$\ref{sec spec van cycles}. In $\S$\ref{sec Higgs family curves}, we summarise various results of Langer that we need about moduli spaces of Higgs bundles in the relative settings and finally in $\S$\ref{sec proof pos char}, we give the proofs of our main theorems in positive characteristic.

\subsection{Semi-projective $\GG_m$-actions in the relative setting}\label{sec rel semi-proj Gm}

A $\GG_m$-action on a smooth quasi-projective $k$-scheme $X$ is \emph{semi-projective} if all zero limits exist and the $\GG_m$-fixed locus is proper. In particular, by work of Bia{\l}ynicki-Birula, the downwards flow as $t \ra 0$ induces a deformation retract onto the fixed locus. Consequently the cohomology (and motive) of $X$ is pure: it is a direct sum of Tate twists of the cohomology of the connected components of the fixed locus. 

In this section, we describe how to extend this to a relative setting. We will only apply this when the base scheme $S$ is the spectrum of a mixed characteristic DVR, but we write this section for a general Noetherian scheme $S$.
 
\begin{defn}
A fibrewise $\GG_{m,S}$-action on a separated Noetherian $S$-scheme $X$ is \emph{semi-proper} if the following two conditions hold.
\begin{enumerate}
\item All zero limits exist: for all $x \in X$ over $s \in S$, the morphism $\GG_{m}  \ra X_s$ given by the action on $x$ extends to $\AA^1$.
\item The $\GG_{m}$-fixed locus is proper over $S$.
\end{enumerate}
If moreover $X/S$ is quasi-projective, we call such an action \emph{semi-projective}.
\end{defn}

For a fibrewise semi-projective $\GG_{m,S}$-action on a smooth quasi-projective $S$-scheme $X$, there is an associated (relative) Bia{\l}ynicki-Birula decomposition which induces a decomposition of the relative motive $M_S(X) \in \DM(S,\ZZ)$. We state this as the following theorem, which puts together results from \cite{Richarz}, \cite{bb-reductive} and \cite{Sumihiro} to give the geometric decomposition and uses the same arguments as over a field (see \cite[Appendix A]{HPL_Higgs}) to prove the motivic decomposition.

\begin{thm}\label{thm rel BB}
For a Noetherian scheme $S$ and a fibrewise  $\GG_{m,S}$-action on a separated finite type $S$-scheme $X$, write the fixed locus $X^{\GG_m} = \sqcup_{i\in I} X_i$ as a disjoint union of connected components. Then there are attracting schemes $X_i^+ / S$ of $X/S$ with monomorphisms $X_i^+ \ra X$ and affine $S$-morphisms $p_{i}:X_i^+ \ra X_i$ (given by flowing under the action to zero) with the following properties.
\begin{enumerate}[label=\emph{\roman*)}]
\item If $X / S$ is smooth, then the fixed loci $X_i$ are smooth over $S$ and $p_{i}:X_i^+ \ra X_i$ are Zariski locally trivial affine space fibrations over $S$. Thus $X_i^+$ is also smooth over $S$.
\item If the $\GG_{m,S}$-action is semi-proper, then the fixed loci are proper over $S$ and $X/S$ is the disjoint union of the (images of the) schemes $X_i^+/S$.
\item If $X / S$ is smooth and quasi-projective and the action is semi-projective, then $X_i^+/S$ are locally closed subschemes of $X/S$ and moreover these subschemes are filterable: there is an ordering $I \simeq \{1,\dots, n\}$ and a filtration of $X/S$ by closed subschemes
  \[
\emptyset=Z_n \subset Z_{n-1}\subset \ldots\subset Z_0=X
\]
such that, for all $1\leq k\leq n$, we have that $Z_{k-1} - Z_{k}=X_{k}^+$.  
\item Under the assumptions of \emph{iii)}, there is a associated relative motivic Bia{\l}ynicki-Birula decomposition 
\[ M_S(X) \bigoplus_{i\in I} M_S(X_i) \{c_i \}  \in \DM(S,\ZZ) \]
where $c_i$ denotes the relative codimension of $X_i^+/S$ in $X/S$. 
\end{enumerate}
\end{thm}
\begin{proof}
Under these assumptions, the $\GG_{m,S}$-action on $X/S$ is \'{e}tale locally linearisable (i.e. there is a $\GG_{m,S}$-equivariant \'{e}tale local cover $X/S$ by affine $S$-schemes) by \cite[Corollary 20.2]{AHR} and so we can apply the results of \cite{Richarz}. The existence of the attractor schemes $X_i^+$ and the affine morphisms $p_{i}:X_i^+ \ra X_i$ follows from \cite[Theorem A]{Richarz} (where, in the notation of loc.\ cit., we denote by $X_{i}^{+}$ the preimage of $X_{i}$ via the map $q^{+}:X^{+}\to X^{0}$). Moreoever, the $S$-morphisms $X_i^+ \ra X$ are monomorphisms as $X/S$ is separated (see \cite[Remark 1.19]{Richarz}). 
  
For i), the smoothness of $X_i/S$ follows by \cite[Theorem A iii)]{Richarz}. We claim that the retraction $p_i :X_i^+ \ra X_i$ is smooth over $S$ by a relative version of the argument in \cite[corollary 7.3]{bb-reductive}: the retraction $p_i :X_i^+ \ra X_i$ is smooth along the section given by the inclusion $X_{i}\subset X_{i}^{+}$ (which can be checked by computing the weight decomposition of the tangent bundle of $X_{i}^{+}$ restricted to $X_{i}$), and the non-smooth locus of $p_i$ would be a closed subscheme of $X_i^+$ which is invariant under $\GG_m$ and taking limits as $t \ra 0$; however, all such limits lie in $X_i$, along which $p_i$ is smooth, thus we conclude $p_i$ must be smooth. Since the $S$-morphism $p_{i}:X_i^+ \ra X_i$ is affine, smooth and given by flowing under the $\GG_{m}$-action, it is a Zariski locally trivial affine space fibration over $S$ by \cite[Lemma 7.2]{bb-reductive}.
  
For ii), the properness of $X_i/S$ follows by assumption of the action being semi-projective. Moreover, as all zero limits exist (and are unique by separatedness of $X/S$), $X/S$ is the disjoint union of the images of the schemes $X_i^+/S$. Note that, a priori, these images are only constructible subsets of $X$.

In case iii), there is a $\GG_{m,S}$-equivariant embedding $X \hookrightarrow \PP_S(\cV)$ into a projective bundle over $S$ by \cite[Theorem 2.5]{Sumihiro} and we can deduce that the attracting schemes in $X$ are locally closed and filterable from the corresponding statement for the attracting schemes in $\PP_S(\cV)$, where the ordering of the weights for the $\GG_m$-action on $\cV$ are used to give this filtration (for example, see \cite[Theorem A.2]{HPL_Higgs}). Then iv) follows exactly as in the case over a field (for example, see \cite[Theorem A.4]{HPL_Higgs}).
\end{proof}

\subsection{Specialisation from the generic fibre to the special fibre}\label{sec spec van cycles}

Let $R$ be a DVR with uniformising parameter $\pi$, fraction field $K = \mathrm{Frac}(R)$ and residue field $k = R/(\pi)$. Write $S := \Spec (R)$. We let $i : \Spec (k) \hookrightarrow S$ (resp. $j : \Spec (K) \hookrightarrow S$) denote the inclusion of the special (resp. generic) point. 

We have the \emph{tame}\footnote{Although we use the same notation, this situation is slightly different to that of Section~\ref{sec beta hat mot} and Appendix~\ref{sec motivic van cycles}, where we were considering nearby cycles on Artin stacks in characteristic $0$.} nearby cycles functor
\[ \psi_{\pi} :\DM(K,\QQ) \ra \DM(k,\QQ) \]
defined in \cite[\S 10 Equation (97)]{Ayoub_etale}. By \cite[Proposition 3.2.9]{Ayoub_these_2}, this functor is part of a specialisation system in the sense of \cite[Chapitre 3]{Ayoub_these_2} and in particular satisfies the same smooth and proper base change properties as tame nearby cycles in the étale setting. 

By \cite[Section 7]{LPLS}, for $M \in \DM(S)$, there is a specialisation map fitting into a distinguished triangle
\[ i^*M  \stackrel{\Sp(M)}{\longrightarrow} \psi_{\pi}j^*M \lra \phi_{\pi}M \]
with $\phi_{\pi}M$ the motivic tame vanishing cycles of $M$, and for $M = \one_S$, we have $\Sp(\one_S) : \one_k = i^* \one_S \ra \psi_{\pi}j^*\one_S = \psi_{\pi} \one_K$ is an isomorphism.

We use $\psi_{\pi}$ to relate statements about motives over $K$ to motives over $k$ as follows.

\begin{prop}\label{prop nearby cycles DVR}
Let $f : X \ra S$ be smooth over $S = \Spec (R)$. 
\begin{enumerate}[label=\emph{\roman*)}]
\item If $f$ is smooth and proper, then the specialisation map $\Sp(M_S(X))$ is an isomorphism.
\item If $f$ is smooth and quasi-projective with a semi-projective $\GG_m$-action, then the specialisation map $\Sp(M_S(X))$ is an isomorphism.
\item If $g : Y \ra S$ is smooth and $\alpha: M_S(X) \ra M_S(Y)$ is a morphism between the relative motives of $X$ and $Y$ over $S$ such that the pull back to the generic fibre $\alpha_K: M_K(X_K) \ra M_K(Y_K)$ is an isomorphism and the specialisation maps $\Sp(M_S(X))$ and $\Sp(M_S(Y))$ are isomorphisms, then $\alpha_k = \psi_{\pi}(\alpha_K)$ is also an isomorphism. 
\item If $\alpha : A \ra B \in \DM(S)$ is a morphism such that $\alpha_K: A_K:=j^*(A) \ra B_K = j^*(B)$ is an isomorphism and the specialisation maps $\Sp(A)$ and $\Sp(B)$ are isomorphisms, then $\alpha_k = \psi_{\pi}(\alpha_K): A_k \ra B_k$ is also an isomorphism. 
\end{enumerate}
\end{prop}
\begin{proof}
Since $f$ is smooth, we have $i^*M_S(X) \simeq M_k(X_k):= (f_k)_! (f_k)^!\one_{K} \simeq (f_k)_* (f_k)^*\one_{K}\{ d \} $, where $d$ is the relative dimension of $f$, and similarly $j^*M_S(X) \simeq M_K(X_K):=(f_K)_! (f_K)^!\one_{K} \simeq (f_K)_* (f_K)^*\one_{K}\{ d \}$. 

For i), by using proper base change (PBC) and smooth base change (SBC) for the specialisation system $\psi_{\pi}$ with respect to $f$, we obtain a square
\[ \xymatrixcolsep{7pc} \xymatrix{ i^*M_S(X) \quad \ar[r]^{\Sp(M_S(X)) \quad} \ar[d]^{\wr}  & \quad \psi_\pi M_K(X_K) \quad \ar[r]^{\mathrm{PBC} \quad}_{\sim \quad} & (f_k)_* \psi_{\pi \circ f} f_K^* \one_K \{ d \} \ar[d]^{\parallel} 
\\ M_k(X_k) \ar[r]^{(f_k)_* (f_k)^* \Sp(\one_S)\{d\} \quad \quad}_{\sim} & (f_k)_*(f_k)^* \psi_\pi \one_K \{ d \} \ar[r]^{\mathrm{SBC}}_{\sim} & (f_k)_* \psi_{\pi \circ f} f_K^* \one_K \{ d \} }\]
which commutes by construction of $\Sp$ \cite[Section 7]{LPLS}, showing that $\Sp(M_S(X)) : M_k(X_k) \ra \psi_\pi M_K(X_K)$ is an isomorphism.

For ii), we use the relative motivic Bia{\l}ynicki-Birula decomposition of Theorem \ref{thm rel BB}
\[ M_{S}(X) \simeq \bigoplus_i M_{S}(X_{i})\{c_i \}  \in \DM(S,\QQ). \]
Since $f_i : X_i \ra S$ are smooth and proper, $\Sp(M_S(X_i))$ is an isomorphism by i). Thus 
\[ \Sp(M_S(X)) = \bigoplus_{i} \Sp(M_S(X_i))\{ c_i \} \]
is also an isomorphism.

For iii), we have a commutative diagram
\[ \xymatrixcolsep{4pc}\xymatrix{ M_k(X_k) \ar[r]^{\sim} \ar[d]^{\alpha_k} & i^*M_S(X) \ar[d]^{i^*\alpha} \ar[r]^{\Sp(M_S(X))}_{\sim}  & \psi_{\pi} j^*M_S(X) \ar[d]^{\psi_{\pi}j^*\alpha} \ar[r]^{\sim} & \psi_{\pi} M_K(X_K) \ar[d]^{\psi_{\pi}(\alpha_K)}_{\wr} \\
M_k(Y_k) \ar[r]^{\sim} & i^*M_S(Y) \ar[r]^{\Sp(M_S(Y))}_{\sim} & \psi_{\pi} j^*M_S(Y) \ar[r]^{\sim} & \psi_{\pi} M_K(Y_K)
} \]
showing that $\alpha_k = \psi_{\pi}(\alpha_K)$ is an isomorphism. The proof for iv) is identical to that of iii).
\end{proof}

\begin{rmk}\label{rmk nearby cycles DVR}
For a finite group $G$ acting fibrewise on a scheme $X/S$, there is an induced $G$-action on $M_S(X)$ and isotypical decomposition $M_S(X) = \oplus_{\kappa} M_S(X)_\kappa$ (possibly after extending coefficients). If $\Sp(M_S(X))= \bigoplus_\kappa \Sp(M_S(X)_\kappa)$ is an isomorphism, then so is $\Sp(M_S(X)_\kappa)$.
\end{rmk}

\subsection{Higgs moduli spaces for families of curves}\label{sec Higgs family curves}

Since the deformation theory of curves is unobstructed (see \cite[Theorem 5.19]{IllusieFGAExplained} and \cite[III 7.3]{SGA1}), we can extend any smooth projective geometrically connected genus $g$ curve $C/k$ to a family of smooth projective geometrically connected genus $g$ curves $\cC/S$ where $S= \Spec(R)$ is the spectrum of any complete DVR with residue field $k$ such that $\cC_k = C$. Furthermore, the deformation theory of line bundles on curves is unobstructed by \cite[Corollary 5.5]{IllusieFGAExplained} and so any line bundle $L \ra C$ can be lifted to a family of line bundles $\cL \ra \cC/S$ such that $\cL_k \ra \cC_k $ is isomorphic to $L \ra C$.

In the situation of a family of smooth projective curves $\cC$ over a base scheme $S$ one can construct quasi-projective relative moduli spaces of Higgs bundles; this was first done by Simpson for a family of smooth projective schemes over a base of finite type over a characteristic zero field and later generalised by Langer to a family of projective finite type schemes over a base of finite type over a universally Japanese ring \cite[Theorem 1.1]{Langer} and then more recently over a Noetherian base \cite[Theorem 1.1]{Langer_new}. For our application, we will only need the case where the base $S$ is the spectrum of a complete DVR and $\cC/S$ is a family of smooth projective curves, but we will state the basic properties in greater generality.

Let $\cC$ be a family of smooth projective geometrically connected genus $g$ curves over a Noetherian base scheme $S$. Then there is a quasi-projective $S$-scheme $\cM(\cC/S):=\cM_{n,d}(\cC/S)$ which is a moduli space of semistable rank $n$ degree $d$ Higgs bundles over $\CC/S$ and this construction is compatible with base change \cite[Theorem 1.1]{Langer_new}. If $n$ and $d$ are coprime and $S$ is reduced and of finite type over a DVR, then $\cM_{n,d}(\cC/S)$ is smooth over $S$ by \cite[Proposition 3.1 (ii)]{dCZ_NAHTcharp}.

If $\cL \ra \cC/S$ is a family of line bundles, then there is a closed subscheme $\cM_{\cL}(\cC/S):=\cM_{n,\cL}(\cC/S) \hookrightarrow \cM_{n,d}(\cC/S)$ of trace-free Higgs fields with determinant $L$; we refer to this as the $\SL_n$-Higgs moduli space for $\cC/S$. If $n$ and $d$ are coprime and $S$ is reduced and of finite type over a DVR, then $\cM_{n,d}(\cC/S)$ is smooth over $S$; this follows by adapting the argument in the $\GL_n$ case given in \cite[Proposition 3.1 (ii)]{dCZ_NAHTcharp} by replacing the relative Jacobian with a relative Prym variety in the $\SL_n$ case.

For a family $\cC/S$, there is a commutative finite group scheme over $S$
\[ \Gamma_S := \Jac(\cC/S)[n]\]
which acts on $\cM_{\cL}(\cC/S)$ by tensorisation. In the applications we are interested in $\Gamma_S$ will be a finite étale group scheme $(\ZZ/n\ZZ)^{2g}_S$ over $S$ (so that in particular $n$ is invertible on $S$)  and so for the following we assume this is the case. For a section $\gamma  : S \ra \Gamma_S$ (i.e. an element in $(\ZZ/n\ZZ)^{2g}$), we denote the $\gamma$-fixed locus by $\cM_\gamma(\cC/S)$, which is a smooth quasi-projective $S$-scheme; formal smoothness can be established by averaging over the linearly reductive group $\langle \gamma \rangle$, as in the much more general result \cite[Theorem 4.3.6]{Romagny}. There is an induced $\Gamma_S$-action on both $\cM_{\cL}(\cC/S)$ and $\cM_\gamma(\cC/S)$. Consequently, the relative motives of these $S$-schemes with coefficients in $\Lambda:= \QQ(\zeta_n)$ admit isotypical decompositions in $\DM(S,\Lambda)$
\[ M_S(\cM_{\cL}(\cC/S)) = \bigoplus_{\kappa \in \widehat{\Gamma}_S} M_S(\cM_{\cL}(\cC/S))_\kappa \quad \text{and} \quad M_S(\cM_{\gamma}(\cC/S)) = \bigoplus_{\kappa \in \widehat{\Gamma}_S} M_S(\cM_{\gamma}(\cC/S))_\kappa  \]
analogously to the decomposition \eqref{eq isotypical mot} over a field.

The $S$-schemes $\cM(\cC/S)$ and $\cM_{\cL}(\cC/S)$ and $\cM_\gamma(\cC/S)$ are all smooth and quasi-projective over $S$ and admit a fibrewise (over $S$) semi-projective $\GG_{m}$-action given by scaling the Higgs field. In the $\GL_n$-case, the fact that this $\GG_m$-action is fibrewise semi-projective is shown in the proof of \cite[Proposition 2.13]{dCZ_ProjCompl}: the $\GG_{m}$-fixed point set is proper over $S$ and all zero limits exist following Langer's Langton-type proof \cite[Theorem 5.1]{Langer}. Since $\cM_{\cL}(\cC/S)$ and $\cM_\gamma(\cC/S)$ are closed subschemes, the same is true for their induced $\GG_{m}$-action.

\subsection{Proofs in positive characteristic}\label{sec proof pos char}

Let $C/k$ be a smooth projective connected curve of genus $g \geq 2$. Fix a rank $n$ such that $p \nmid n$ and a degree $d$ coprime to $n$. 

\begin{proof}[Proof of Theorem \ref{main_thm_D} (in positive characteristic)]

Assume that $k$ is an algebraically closed field of characteristic $p >0$. Fix a complete DVR $R$ with residue field $k$ and fraction field $K$ of characteristic zero and let $S = \Spec (R)$. As in $\S$\ref{sec Higgs family curves}, lift $C/k$ to a family of smooth projective curves $\cC/S$ and lift $L \ra C/k$ (resp. $D \ra C/k$)  to a family of line bundles $\cL \ra \cC/S$ (resp. divisors $\cD \ra \cC/S$). We then consider the $\cD$-twisted $\SL_n$-Higgs moduli space $\cM_{\cL}^{\cD}(\cC/S)$ for the family of curves $\cC/S$. In this case, $\Gamma_S$ is a finite \'{e}tale group scheme over $S$ (as $p \nmid n$). Such a finite \'etale group scheme is equivalent to a representation of the étale fundamental group $\pi_{1}^{\et}(S,\Spec(k))$ on $\Gamma_{k}\simeq (\ZZ/n\ZZ)^{2g}_k$. By \cite[Exposé X Théorème 2.1]{SGA1}, the natural map $\pi_{1}^{\et}(S,\Spec(k))\to \pi_{1}^{\et}(\Spec(k),\Spec(k))=\{e\}$ is an isomorphism, so we have
\[ \Gamma_S \simeq (\ZZ/n\ZZ)^{2g}_S\]
and its character group $\widehat{\Gamma}_S$ is also constant.

For each $\gamma \in \Gamma_S$, we let $\cM^{\cD}_\gamma(\cC/S)$ denote the $\gamma$-fixed locus in $\cM_{\cL}^{\cD}(\cC/S)$, which has an induced $\Gamma$-action. Consequently, working with coefficients in $\Lambda = \QQ(\zeta_n)$, 
the relative motive of $\cM^{\cD}_\gamma(\cC/S)$ in $\DM(S,\Lambda)$ admits an isotypical decomposition for the $\Gamma_S$-action.

Since the generic fibre $\cC_K$ is over a field $K$ of characteristic zero, we know by the proof of Theorem \ref{main_thm_D} in characteristic zero that motivic mirror symmetry holds over the algebraic closure $\cC_{\overline{K}}$. Consequently, these motivic statements hold for a finite extension $K'/K$ (see Remark \ref{rmk non-alg closed field}); that is, for any $\gamma \in \Gamma_{S'} \simeq (\ZZ/n\ZZ)^{2g}_{S'}$, we have an isomorphism
\begin{equation}\label{eq mirror sym K'}
 \nu_{\gamma,K'}^{D_{K'}} :M(\cM^{D_{K'}}_{\gamma}(\cC_{K'}))_{\kappa(\gamma)} \{ d_\gamma \} \stackrel{\sim}{\lra} M(\cM^{D_{K'}}_{\cL_{K'}}(\cC_{K'}))_{\kappa(\gamma)} \in \DM(K',\Lambda).
\end{equation}

We can replace our DVR $R$ with the integral closure $R'$ of $R$ in $K'$; then the residue field of $R'$ is still $k$, as it is finite over the algebraically closed field $k$. Let $\cC' \ra S' :=\Spec(R')$ denote the corresponding family of curves and $\cL' \ra \cC'/S'$ the corresponding family of line bundles obtained by base change along $S' \ra S$. Then $\cM_{\cL'}(\cC'/S')$ and $\cM_\gamma(\cC'/S')$ are quasi-projective $S'$-schemes with fibrewise semi-projective $\GG_m$-actions that commutes with their $\Gamma_{S'}$-actions.

By Proposition \ref{prop nearby cycles DVR} (or more precisely, the version featuring the $\Gamma$-action in Remark \ref{rmk nearby cycles DVR}), we see that by applying the nearby cycles functor $\psi_{\pi} : \DM(K',\Lambda) \ra  \DM(k,\Lambda)$ for a uniformising parameter $\pi$ in $R'$, 
the isomorphism in Equation \eqref{eq mirror sym K'} yields an isomorphism of motives 
\[ \tilde{\nu}_{\gamma,k}^D:=\psi_\pi(\nu_{\gamma,K'}^{D_{K'}}) \in \DM(k,\Lambda),\]
over the algebraically closed field $k$ of positive characteristic, which completed the proof.
\end{proof}

Motivic $\chi$-independence in positive characteristic is proved similarly, so we sketch the proof.

\begin{proof}[Proof of Theorem \ref{main_thm_chi_D} (in positive characteristic)]
As in the above proof, we consider a family of curves $\cC/S$, where $S$ is the spectrum of a complete DVR $R$ with residue field $k$ and fraction field $K$ of characteristic zero, whose special fibre is $C/k$. Since $C(k)\neq\emptyset$ and $C$ is smooth, we can assume that $\cC$ has a section and so in particular $C(K)\neq\emptyset$. We let $\cM_{n,d}^{\cD}(\cC/S)$ denote the $\cD$-twisted $\GL_n$-Higgs moduli space for the family of curves $\cC/S$. By the proof of Theorem \ref{main_thm_chi_D} in characteristic zero, motivic $\chi$-independence holds for $\cC_{K}$; that is, there is an isomorphism
\begin{equation}\label{eq chi ind K GL_n}
\nu^{{D_{K}}}_{\chi,K}  : M(\cM^{D_{K}}_{n,d}) \lra M(\cM^{D_{K}}_{n,d'})  \in \DM(K,\QQ)
\end{equation}
and by applying the nearby cycles functor as above we obtain the corresponding isomorphism $\tilde{\nu}^{D}_{\chi,k} := \psi_\pi(\nu^{{D_{K}}}_{\chi,K})$ over the field $k$ of positive characteristic.
\end{proof}

\begin{rmk}
In these proofs,  $\tilde{\nu}_{k}$ is defined by applying the vanishing cycles functor to an isomorphism $\nu_{K}$ over $K$ rather than giving a geometric construction over $k$ or $S$. However, in this relative setting over the spectrum $S$ of a mixed characteristic DVR, we expect that one should be able to construct relative versions $\nu_{S}$ of the morphisms $\nu_{K}$ appearing above such that the special fibre is obtained by applying the vanishing cycles functor to the generic fibre, so we should have $\nu_{k} = \psi_\pi(\nu_{K})$; the main obstruction to proving this is that we could not find a proof of the properness of the Hitchin fibration in this relative mixed characteristic setting (the properness in characteristic zero goes back to Hitchin and Simpson, whereas \cite{Langer} deals with the case in positive characteristic, but we could not find a mixed characteristic version). 
\end{rmk}

\appendix

\section{Motivic sheaves and motivic vanishing cycles for stacks}\label{sec motivic van cycles}

In this appendix, $\Lambda$ denotes an arbitrary $\QQ$-algebra. We first summarise how to extend $\DM$ to Artin stacks following the approach of Khan \cite[Appendix A]{khan-stacks} and then we extend the construction of motivic nearby and vanishing cycles functors to Artin stacks by following the approach for schemes of Ayoub \cite[Chapitre 3]{Ayoub_these_2}.

\subsection{Extending $\DM$ to Artin stacks}\label{sec motives stacks}

In this section, we review how to extend $\DM(-,\Lambda)$ to Artin stacks by étale descent. For this, we follow the Khan's approach of \cite[Appendix A]{khan-stacks} to extending étale motivic homotopy categories $\SH_{\et}(-)$ to derived Artin stacks using an $\infty$-categorical approach. His construction and results apply just as well to $\DM(-,\Lambda)$: the key inputs are the six operations for schemes and the étale descent property, which are satisfied in both cases. Khan's construction works in two steps, first he uses Nisnevich descent to extend to (derived) algebraic spaces and then \'{e}tale descent to extend to (derived) Artin stacks. For us, all stacks we are interested in are non-derived Artin stacks, and so we will state everything for non-derived Artin stacks. Since we only work with Artin stacks with an atlas given by a scheme, we could strictly speaking bypass the first step. 

To do this extension, it is necessary to use the formalism of $\infty$-categories. In particular, when we invoke categorical notions such as functor, (co)limits, adjunctions, etc. these should be interpreted as $\infty$-categorical. The following theorem summarises the main results of Khan's construction in the setting of \'{e}tale motivic sheaves.

\begin{thm}[Khan]\label{thm Adeel translated to DM}
The formalism of six operations on $\DM(-,\Lambda)$ extends to algebraic spaces. Moreover, the presheaf of $\infty$-categories 
\[ X \ra \DM(X,\Lambda), f \ra f^*\]
on the site of algebraic spaces admits a right Kan extension to the site of Artin stacks with the following properties.
\begin{enumerate}[label=\emph{(\roman*)}]
\item \label{dm stack smooth}\cite[Eq. (A.4)]{khan-stacks} Let $\fX$ be an Artin stack. Let $\mathrm{Lis}_{\fX}$ be the full sub-$(2,1)$-category of the $(2,1)$-category of $\fX$-stacks whose objects are smooth morphisms $ X \ra \fX$ with $X$ a scheme. We then have
  \[
\DM(\fX,\Lambda)\simeq \Lim_{X\in\mathrm{Lis}_{\fX}} \DM(X,\Lambda).
\]
In particular, the collection of functors $(u^*:\DM(\fX,\Lambda)\to \DM(X,\Lambda))_{(u:X\to \fX)\in\mathrm{Lis}_{\fX}}$ is jointly conservative.
\item \label{dm stack atlas}\cite[Eq. (A.3)]{khan-stacks} More precisely, but less canonically, one has the following description in terms of a fixed atlas. If $p : X \ra \fX$ is a smooth surjection from an algebraic space $X$ to an Artin stack $\fX$, then 
\[ \quad \DM(\fX,\Lambda) = \Lim\left( \DM(X,\Lambda) \begin{smallmatrix} \ra \\ \ra \end{smallmatrix} \DM(X \times_{\fX} X,\Lambda)  \begin{smallmatrix} \ra \\ \ra \\ \ra \end{smallmatrix}  \DM(X \times_{\fX} X \times_{\fX} X,\Lambda)   \cdots  \right). \]
\item\label{dm stack monoidal}  \cite[Theorem A.5 (i)]{khan-stacks} For every Artin stack $\fX$, there is a closed symmetric monoidal structure on $\DM(\fX,\Lambda)$ and a pair of adjoint bifunctors $(\otimes, \underline{\Hom})$.
\item\label{dm stack bc proj} \emph{(Adjunctions, projection and base change formulae, \cite[Theorem A.5 (ii-iv)]{khan-stacks})} For any locally of finite type morphism $f : \fX \ra \cY$ between Artin stacks, there are adjunctions
\[f^* :  \DM(\cY,\Lambda) \stackrel{\longrightarrow}{\longleftarrow} \DM(\fX,\Lambda) : f_* \]
and
\[f_! :  \DM(\fX,\Lambda) \stackrel{\longrightarrow}{\longleftarrow} \DM(\cY,\Lambda) : f^! \]
which satisfy the projection formula $f_!(\cF) \otimes \cG \simeq f_!(\cF \otimes f^*(\cG))$ and base change formulae: 
\[ g^*f_! \stackrel{\sim}{\longrightarrow} \tilde{f}_!\tilde{g}^* \quad \quad \text{and} \quad \quad \tilde{g}_*\tilde{f}^! \stackrel{\sim}{\longrightarrow} f^!g_* \]
for any cartesian square
\[  \xymatrix{ \cW  \ar[r]^{\tilde{f}} \ar[d]_{\tilde{g}} & \cZ \ar[d]^{g} \\ \fX \ar[r]^{f} & \cY. }\]
\item\label{dm stack sbc} \emph{(Purity isomorphism and smooth base change, \cite[Theorem A.13]{khan-stacks})} For any smooth morphism $f : \fX \ra \cY$ of pure relative dimension $d$, there is a purity isomorphism $f^! \simeq f^*\{ d \}$ and smooth base change $f^*g_* \simeq \tilde{g}_*\tilde{f}^*$ for a cartesian square as above.
\item\label{dm stack pbc}\emph{(Proper base change, \cite[Theorems A.5 (iv) and A.7]{khan-stacks})} If $f : \fX \ra \cY$ is Deligne--Mumford-representable (i.e. represented by Deligne--Mumford stacks), there is a natural transformation $f_! \ra f_*$, which is an isomorphism if $f$ is proper. For $f$ proper and Deligne--Mumford-representable, there is a proper base change $g^*f_* \simeq \tilde{f}_*\tilde{g}^*$ for a cartesian square as above.
\end{enumerate}
\end{thm}

\begin{ex}
For a finite group scheme $G/S$, the morphism $\delta: BG \ra S$ is not representable, but is Deligne--Mumford-representable and is proper; thus the natural transformation $\delta_! \ra \delta_*$ is an isomorphism. More generally, the same is true for any $G$-gerbe $\delta$.
\end{ex}

The Betti realisation also extends readily to the context of Artin stacks. First, we need to extend its target. Let $X$ be a finite type scheme over $\CC$. We denote by $D(X^{\an},\Lambda)$ the $\infty$-category of sheaves of complexes of $\Lambda$-modules on the topological space $X^{\an}$. Then the assignment $X\mapsto D(X^{\an},\Lambda)$ has the same basic functoriality as $X\mapsto\DM(X,\Lambda)$, and one can follow the approach of \cite[Appendix A]{khan-stacks} to define a category $D(\fX^{\an},\Lambda)$ for every finite type Artin stack $\fX$ over $\CC$. Let $\sigma:k\to \CC$ be a complex embedding. There is a Betti realisation functor
\[
R_B:\DM(X,\Lambda)\to D(X^\an,\Lambda)
\]
defined at the triangulated level in \cite{Ayoub_Betti} but which can be easily refined to an $\infty$-functor, see e.g. \cite[Definition 1.21]{Ayoub-anabel}.

The Betti realisation functor for motives of schemes commutes with pullback by arbitrary morphisms \cite[Theoreme 3.19 A]{Ayoub_Betti}, so in particular by smooth morphisms. By Theorem \ref{thm Adeel translated to DM} \emph{\ref{dm stack smooth}}, this implies that we can extend the Betti realisation functor to a functor
\[
R_{B}:\DM(\fX,\Lambda)\to D(\fX_{\sigma},\Lambda).
\]

\subsection{Motivic nearby and vanishing cycles functors for stacks}

Motivic nearby cycles for the categories $\DM(-,\Lambda)$ of étale motivic sheaves on schemes have been introduced in \cite[Chapitre 3]{Ayoub_these_2} and studied further in \cite{Ayoub_Betti,Ayoub_etale}. The closely related functor of motivic vanishing cycles was not defined in those references, but is not too difficult to construct once we have motivic nearby cycles. Our goal is to extend this to Artin stacks. We could employ Theorem \ref{thm Adeel translated to DM} \emph{\ref{dm stack smooth}} combined with the smooth base change properties for nearby cycles of schemes; we prefer a slightly different (and eventually equivalent) approach with a concrete formula.

Throughout this subsection, we let $S$ be a Noetherian finite dimensional base scheme of characteristic zero; for this paper we only need $S=\Spec(k)$, but the general case is exactly the same. Since we are in characteristic $0$, we only need the ``tame'' version of these constructions, which is the only one considered in \cite[Chapitre 3]{Ayoub_these_2}. Furthermore, as we work with coefficients in a $\QQ$-algebra, we can also use the alternative ``logarithmic'' description of motivic nearby cycles considered in \cite[Section 3.6]{Ayoub_these_2}. For this, let $\Log\in \DM(\GG_{m,S},\Lambda)$ be the (dual) logarithm motive constructed in \cite[Definition 3.6.29]{Ayoub_these_2}, which comes with a unit morphism $\one_{\GG_m,S} \ra \Log$.

\begin{defn}\label{def mot van cycles}
For an Artin stack $\fX$ and regular function $f:\fX\to \AA^1_S$, construct the commutative diagram with cartesian squares:
\[
\xymatrix{
  \fX_\eta \ar[r]^{j_{\fX}} \ar[d]_{f_\eta} & \fX \ar[d]_f &  \fX_0 \ar[l]_{i_{\fX}} \ar[d]^{f_0}\\
  \GG_{m,S} \ar[r]^j & \AA^1_S & S \ar[l]_i}
\]
We define the \emph{unipotent nearby cycles functor} of $f$ as
 \[
 \psi_{f}^{\un}:\DM(\fX_{\eta},\Lambda)\to \DM(\fX_{0},\Lambda), \quad M\mapsto i_{\fX}^*(j_{\fX})_* (M \otimes f^*_\eta  \Log).
 \]
 There is a natural transformation $i_{\fX}^* \ra i_{\fX}^*j_{\fX_*}j^*_{\fX} \ra i_{\fX}^*j_{\fX_*}(j_{\fX}^*(-) \otimes f_\eta^* \Log) =  \psi_{f}^{\un} j_{\fX}^*$ induced by adjunction and the morphism $\one_{\GG_m,S} \ra \Log$ and we define the \emph{unipotent vanishing cycles functor} of $f$ as the cofibre of this natural transformation
  \[
 \phi_{f}^{\un}:\DM(\fX,\Lambda)\to \DM(\fX_{0},\Lambda), \quad M\mapsto \mathrm{cofib} \left(i_{\fX}^*(M) \ra  \psi_{f}^{\un} j_{\fX}^*(M)\right).
 \]
For convenience, we also let $\widetilde{\phi}_{f}^{\un}:= (i_{\fX})_* \circ \phi_{f}^{\un} : \DM(\fX,\Lambda)\to \DM(\fX,\Lambda)$.

To construct the nearby and vanishing cycles functors, for $n >0$, let $p_n : \AA^1_S \ra \AA^1_S$ denote the $n$th power map and let $\fX_n := \fX \times_{f,\AA^1_S,p_n} \AA^1_S$ denote the base change with projection maps $f_n : \fX_n \ra \AA^1_S$ and $e_n : \fX_n \ra \fX$. We similarly construct $f_{\eta,n} : \fX_{\eta,n} \ra \GG_{m,S}$ and $f_{0,n} : \fX_{0,n} \ra S$ by base-change along restrictions of $p_n$ and we have an open immersion $j_n : \fX_{\eta,n} \ra \fX_n$ and closed immersion $i_n : \fX_{0,n} \ra \fX_n$.

The induced map $\fX_{0,n}\to \fX_{0}$ is a closed immersion which induces an isomorphism of reduced schemes, so we have compatible equivalences $\DM(\fX_{0,n},\Lambda)\simeq \DM(\fX_{0},\Lambda)$ which are used implicitely in the following. We define the \emph{nearby cycles functor} of $f$ as
 \[
 \psi_{f}:\DM(\fX_{\eta},\Lambda)\to \DM(\fX_{0},\Lambda), \quad M\mapsto \colim_{n \in (\NN^*,|)} \left(  i_n^*(j_{n})_* (e_{\eta,n}^*M \otimes f^*_{\eta,n}  \Log) \right).
 \]
 
Using the same construction as in the unipotent case, we have a natural transformation $i^{*}_{\fX}\to \psi_{f}j_{\fX}^{*}$ and we define the \emph{vanishing cycles functor} of $f$ as
  \[
\phi_{f}:\DM(\fX,\Lambda)\to \DM(\fX_{0},\Lambda), \quad M\mapsto \mathrm{cofib} \left(i_{\fX}^*(M) \ra  \psi_{f} j_{\fX}^*(M)\right).
 \]
The functors involved in the definition of $\phi_{f}$ all commute with colimits, so that we also have
 \[
\phi_{f}(M)\simeq \colim_{n \in \NN^*} \left( \mathrm{cofib} \left(i_n^*(M) \ra  i_n^*(j_n)_*(e_{\eta,n}^*(j_{\fX}^*M) \otimes f_{\eta,n}^* \Log) \right)\right).
 \]
For convenience, we also write $\widetilde{\phi}_{f}:= (i_{\fX})_* \circ \phi_{f} : \DM(\fX,\Lambda)\to \DM(\fX,\Lambda)$.
\end{defn}

\begin{rmk}
One advantage of the logarithmic definition, compared to the more general \cite[Definition 3.5.6]{Ayoub_these_2}, is that it does not require introducing motives over diagrams of schemes. One can make sense of categories of motives over a diagram of schemes $\infty$-categorically, but we prefer to avoid this additional complication and to use the formalism of \cite{khan-stacks} to directly define nearby and vanishing cycles (both for schemes and stacks) at the level of $\infty$-categories.
\end{rmk}

\begin{prop}[Proper base change for vanishing cycles of stacks]\label{prop PBC van cycles}
Let $g : \cY \ra \fX$ be a proper Deligne--Mumford-representable morphism of stacks over $S$ and $f : \fX \ra \AA^1_S$ be a regular function; consider the commutative diagram with cartesian squares
\begin{equation}\label{eq PBC van cycles}
\xymatrix{
  \cY_\eta \ar[r]^{j_\cY} \ar[d]_{g_\eta} & \cY \ar[d]_g &  \cY_0 \ar[l]_{i_\cY} \ar[d]^{g_0}\\
  \fX_\eta \ar[r]^{j_\fX} \ar[d]_{f_\eta} & \fX \ar[d]_f &  \fX_0 \ar[l]_{i_\fX} \ar[d]^{f_0}\\
  \GG_{m,S} \ar[r]^j & \AA^1_S & S. \ar[l]_i}
\end{equation}
Then there is a natural isomorphism
\[ \phi_f g_* \simeq (g_0)_* \phi_{f \circ g} : \DM(\cY,\Lambda) \ra \DM(\fX_0,\Lambda).\]
\end{prop}
\begin{proof}
Since $g$ is a proper Deligne--Mumford-representable morphism, the natural transformations $g_! \ra g_*$ and $(g_\eta)_! \ra (g_\eta)_*$ and $(g_0)_! \ra (g_0)_*$ are isomorphisms by Theorem \ref{thm Adeel translated to DM} \emph{\ref{dm stack pbc}}. 

Let us first show the corresponding statement for the unipotent vanishing cycles functor. Since the unipotent vanishing cycles functor is defined as a cofibre of the natural transformation $i_{\fX}^* \ra \psi_{f}^{\un} j_{\fX}^*$, it suffices to show there is a natural isomorphism
\[ \psi_f^{\un} (g_\eta)_* \simeq (g_0)_* \psi_{f \circ g}^{\un} : \DM(\cY_\eta) \ra \DM(\fX_0,\Lambda). \]
This natural isomorphism is obtained from the following chain of isomorphisms which are natural in $M \in \DM(\cY_\eta)$:
\begin{align*}
\psi_f^{\un} (g_\eta)_*(M):= \: & i_{\fX}^*(j_{\fX})_* ((g_\eta)_*(M) \otimes f^*_\eta  \Log) \\
 \simeq \: & i_{\fX}^*(j_{\fX})_*(g_\eta)_*(M \otimes g_\eta^* f_\eta^*\Log) \\
 \simeq \: & i_{\fX}^*g_*(j_\cY)_*(M \otimes (f \circ g)_\eta^* \Log) \\
 \simeq \: & (g_0)_*(i_{\cY}^*(j_{\cY})_*(M \otimes (f \circ g)_\eta^* \Log)=: (g_0)_* \psi_{f \circ g}^{\un}(M).
\end{align*}
Here the first isomorphism comes from the projection formula, the second comes from the commutativity of the upper left square in \eqref{eq PBC van cycles} and the final isomorphism comes from proper base change for the proper Deligne--Mumford-representable map $g$ (Theorem \ref{thm Adeel translated to DM} \emph{\ref{dm stack pbc}}).

The corresponding statement for the full vanishing cycles functor and nearby cycles functor then follows as colimits commute with both $(g_{\eta})_{*}\simeq (g_{\eta})_!$ and $(g_{0})_{*}\simeq (g_{0})_{!}$ because they are left adjoints.
\end{proof}

We also need the next result about pulling back vanishing cycles along smooth morphisms.

\begin{prop}[Smooth base change for vanishing cycles of stacks]\label{prop SBC van cycles}
Let $g : \cY \ra \fX$ be a smooth morphism of stacks over $S$ and $f : \fX \ra \AA^1_S$ be a regular function; then using the notation in Diagram \eqref{eq PBC van cycles}, there is a natural isomorphism 
\[g_0^*\phi_f \simeq \phi_{f \circ g} g^* : \DM(\fX,\Lambda) \ra \DM(\cY_0,\Lambda).\]
\end{prop}
\begin{proof}
As in the proof of Proposition \ref{prop PBC van cycles}, this boils down to constructing a natural isomorphism
\[ g_0^* \psi_f^{\un} \simeq  \psi_{f \circ g}^{\un} g_\eta^* : \DM(\fX_\eta) \ra \DM(\cY_0,\Lambda). \]
This isomorphism is defined as the following composition:
\begin{align*}
g_0^* \psi_f^{\un}(M):= \: & g_0^* i_{\fX}^* (j_{\fX})_*(M \otimes f_\eta^*\Log) \\
 \simeq \: & i_{\cY}^* g^* (j_{\fX})_*(M \otimes f_\eta^*\Log) \\
\simeq \: & i_{\cY}^* (j_{\cY})_*g_\eta^*(M \otimes f_\eta^*\Log) \\
 \simeq \: & i_{\cY}^*(j_{\cY})_*(g_\eta^*(M) \otimes (f \circ g)_\eta^* \Log)=: \psi_{f \circ g}^{\un} g_\eta^*(M),
\end{align*}
where the middle isomorphism follows from smooth base change (Theorem \ref{thm Adeel translated to DM} \emph{\ref{dm stack sbc}}).
\end{proof}

\begin{rmk}\label{rmk two defs psi}
It follows from this smooth base change property that the functors $\psi_{f}$ and $\phi_{f}$ defined above are canonically equivalent to the functors obtained by extending $\psi_{f}$ and $\phi_{f}$ for schemes using Theorem \ref{thm Adeel translated to DM} \emph{\ref{dm stack smooth}}.
\end{rmk}

\begin{rmk}
For a morphism of stacks $g : \cY \ra \fX$ and regular function $f : \fX \ra \AA^1_S$, it is sometimes convenient to work with $\widetilde{\phi}_{f}:= i_* \circ \phi_{f} : \DM(\fX,\Lambda)\to \DM(\fX,\Lambda)$. Then the above two results translate into the following statements. 
\begin{enumerate}
\item By Proposition \ref{prop PBC van cycles}, for $g$ proper and representable by Deligne--Mumford stacks, we have $\widetilde{\phi}_f g_* \simeq g_* \widetilde{\phi}_{f \circ g} : \DM(\cY,\Lambda) \ra \DM(\fX,\Lambda)$.
\item By Proposition \ref{prop SBC van cycles}, for $g$ smooth, we have $g^*\widetilde{\phi}_f \simeq \widetilde{\phi}_{f \circ g} g^*: \DM(\fX,\Lambda) \ra \DM(\cY,\Lambda)$.
\end{enumerate}
\end{rmk}

The theory of motivic vanishing cycles is strongly inspired by the theory of nearby and vanishing cycle functors for sheaves of $\Lambda$-modules on complex algebraic varieties \cite[Exposé XIV]{SGA7-II}. Let $X$ be a finite type $\CC$-scheme and $f:X\to \AA^1_\CC$ be a regular function. For concreteness, we use the formulation in \cite[Paragraph before Proposition 4.8]{Ayoub_Betti} to define
\[
\psi^\an_f:D(X^\an_\eta,\Lambda)\to D(X^\an_0,\Lambda)
\]
and we define $\phi^\an_f:D(X^\an,\Lambda)\to D(X^\an_0,\Lambda)$ as a cofibre of the natural map $i^*\to \psi^\an_f\circ j^*$. By \cite[Exposé XIV]{SGA7-II}, these functors satisfy smooth and proper base change properties. Moreover, the definition is in terms of the six operations and so makes sense at the level of $\infty$-categories of sheaves.

We claim that these nearby and vanishing cycles functors in the sheaf setting can be readily extended to Artin stacks. Let $\fX$ be a finite type Artin stack over $\CC$ and $f:\fX\to \AA^{1}_{\CC}$ be a regular function. As when defining the Betti realisation, we define $D(\fX^{\an},\Lambda)$ as a limit over $\mathrm{Lis}_\fX$. Moreover, $\phi^\an_f$ also satisfies smooth base change, so we can extend directly to a functor
\[
\phi^\an_{f}:D(\fX^{\an},\Lambda)\to D(\fX_0^{\an},\Lambda)
\]

Moreover, these motivic and sheaf vanishing cycles functors commute with Betti realisation in the following sense.

\begin{thm}\label{thm van cycles commute with Betti realisation}
  Let $\sigma:k\to \CC$ be a complex embedding. Let $\fX$ be a finite type Artin stack over $k$ and $f:\fX\to \AA^1_k$ be a regular function.
  \begin{enumerate}[label=\emph{(\roman*)}]
\item  There is a natural transformation
  \[
\omega_{f}:R_{B}\circ \phi_{f}\to \phi^\an_{f}\circ R_{B}
\]
of functors $\DM(\fX,\Lambda)\to \DM(\fX_{0},\Lambda)$.
\item The construction of $\omega_{f}$ commutes with base change by smooth morphisms (modulo smooth base change for vanishing cycles).
\item Let $M\in \DM(\fX,\Lambda)$ be such that for every finite type $k$-scheme $X$ and every smooth morphism $u:X\to \fX$, the motive $u^{*}M$ is constructible in $\DM(X,\Lambda)$. Then $\omega_{f}(M)$ is an isomorphism.
\end{enumerate}    
\end{thm}

\begin{rmk}
Condition (iii) above should be understood as the assumption that ``$M$ is a constructible motive on $\fX$'', but we do not want or need to introduce such a notion, as (iii) is easy to check in our cases of interest.
\end{rmk}

\begin{proof}
  Both sides of $\omega_f$ are obtained by passing to the limit over $\mathrm{Lis}_{\fX}$, so to show (i) and (ii) it suffices to construct $\omega_f$ for schemes and to proves that its construction commutes with base change by smooth morphisms.

  Such a natural transformation $\omega_{f}$ is constructed for nearby cycles in the context of triangulated categories of motives over schemes by Ayoub in \cite[Proposition 4.8]{Ayoub_Betti}. He also proves that his construction commutes with base change by smooth morphisms. Unfortunately, Ayoub uses in \textit{loc.\ cit.}\ a slightly different definition of nearby cycles involving diagram of schemes. We claim that with our logarithmic definition, the construction of $\omega_f$ is even simpler than in \cite[Proposition 4.8]{Ayoub_Betti}.

  Let us explain the construction of the analogue $\omega_f^\un:R_{B}\circ \psi^\un_{f}\to \psi_{f^{\an}}\circ R_{B}$; the general case then follows from passing to the limit over the $n$-th power maps and taking a cofiber. The key point is that by construction of $\Log$ (see \cite[Definition 3.6.29]{Ayoub_these_2}), the object $R_B\Log\in D(\GG_{m,\CC}^\an,\Lambda)$ is an ind-(unipotent local system) on $\CC^\times$ and it thus becomes canonically trivialised when pulling back to the universal cover $\mathrm{exp}:\CC\to\CC^\times$. Looking at the definition of $\psi^\an_{f}$ in \cite[Paragraph 4.8]{Ayoub_Betti}, we see that this provides a natural transformation $\omega_f^\un:R_{B}\circ \psi^\un_{f}\to \psi_{f^{\an}}\circ R_{B}$. The fact that the construction of the resulting natural transformation $\omega_f$ satisfies base change by smooth morphisms is then an easy consequence of smooth base change for $j_*$. This finishes the proof of (i) and (ii).

It remains to show (iii). By (ii) and the fact that the collection of functors $(u_X^*:\DM(\fX,\Lambda)\to \DM(X,\Lambda))_{(u:X\to \fX)\in\mathrm{Lis}_{\fX}}$ is jointly conservative, it suffices to show (iii) when $\fX=X$ is a scheme and $M$ is constructible. Since it is a matter of checking that a certain morphism is an isomorphism, we can work at the level of the (triangulated) homotopy categories. In \cite[Theoreme 4.9]{Ayoub_Betti}, Ayoub proves precisely this result for his definition of $\omega_f$ (and nearby cycles, but passing to vanishing cycles is then easy). We claim that the arguments in \cite[Theoreme 4.9]{Ayoub_Betti} apply just as well to our definition of $\omega_f$. The proof of \cite[Theoreme 4.9]{Ayoub_Betti} proceeds by reducing, using smooth and proper base change for nearby cycles (and the resulting machinery of ``specialisation systems''), to a very simple situation. The same reduction applies in our case, and the computation in the simple situation is then also easy to do.
\end{proof}  

\subsection{Motivic vanishing cycles for homogeneous functions} 

Our goal in this section is to prove the following theorem about motivic vanishing cycles functors for quadratic forms, which is a special case of a more general result about vanishing cycles of homogeneous functions (see Theorem \ref{thm mot van cycles for homog f}). Recall for a half integer $r \in \frac{1}{2} \ZZ$, we defined Tate twists $\{r \} := (\lfloor r \rfloor)[2r]$, which are pure if and only if $r \in \ZZ$.

\begin{thm}\label{thm mot van cycles for homog quad form}
Let $V$ be a vector space of dimension $d > 0$ over an algebraically closed field $k$ of characteristic zero and $q: V \ra \AA^1$ be a non-degenerate quadratic form. Then in $\DM(V_0,\Lambda)$, we have
\[ \phi_q (\one_V) \simeq (i_0)_*\one \{ - (d-1)/2 \},\]
where $i_0: \spec(k) \ra V_0:=q^{-1}(0)$ denotes the inclusion of the origin.
\end{thm}

This result is well-known in the étale setting as part of the Picard--Lefschetz theory in SGA7 \cite[Exposé XV 2.2.5 D E]{SGA7-II}. Our assumptions on $k$ are certainly not optimal but make the statement and proof somewhat simpler. For most of the proof we can work with a non-degenerate homogeneous regular function $f: V \ra \AA^1$ and we will describe the vanishing cycles functor as the reduced cohomological motive of the fibre $V_1:=f^{-1}(1)$ as stated in the next theorem. Versions of the following theorem were already well-known in other contexts and particularly in Donaldson--Thomas theory. More precisely, the virtual motivic nearby cycles functor of a homogeneous form (or more generally a certain torus equivariant regular function) is described in \cite[Theorem B.1]{BBS} and in a weighted homogeneous setting in \cite[Theorem 4.1.1]{NicaisePayne}, confirming a conjecture of Davison and Meinhardt \cite{BenSven}.

\begin{thm}\label{thm mot van cycles for homog f}
Let $V$ be a vector space over a field $k$ of characteristic zero and $f: V \ra \AA^1$ be a non-degenerate homogeneous function. Then in $\DM(V_0,\Lambda)$, we have
\[ \phi_f (\one_V) \simeq (i_0)_*\overline{M}_{\mathrm{coh}}(V_1),\]
where $i_0: \spec(k) \ra V_0:=f^{-1}(0)$ denotes the inclusion of the origin.
\end{thm}

Theorem \ref{thm mot van cycles for homog quad form} directly follows from Theorem \ref{thm mot van cycles for homog f} and Proposition \ref{prop motive affine quadric} below, which computes the (cohomological) motive of a (split) affine quadric and is based on Rost's computation in the projective setting \cite{Rost} and Bachmann's computation in the affine setting \cite{Bachmann}.

\begin{prop}\label{prop motive affine quadric}
Let $Q_n$ be a smooth affine quadric of dimension $n$ over an algebraically closed field $k$; then in $DM(k,\Lambda)$ we have an isomorphism
\[ M_\mathrm{coh}(Q_n) \simeq \one \oplus \one \{ -n/2 \}. \]
\end{prop}
\begin{proof}
We prove the proposition by induction on $n$. For $n = 0$, the quadric $Q_0$ consists of two points and thus $M_{\mathrm{coh}}(Q_0) \simeq \one \oplus \one$. For $n = 1$, the quadric $Q_1$ is isomorphic to $\GG_m$ and thus $M_\mathrm{coh}(Q_1) \simeq \one \oplus \one (- 1 ) [-1]$. 

For the inductive step, we relate the reduced motives of $Q_n$ and $Q_{n-2}$ by using \cite[Lemma 34]{Bachmann} and \cite[Proposition 1]{Rost} to conclude
\[ \overline{M}(Q_n) \simeq \overline{M}(Q_{n-2}) \{ 1 \}.\] 
For example, this isomorphism appears in \cite[Eq (1)]{BachmannVishik}. By dualising, we conclude $\overline{M}_{\mathrm{coh}}(Q_n) \simeq \overline{M}_{\mathrm{coh}}(Q_{n-2}) \{ -1 \}$ which completes the inductive proof.
\end{proof}

The goal for the rest of this section is to prove Theorem \ref{thm mot van cycles for homog f}. Throughout this section, we write $f_0 : V_0 \ra \{ 0 \}$ and $f_\eta : V_\eta \ra \GG_m$ as in the previous subsection. 

First of all, as the critical locus of the non-degenerate homogeneous function $f : V \ra \AA^1$ is the origin, the vanishing cycles functor is concentrated at the origin in the following sense.

\begin{lemma}\label{lemma van cycle concentrated at zero}
Let $X$ be a finite type $k$ scheme and $f: X \ra \AA^1$ be a regular function with a single isolated critical point $x_{0}\in X(k)$ over $0$ with $i_{0}:\Spec(k)\to X_{0}$ the corresponding closed immersion. We have 
\[ \phi_f (\one_X) \simeq (i_0)_*(i_0)^*\phi_f (\one_X) \simeq (i_0)_*(f_0)_* \phi_f (\one_X). \]
\end{lemma}
\begin{proof}
Since the restriction $f^\times : X^\times \ra \AA^1$ of $f$ to $X^\times:= X \setminus \{ x_{0} \}$ is smooth, we have $\psi_{f^\times}(\one_{X_\eta}) \simeq \one_{X_0}$ and $\phi_{f^\times} (\one_X) \simeq 0$. Let $ X_0^{\times} := X_0 \setminus \{ x_{0} \}$ and $j_0: X_0^{\times} \hookrightarrow X_0$ denote the open immersion; then by smooth base change for vanishing cycles, we have $\phi_{f^\times}(\one_{X^\times}) \simeq j_0^* \phi_f (\one_X)$. By considering the localisation triangle for the pair $(j_0 : X_0^\times \hookrightarrow X_0, \: i_0: \{ x_{0} \} \hookrightarrow X_0)$ 
\[  (j_0)_!(j_0)^*\phi_f (\one_X) \ra \phi_f (\one_X) \ra (i_0)_*(i_0)^*\phi_f (\one_X)  \rap \]
we obtain the first claimed isomorphism, as the left term is zero. The second claimed isomorphism then follows as $i_0$ is a section of $f_0$ and again using the localisation sequence for $(i_0,j_0)$ together with the fact that $(j_0)^*\phi_f (\one_X) \simeq 0$.
\end{proof}

The next step is to relate the vanishing cycles functor for $f$ with the vanishing cycles functor for $\mathrm{Id} : \AA^1 \ra \AA^1$.

\begin{prop}\label{prop van cycles f and Id}
For a homogeneous non-degenerate function $f: V \ra \AA^1$ on a vector space $V$ over a field $k$, we have $\phi_{\mathrm{Id}}(f_* \one_V) \simeq (f_0)_* \phi_f(\one_V)$.
\end{prop}
\begin{proof} 
Ideally we would like to apply proper base change, but the morphism $f$ is not proper. To remedy this, we construct a fibrewise projectivisation $\widetilde{f} : \widetilde{V} \ra \AA^1$ with smooth boundary divisor $D:= \widetilde{V} \setminus V$ which is a constant family over $\AA^1$ and use proper base change for $\widetilde{f}$. More precisely, we define $\widetilde{V}$ to be the relative projective spectrum over $\AA^1 = \Spec k[t]$ of the graded ring $k[V,z,t]/(f - tz^r)$ where $r$ denotes the degree of the homogeneous function $f$ and the coordinate $t$ has weight 0, the coordinate $z$ has weight $1$ and the coordinates on $V$ have weight $1$. By construction, the fibre $\widetilde{V}_\lambda$ over $\lambda \in \AA^1$ is the projective variety defined by the vanishing locus of the equation $f(v) = \lambda z^r$ in $\PP(V \oplus k) $ with coordinates $[v:z]$ and it contains $V_\lambda$ as the intersection with $V \cong \{ z \neq 0 \}$ and has boundary $D_\lambda$ given by the projective variety $f(v) = 0$ in $\PP(V) \cong \{ z = 0 \}$. 

Let $f_D : D \ra \AA^1$ denote the restriction of $\widetilde{f}$ to $D$, so that we have a commutative diagram
\[ \xymatrix{ V \ar[r]^{\nu} \ar[dr]_{f} & \widetilde{V} \ar[d]^{\widetilde{f}} & D \ar[ld]^{f_D} \ar[l]_{\iota} \\ & \AA^1} \]
with $\nu$ an open immersion and $\iota$ a closed immersion. Since $D \cong D_0 \times \AA^1$ is a smooth constant family, we have 
\begin{equation}\label{eq van cycle for boundary is zero}
\phi_{f_D}(\one_D) \simeq 0.
\end{equation} 
Furthermore, as $\widetilde{f}$ and $f_D$ are both proper, by Proposition \ref{prop PBC van cycles}, we have
\begin{equation}\label{eqns using PBC van cycles} 
\phi_{\mathrm{Id}}\widetilde{f}_*  \simeq (\widetilde{f}_0)_* \phi_{\widetilde{f}} \quad \text{and} \quad \phi_{\mathrm{Id}}f_{D_*} \simeq (f_{D,0})_* \phi_{f_D}.
\end{equation}
By applying $\widetilde{f}_*$ to the localisation sequence for $(\nu : V \hookrightarrow \widetilde{V}, \iota: D \hookrightarrow \widetilde{V})$, we obtain
\[ \widetilde{f}_*\iota_*\iota^! \one_{\widetilde{V}} \ra  \widetilde{f}_*\one_{\widetilde{V}} \ra \widetilde{f}_*\nu_*\nu^*\one_{\widetilde{V}}  \rap \]
and we have isomorphisms $\widetilde{f}_*\nu_*\nu^*\one_{\widetilde{V}} \simeq f_*\one_V$ and $\widetilde{f}_*\iota_*\iota^! \one_{\widetilde{V}}  \simeq (f_D)_* \iota^! \one_{\widetilde{V}} \simeq(f_D)_* \one_D\{ -1 \}$ using the purity isomorphism for the regular map $\iota$. Next we apply the vanishing cycles functor for $\mathrm{Id} : \AA^1 \ra \AA^1$ to this sequence to obtain
\[ \phi_{\mathrm{Id}} ((f_D)_*\one_D\{ -1 \}) \ra \phi_{\mathrm{Id}}(\widetilde{f}_*\one_{\widetilde{V}}) \ra \phi_{\mathrm{Id}} (f_*\one_V) \rap \]
where the left term is zero and the middle term is isomorphic to $(\widetilde{f}_0)_*\phi_{\widetilde{f}}(\one_{\widetilde{V}})$ by Equations \eqref{eq van cycle for boundary is zero} and \eqref{eqns using PBC van cycles}. Therefore, to prove that $\phi_{\mathrm{Id}}(f_* \one_V) \simeq (f_0)_* \phi_f(\one_V)$, it suffices to show that there is an isomorphism $(\widetilde{f}_0)_*\phi_{\widetilde{f}}(\one_{\widetilde{V}}) \ra  (f_0)_* \phi_f(\one_V)$. This final claim is proved by using Lemma \ref{lemma van cycle concentrated at zero} for both $f$ and $\widetilde{f}$ together with the isomorphism $\widetilde{i_0}^* \phi_{\widetilde{f}} \simeq i_0^* \phi_f$, where $i_0$ (resp. $\widetilde{i_0}$) denote the inclusion of the origin in $V_0$ (resp. $\widetilde{V}_0$), which follows from smooth base change for vanishing cycles applied to the open immersion $\nu$.
\end{proof}

We are now ready to prove Theorem \ref{thm mot van cycles for homog f}.

\begin{proof}[Proof of Theorem \ref{thm mot van cycles for homog f}]
Let $r$ denote the degree of the homogeneous function $f : V \ra \AA^1$. We let $p_r : \AA^1 \ra \AA^1$ denote the $r$th power map and write $f_{(r)} : V_{(r)} \ra \AA^1$ for the base change of $f$ along $p_r$. Then, because $f$ is homogeneous of degree $r$, the generic fibre $f_{(r),\eta} : V_{(r),\eta} \ra \GG_m$ is the constant family: $V_{(r),\eta} \cong V_{(r),1} \times \GG_m \cong V_1 \times \GG_m$. 

By \cite[Lemme 3.5.8 and following paragraph]{Ayoub_these_2}, the nearby cycles functor for $\mathrm{Id} : \AA^1 \ra \AA^1$ is unchanged when applying $p_{r,\eta}^*$; that is, $\psi_{\mathrm{Id}} \simeq \psi_{\mathrm{id}} p_{r,\eta}^*$. Therefore, we have isomorphisms 
\[\psi_{\mathrm{Id}} ((f_\eta)_*\one_{V_\eta}) \simeq \psi_{\mathrm{Id}} (p_{r,\eta}^*(f_\eta)_*\one_{V_\eta}) \simeq \psi_{\mathrm{Id}} (f_{(r),\eta})_*\one_{V_{(r),\eta} }  \simeq (f_1)_*\one_{V_1}\]
where the first isomorphism is \cite[Lemme 3.5.8 and following paragraph]{Ayoub_these_2}, the middle isomorphism is smooth base change for $p_r$ and the right isomorphism follows as $f_{(r),\eta}$ is constant. Here $f_1 : V_1 \ra \{ 1 \}$ is the restriction of $f$ to the fibre over $1$.

Since $V_0$ is an affine cone over $0$, we have $(f_0)_* \one_{V_0} \simeq \one_k$ by rescaling. Even though $f$ is not proper, we have $i^*f_*\one_V \simeq (f_0)_* \one_{V_0}$ as follows from localisation and proper base change for the compactification $\tilde{f}$ as in the proof of Proposition \ref{prop van cycles f and Id}.  Therefore, we have 
\begin{equation}\label{eq relate to red coh motive of general fibre}
\phi_{\mathrm{Id}}(f_*\one_V) = \mathrm{cofib} \left(i^*f_* \one_V \ra \phi_{\mathrm{Id}} ((f_\eta)_*\one_{V_\eta})\right) \simeq \mathrm{cofib} \left( \one_k \ra (f_1)_*\one_{V_1} \right) \simeq \overline{M}_{\mathrm{coh}}(V_1).
\end{equation}

To complete the proof, we have isomorphisms
\[ \phi_f (\one_V) \simeq  (i_0)_*(f_0)_* \phi_f (\one_V) \simeq   (i_0)_*\phi_{\mathrm{Id}}(f_* \one_V) \simeq  (i_0)_* \overline{M}_{\mathrm{coh}}(V_1) \]
coming from Lemma \ref{lemma van cycle concentrated at zero}, Proposition \ref{prop van cycles f and Id} and Equation \eqref{eq relate to red coh motive of general fibre}.
\end{proof}

\begin{rmk}
We expect that Theorem \ref{thm mot van cycles for homog f} can be generalised to the setting of a non-degenerate weighted homogeneous function $f : V \ra \AA^1$.
\end{rmk}

\section{Motives of stacks of vector bundles with fixed determinant}\label{sec appendix motive stacks of bundles fixed det}

Throughout this section, we assume that $k$ is an arbitrary field and $C(k) \neq \emptyset$. We will compute motives of stacks of vector bundles over $C$ (or families $\cC/T$ of curves) with fixed determinant in $\DM(k,\QQ)$ by extending results in \cite{HPL,HPL_formula}.

\subsection{A formula for the motive of the stack of bundles with fixed determinant}

Fix a line bundle $L \ra C$ and consider the stack $\Bun_{n,L}$ of rank $n$ vector bundles with determinant isomorphic to $L$. If $d = \deg(L)$, then $\Bun_{n,L}$ is a smooth codimension $g$ substack of $\Bun_{n,d}$. In this section, we will prove the following explicit formula.

\begin{thm}\label{thm motive bun fixed det}
Assume $C(k) \neq \emptyset$. Then in $\DM(k,\QQ)$, we have
\[ M(\Bun_{n,L}) \simeq M(B\GG_m) \otimes  \bigotimes_{i=1}^{n-1} Z(C, \QQ\{i\}).\]
\end{thm}
\begin{proof} 
Fix $x \in C(k)$ and for $l \in \NN$, we consider as in \cite[$\S$4.3]{HPL}, the scheme
\[ \Div_{n,L}(l):= \{ E \subset \cO_C(lx)^{\oplus n} : \rk(E) = n, \det(E) \cong L \} \]
which is a smooth codimension $g$ closed subvariety of the Quot scheme $\Div_{n,d}(l)$ of length $nl-d$ torsion quotients of $ \cO_C(lx)^{\oplus n}$, where $d = \deg(L)$. Moreover, one has
\[ M(\Bun_{n,L}) \simeq \hocolim_{l} M(\Div_{n,L}(l))\]
as in the proof of the first paragraph of \cite[Theorem 4.6]{HPL}.

We also consider the smooth closed subscheme $\FDiv_{n,L}(l)$ in the full-flag Quot scheme $\FDiv_{n,d}(l)$ given by
\[  \FDiv_{n,L}(l):=\left\{ F_0 \subset \cdots \subset F_{nl-d} = \cO_{C }(lx)^{\oplus n} :  \rk(F_i) = n, \deg(F_i) = d + i, \det(F_0) \simeq L  \right\}. \]
The support map $\supp : \FDiv_{n,d} \ra C^{nl-d}$ given by $(F_0 \subset \cdots \subset F_{nl-d}) \mapsto \supp(F_i/F_{i-1})_{1 \leq i \leq nl-d}$ is an $(nl-d)$-iterated projective bundle. Let $(C^{nl-d})_L:= \supp (\FDiv_{n,L}(l))$; then by the projective bundle formula
\[ M(\FDiv_{n,L}(l)) \simeq M((C^{nl-d})_L) \otimes M(\PP^{n-1})^{\otimes (nl-d)}. \]

We will now complete the proof by adapting the argument in \cite[$\S$4.3]{HPL_formula}.
Using the decompositions $M(\PP^{n-1}) \simeq \oplus_{i=0}^{n-1} \QQ\{i \}$ as in \cite[Remark 4.6]{HPL_formula}, we can write
\[ M(\FDiv_{n,L}(l)) \simeq \bigoplus_{I \in \cI_l} M((C^{nl-d})_L) \otimes \QQ\{ | I | \},\]
where $\cI_l = \{ 0, \dots, n-1\}^{\times nl-d}$ and for $I = (i_1, \dots, i_{nl-d}) \in \cI_l$, we write $|I|:= \sum_{j=1}^{nl-d} i_j$. Let $\cB_l$ denote the set of $\underline{m}=(m_0,\dots, m_{n-1}) \in \NN^{n}$ with $\sum_{i=0}^{n-1} m_i = nl-d$. Then as in \cite[Lemma 4.7]{HPL_formula}, we can conclude that
\[  M(\Div_{n,L}(l)) \simeq \bigoplus_{\underline{m} \in \cB_l} M(C_L^{(\underline{m})}) \{ c_{\underline{m}} \},\]
where $c_{\underline{m}}:=\sum_{i=0}^{n-1} i m_i$ and $C_L^{(\underline{m})} \subset C^{(\underline{m})}:= C^{(m_0)} \times \cdots \times C^{(m_{n-1})}$ is the image of $(C^{nl-d})_L$ under the quotient $C^{nl-d} \ra C^{(m_0)} \times \cdots \times C^{(m_{n-1})}$ by the product of symmetric groups $\prod_{i=0}^{n-1} S_{m_i}$. Furthermore, as in \cite[Lemma 4.7 (ii)]{HPL_formula}, the transition map $M(\Div_{n,L}(l)) \ra M(\Div_{n,L}(l+1))$ is induced by direct sums over $\underline{m} \in \cB_l$ and $\underline{m'} \in \cB_{l+1}$ of the maps $M(\kappa_{\underline{m},\underline{m'},L})\{c_{\underline{m}} \}$, where  
\[ \kappa_{\underline{m},\underline{m'},L} : C_L^{(\underline{m})} \ra C_L^{(\underline{m'})}\] is zero unless $\underline{m'} = \underline{m} + (n,0,\dots, 0)$ (and thus $c_{\underline{m}} = c_{\underline{m'}}$) and in this case, $\kappa_{\underline{m},\underline{m'},L}$ is the restriction of the map $\kappa_{\underline{m},\underline{m'}} : C^{(\underline{m})} \ra C^{(\underline{m'})}$ induced by $(x,\dots,x) \times \mathrm{id}_{C^{nl-d}} \colon C^{nl-d} \ra C^{n(l+1) -d}$ which includes $n$ copies of $x$. 

The rest of the proof follows exactly as in \cite[Theorem 4.6]{HPL}, so we simply outline the idea. For $\underline{m}^\flat = (m_1, \dots, m_{n-1}) \in \NN^{n-1}$, we define $c_{\underline{m}^\flat} :=\sum_{i=1}^{n-1} i m_i$ and $m^\flat_0(l):=nl -d - \sum_{i=1}^{n-1} m_i$ and write $\underline{m}^\flat(l) := (m^\flat_0(l), \underline{m}^\flat) \in \ZZ \times \NN^{n-1}$; then $c_{\underline{m}^\flat(l)} =c_{\underline{m}^\flat}$. From the above description of the transitions maps one obtains that
\[ M(\Bun_{n,L}) \simeq \bigoplus_{\underline{m}^\flat \in \NN^{n-1}} \hocolim_{l} P^L_{\underline{m}^\flat,l} \]
where $P^L_{\underline{m}^\flat,l} := M(C_L^{(\underline{m}^\flat(l))}) \{ c_{\underline{m}^\flat}\}$ if $m^\flat_0(l) \geq 0$ and is zero otherwise. Next for $\underline{m} = (m_0, \dots , m_{n-1})$, we use a generalised Abel-Jacobi map
\[ C_L^{(\underline{m})} \ra C^{(m_1)} \times \cdots \times C^{(m_{n-1})},\]
which is a $\PP^{m_0 -g}$-bundle if $m_0 > 2g -2$, to deduce that
\begin{align*}
M(\Bun_{n,L}) &  \simeq  \bigoplus_{\underline{m}^\flat \in \NN^{n-1}} \hocolim_{l : m^\flat_0(l) \geq 0} P^L_{\underline{m}^\flat,l} \simeq \bigoplus_{\underline{m}^\flat \in \NN^{n-1}} \hocolim_{l : m^\flat_0(l) > 2g-2} M(\PP^{m^\flat_0(l) -g }) \otimes M(C^{(\underline{m}^\flat)})\{c_{\underline{m}^\flat }\} \\
& \simeq \bigoplus_{\underline{m}^\flat \in \NN^{n-1}} M(B\GG_m) \otimes M(C^{(\underline{m}^\flat)})\{c_{\underline{m}^\flat }\} \simeq M(B\GG_m) \otimes  \bigotimes_{i=1}^{n-1} Z(C, \QQ\{i\}),
\end{align*}
where $C^{(\underline{m}^\flat)}:= C^{(m_1)} \times \cdots \times C^{(m_{n-1})}$. This completes the proof.
\end{proof}

In the case when $L = \cO_C$, we have that $\Bun_{SL_n} \ra \Bun_{n,\cO_C}$ is a $\GG_m$-torsor (see \cite[$\S$4.3]{HPL}) and so we deduce the following corollary.

\begin{cor}
Assume $C(k) \neq \emptyset$. Then in $\DM(k,\QQ)$, we have
\[ M(\Bun_{\SL_n}) \simeq \bigotimes_{i=1}^{n-1} Z(C, \QQ\{i\}).\]
\end{cor}
\begin{proof} 
This follows from Theorem \ref{thm motive bun fixed det} using the arguments of the proof of \cite[Theorem 4.7]{HPL}.
\end{proof}

\subsection{Relative formulae for families of curves}

Throughout this section, we fix a (Noetherian finite-dimensional) scheme $T$ and consider a family $\cC$ of smooth projective geometrically connected genus $g$ curves over $T$ and we assume that this family admits a section. We let $\Bun_{\cC/T,n,d}$ denote the stack (over $T$) of vector bundles on $\cC/T$ of rank $n$ and degree $d$. We first consider the relative case without fixing the determinant and then consider the relative case with fixed determinant.

For a $T$-scheme $X$ (or stack), we write $M_T(X) \in \DM(T,\QQ)$ for the relative motive of $X$ over $T$. We write $\QQ_T\{r\} \in \DM(T,\QQ)$ for the pure Tate twists. We shall also write $(X/T)^r$ to denote the $T$-scheme given by the $r$-fold fibre product of $X$ over $T$ and we write $\Sym^r(X/T)$ for the $S_r$-quotient of $(X/T)^r$.

\begin{thm}\label{thm rel bun formula}
Let $\cC/T$ be a family of smooth projective geometrically connected genus $g$ curves over $T$ admitting a section $\sigma: T \ra \cC$. Then in $\DM(T,\QQ)$, we have
\[ M_T(\Bun_{\cC/T,n,d}) \simeq  M_T(\Jac_{\cC/T}) \otimes M_T(B\GG_{m,T}) \otimes  \bigotimes_{i=1}^{n-1} Z_T(\cC/T, \QQ_T\{i\}),\]
where 
\[ Z_T(\cC/T, \QQ_T\{i\}):= \bigoplus_{j \geq 0} M_T(\Sym^j(\cC/T)) \otimes  \QQ_T\{i j\} \in \DM(T,\QQ).\]
\end{thm}
\begin{proof}
Let $\cO_{\cC}(\sigma)$ denote the line bundle whose restriction to $t \in T$ is the degree $1$ line bundle $\cO_{\cC_t}(\sigma(t))$. We now consider relative versions of the (flag)-Quot schemes appearing in \cite{HPL,HPL_formula}: we let
\[ \Div_{\cC/T,n,d}(l):= \mathrm{Quot}_{\cC/T}^{0,nl-d}(\cO_{\cC}(l\sigma)^{\oplus n})\]
denote the relative Quot scheme over $T$ of rank $0$, degree $nl-d$ quotient sheaves of $\cO_{\cC}(l\sigma)^{\oplus n}$ and we let $\FDiv_{\cC/T,n,d}(l)$ denote the relative full flag version, whose points over $S \ra T$ are
\[  \FDiv_{\cC/T,n,d}(l)(S) = \left\{ \cF_0 \subset \cdots \subset \cF_{nl-d} = \cO_{\cC_S}(l\sigma_S)^{\oplus n} : \begin{smallmatrix} \cF_i \\ \downarrow \\ \cC_S \end{smallmatrix}, \rk(\cF_i) = n, \deg(\cF_i) = d + i  \right\}. \]
Both $\Div_{\cC/T,n,d}(l)$ and $\FDiv_{\cC/T,n,d}(l)$ are smooth projective schemes over $T$. 

The natural forgetful morphisms $\Div_{\cC/T,n,d}(l) \ra \Bun_{\cC/T,n,d}$ induce a morphism 
\[  \hocolim_l M_T(\Div_{\cC/T,n,d}(l)) \ra M_T(\Bun_{\cC/T,n,d}) \]
in $\DM(T,\QQ)$, which we claim is an isomorphism. By \cite[Proposition 3.24]{Ayoub_etale}, it suffices to check the pullback of this map along each point $t \in T$ is an isomorphism. However, for each $t \in T$, the pullback to $\DM(\kappa(t),\QQ)$ 
\[ \hocolim_l M(\Div_{\cC_t,n,d}(l\sigma(t))) \ra M(\Bun_{\cC_t,n,d}) \]
coincides with the isomorphism given by \cite[Theorem 3.5]{HPL}.

Extending the absolute case proved in \cite[$\S$4]{HPL_formula} to the relative setting, there is a support map $\supp : \FDiv_{\cC/T,n,d}(l) \ra (\cC/T)^{nl-d}$ sending $
\cF_0 \subset \cdots \subset \cF_{nl-d}$ to $\supp (\cF_i/\cF_{i-1})$. Furthermore, as in \cite[$\S$4]{HPL_formula}, the support map is an $(nl-d)$-iterated $\PP^{n-1}$-bundle and thus
\[ M_T(\FDiv_{\cC/T,n,d}(l)) \simeq M_T( \PP^{n-1}_{\cC})^{\otimes nl-d} \in \DM(T,\QQ). \]
We claim that the following composition 
\[ M_T(\Sym^{nl-d}(\PP^n_{\cC}/T)) \hookrightarrow M_T( \PP^{n-1}_{\cC})^{\otimes nl-d}  \simeq M_T(\FDiv_{\cC/T,n,d}(l)) \ra M_T(\Div_{\cC/T,n,d}(l))\]
is an isomorphism in $\DM(T,\QQ)$, where the last morphism is induced by the natural forgetful map. Again, by \cite[Proposition 3.24]{Ayoub_etale}, it suffices to check the pullback of this map along each point $t \in T$ is an isomorphism, which is proved in \cite[Theorem 1.3]{HPL_formula}. Upgrading \cite[Lemma 4.4]{HPL_formula} to the relative case, we see that 
the transition maps $M_T(\Div_{\cC/T,n,d}(l)) \ra M_T(\Div_{\cC/T,n,d}(l+1))$  correspond to the inclusions
\[  \bigoplus_{i=0}^{nl-d} \Sym^i(M_{\cC/T,n}) \hookrightarrow \bigoplus_{i=0}^{n(l+1)-d}  \Sym^i(M_{\cC/T,n}),\]
where $M_{\cC/T,n}:=\overline{M}_T(\PP^{n-1}_{\cC}/T)$ is defined by the decomposition $M_T(\PP^{n-1}_{\cC}) \simeq \QQ_T\{0\} \oplus M_{\cC/T,n}$ given by $\sigma$. Hence,  similarly to \cite[Theorem 4.5]{HPL_formula}, we conclude that
\begin{align*}
M_T(\Bun_{\cC /T,n,d}) &  \simeq \hocolim_l   \bigoplus_{i=0}^{nl-d} \Sym^i(M_{\cC/T,n}) \simeq \bigoplus_{i=0}^{\infty} \Sym^i(M_{C,n}) \\
& \simeq  M_T(\Jac_{\cC/T}) \otimes M_T(B\GG_{m,T}) \otimes  \bigotimes_{i=1}^{n-1} Z_T(\cC/T, \QQ_T\{i\})
\end{align*}
where we use the decompositions $M_T(\cC) \simeq \QQ_T\{0\} \oplus M_T(\Jac_{\cC/T}) \oplus \QQ_T\{1\}$ (see \cite[Corollary 3.20 (iii)]{Simon_thesis}) and $M_T(\PP^{n-1}_T) \simeq \oplus_{i=0}^{n-1} \QQ_T\{i\}$.
\end{proof}

Now fix a line bundle $\cL \in \Pic^d_{\cC/T}(T)$ and consider the stack $\Bun_{\cC/T,n,\cL}$ of vector bundles on $\cC/T$ of rank $n$ with determinant $\cL$; this is a smooth closed substack of $\Bun_{\cC/T,n,d}$.

\begin{thm}\label{thm motive rel Bun fixed det}
Let $\cC/T$ be a family of smooth projective geometrically connected genus $g$ curves over $T$ admitting a section $\sigma: T \ra \cC$ and $\cL \in \Pic^d_{\cC/T}(T)$. In $\DM(T,\QQ)$, we have
\[ M_T(\Bun_{\cC/T,n,\cL}) \simeq M_T(B\GG_{m,T}) \otimes  \bigotimes_{i=1}^{n-1} Z_T(\cC/T, \QQ_T\{i\}).\]
\end{thm}
\begin{proof}
One defines subschemes $\Div_{\cC/T,n,\cL}(l) \hookrightarrow \Div_{\cC/T,n,d}(l)$ where the subsheaf $\cE \subset \cO_{\cC}(l\sigma)^{\oplus n}$ has determinant $\cL$; this is a closed subscheme and smooth over $T$. One obtains an isomorphism
\[ M_T(\Bun_{\cC/T,n,\cL}) \simeq \hocolim_l M_T(\Div_{\cC/T,n,\cL}(l)) \]
by adapting the proof of Theorem \ref{thm rel bun formula} using Theorem \ref{thm motive bun fixed det}. 

We define a closed subscheme $\FDiv_{\cC/T,n,\cL}(l) \hookrightarrow \Div_{\cC/T,n,d}(l)$ with fixed determinant $\cL$ and consider the composition
\[ \bigoplus_{\underline{m} \in \cB_l} M_T((\cC/T)^{(\underline{m})_{\cL}}\{c_{\underline{m}}\} \hookrightarrow M_T((\PP^{n-1})^{nl-d} \times ({\cC}/T)^{nl-d}_{\cL}) \ra M_T(\FDiv_{\cC/T,n,\cL}(l)) \]
with the forgetful map $M_T(\FDiv_{\cC/T,n,\cL}(l)) \ra M_T(\Div_{\cC/T,n,\cL}(l))$; here $({\cC}/T)^{ nl-d}_{\cL}$ denotes the image of $\FDiv_{\cC/T,n,\cL}(l)$ under $\supp : \Div_{\cC/T,n,\cL}(l) \ra ({\cC}/T)^{nl-d}$ and for a partition $\underline{m} \in \cB_{l}$, we let $(\cC/T)^{(\underline{m})}_{\cL}$ denote the image of $({\cC}/T)^{ nl-d}_{\cL}$ under the quotient map  $({\cC}/T)^{ nl-d} \ra ({\cC}/T)^{(\underline{m})} := \prod_{i=0}^{n-1} \Sym^{m_i}(\cC/T)$. We claim that this composition is an isomorphism. Again, by \cite[Proposition 3.24]{Ayoub_etale}, it suffices to check the pullback of this map along each point $t \in T$ is an isomorphism, so we can reduce to $T=\Spec(k)$. But we proved that this morphism is an isomorphism in that case in the proof Theorem \ref{thm motive bun fixed det}. Furthermore, the transition maps are as in Theorem \ref{thm motive bun fixed det} and then the rest of the proof follows in exactly the same way.
\end{proof}

\bibliographystyle{abbrv}
\bibliography{references}

\medskip \medskip
\end{document}